\documentclass[12pt,a4paper]{amsart}
\usepackage[english]{babel}
\usepackage{color}
\usepackage[T1]{fontenc}
\usepackage{amsmath}
\usepackage{amscd}
\usepackage{amstext}
\usepackage{amsbsy}
\usepackage{amsopn}
\usepackage{amsthm}
\usepackage{amsxtra}
\usepackage{upref}
\usepackage{amsfonts}
\usepackage{amssymb}
\usepackage{mathrsfs}
\usepackage{euscript}
\usepackage{array}
\usepackage{stmaryrd}
\usepackage{verbatim}
\usepackage{setspace}
\usepackage{geometry}
\newtheorem*{mytheorem}{Theorem}
\theoremstyle{plain}
      \newtheorem{theorem}{Theorem}
      \newtheorem{lemma}[theorem]{Lemma}
      \newtheorem{corollary}[theorem]{Corollary}
      \newtheorem{proposition}[theorem]{Proposition}
     
      \theoremstyle{definition}
      \newtheorem{definition}[theorem]{Definition}
     
     \theoremstyle{remark}
     \newtheorem{remark}[theorem]{Remark}
     \theoremstyle{Fact}
     
\theoremstyle{notation}
      \newtheorem{notation}[theorem]{Notation}

\newenvironment{proof2}[1][Proof]{\noindent\textit{#1 :} }{\ 
}

\newenvironment{lemmas}[1][Lemma]{\begin{lemma} }{\ \hfill \rule{0.5em}{0.5em}\end{lemma}}
\newenvironment{corollaries}[1][Lemma]{\begin{corollary} }{\ \hfill \rule{0.5em}{0.5em}\end{corollary}}
\newenvironment{propositions}[1][Lemma]{\begin{proposition} }{\ \hfill \rule{0.5em}{0.5em}\end{proposition}}

\newcounter{Step}
\newenvironment{step}[0]{\bigskip\addtocounter{Step}{1}\noindent\textbf{Step \theStep :} }{\
  \begin{flushright} \end{flushright}}
\newenvironment{claim}[0]{\bigskip\noindent\textbf{Claim :} }{\
  \begin{flushright} \end{flushright}}

\newcommand{\red}{
}\newcommand{\re}{}
\newcommand{\gre}{}
\def\N{\mbox{I\hspace{-.15em}N} }
\def\R{\mbox{I\hspace{-.15em}R} }

\def\C{\hspace{.17em}\mbox{l\hspace{-.47em}C} }

\def\F{\mbox{I\hspace{-.15em}F} }
\def\Z{\mbox{Z\hspace{-.30em}Z} }

\def\mt{m\circ1\otimes\tau\otimes1}
\def\tm{(\tau\circ m)}
\def\d{\bar{\delta _{i}}}
\def\o{\otimes}

\def\D{\mathscr{D}}
\def\DD{\mathcal{D}}
\def\H{\mathscr{H}}

\def\n{\noindent}

\def\E{\mathscr{E}}
\def\A{\mathcal{A}}
\def\H{\mathcal{H}}

\def\J{\mathcal{J}}
\def\II{\mathcal{I}}
\def\B{\mathcal{B}}

\begin{document}

\title{A Free Stochastic Partial Differential Equation}
\begin{abstract}
We get stationary solutions of a free stochastic partial differential equation. As an application, we prove equality of non-microstate and microstate free entropy dimensions under a Lipschitz like condition on conjugate variables, assuming also the von Neumann algebra $R^\omega$ embeddable. This includes an N-tuple of q-Gaussian random variables e.g. for $|q|N\leq 0.13$. 
\\
\medskip
\\
\end{abstract}


\author[Y. Dabrowski]{Yoann Dabrowski}\address{ 
Universit\'e de Lyon\\
Universit\'e Lyon 1\\
Institut Camille Jordan\\
43 blvd. du 11 novembre 1918\\
F-69622 Villeurbanne cedex\\
France}
\email{dabrowski@math.univ-lyon1.fr}
\thanks{Research partially supported by NSF grants DMS-0555680 and DMS-0900776.}

\date{}
\maketitle


\begin{center}
\section*{Introduction}
\end{center}
In a fundamental series of papers, Voiculescu introduced analogs of entropy and Fisher information in the context of free probability theory. A first microstate free entropy $\chi(X_{1},...,X_{n})$ is defined as a normalized limit of the volume of sets of microstate i.e. matricial approximations (in moments) of the n-tuple of self-adjoints $X_{i}$ living in a (tracial) $W^{*}$-probability space $M$. Starting from a definition of a free Fisher information \cite{Vo5}, Voiculescu also defined a non-microstate free entropy $\chi^{*}(X_{1},...,X_{n})$, known by the fundamental work \cite{BGC} to be greater than the previous microstate entropy, and believed to be equal (at least modulo Connes' embedding conjecture). For more details, we refer the reader to the survey \cite{VoS} for a list of properties as well as applications of free entropies in the theory of von Neumann algebras.

 Moreover in case of infinite entropy, two other invariants the microstate and non-microstate free entropy dimensions (respectively written $\delta_0(X_{1},...,X_{n})$ and $\delta^*(X_{1},...,X_{n})$) have been introduced to generalize results known for finite entropy. Surprisingly, Connes and Shlyakhtenko found in \cite{CS} a relation between those entropy dimensions and the first $L^{2}$-Betti numbers they defined for finite von Neumann algebras. For instance, for (real and imaginary parts of ) generators of finitely generated groups, $\delta^*$ has been proved in \cite{MS05} to be equal to $\beta^{(2)}_1(\Gamma)-\beta^{(2)}_0(\Gamma)+1$ (cf. e.g. \cite{L2} for $L^{2}$-Betti numbers of groups).

In \cite{S07}, Dimitri Shlyakhtenko obtained lower bounds on microstate free entropy dimension (motivated by the goal of trying to prove equality with non-microstate free entropy dimension), in studying the following free stochastic differential equation :
$$X_{t}^{(i)}=X_{0}^{(i)}-\frac{1}{2}\int_{0}^{t}\xi_{s}^{(i)}ds + S_{t}^{(i)}$$ where $\xi_{s}^{\gre (i)}$ is the i-th conjugate variable of $X_{s}^{(i)}$'s in the sense of \cite{Vo5}, $S_{t}^{(i)}$ a free Brownian motion free with respect to $X_{0}^{(i)}$. Let us recall that for $(M=W^*(X_{1},...,X_{N}),\tau)$, if $X_1,...,X_N$ are algebraically free, the i-th partial free difference quotient $\partial_i:L^2(M)\to L^2(M)\o L^2(M)$ is the unique derivation densely defined (on non-commutative polynomials) such that $\partial_j(X_i)=1_{i=j}1\o 1$. Then the i-th conjugate variable is defined by $\partial_i^*(1\o1)\in L^2(M)$ if it exists. In \cite{S07}, this equation was solved in order to get stationary solutions for analytic conjugate variable, and thus this paper proved that in case of analytic conjugate variable if moreover $W^*(X_1,...,X_N)$ is $R^\omega$ embeddable, then $\delta_0(X_{1},...,X_{N})=\delta^*(X_{1},...,X_{N})=N$. Of course, if we believe in the previous general equality, this should be proved in much more general cases, e.g. for $L^2$ conjugate variable, i.e. finite Fisher information. The goal of this paper is to prove this equality in an intermediate case, under a Lipschitz like condition on conjugate variables.  Let us emphasize our definition does not involve operator Lipschitz functions and is relative to $M$, but it is nothing but the usual notion of being a Lipschitz function of $X$ (for instance applied by functional calculus)  in the one variable case (this is a Sobolev like definition of lipschitzness in the one variable case)~:

\begin{definition}
$(M=W^*(X_{1},...,X_{N}),\tau)$ is said to satisfy a \textit{Lipschitz conjugate variable condition} if the partial free difference quotients $\partial_i$ are defined and if the conjugate variables $\partial_i^*1\o 1$ exist in $L^2(M)$ (for all $i$) and moreover are in the domain of the closure $\overline{\partial}$ of $(\partial_1,...,\partial_N)$  with $\overline{\partial_j}\partial_i^*1\o 1\in M\overline{\o} M^{op}\subset L^2(M\o M^{o\red p})\simeq L^2(M\o M)$ (von Neumann tensor product, $M^{op}$ the opposite algebra).
\end{definition}

Let us state a precise result, the main byproduct of our work in this respect (cf. corollary \ref{MicrostateChap2}) is the following~:

\begin{mytheorem}
Consider $(M=W^*(X_{1},...,X_{N}),\tau)$ a $R^\omega$-embeddable finite von Neumann algebra satisfying a Lipschitz conjugate variable condition. 
Then the microstate entropy dimension $\delta_0(X_1,...,X_N)=N$.
\end{mytheorem}



We show in section {\red 4.3} that q-Gaussian variables (introduced in \cite{BS91}) are a non-trivial instance of non-commutative variables having Lipschitz conjugate variables for small $q$ (e.g. $|q|N\leq 0.13$ thus improving a computation in \cite{S07} and proving that $\delta_0$ does not only converge to $N$ for small $q$ but is identically equal to $N$ and thus equal to $\delta^*(X_1,...,X_N)$. {\gre One could actually prove with our techniques  $\delta_0(X_1,...,X_N)=N$ is still valid on a slightly larger range of $q$'s, i.e. as soon as $|q|N<1$ and $|q|\sqrt{N}\leq 0.13$, we will detail this elsewhere}.

Let us come back to our stochastic differential equation setting. By lack of a theory of  ``non-commutative Lipschitz functions", we will rather solve a dual Stochastic Partial Differential Equation
with the right stationarity property to get this result.

To explain the equation we solve, let us note that if we call $\Phi_s(X)=X_s$ the automorphism we hope being able to build solving the above equation, then Ito formula implies e.g. for any non-commutative polynomials P :
$$\Phi_t(P(X_{0}))=\Phi_0(P(X_{0}))-\frac{1}{2}\int_{0}^{t}\Phi_s(\Delta (P(X_{0})))ds + \int_{0}^{t}\Phi_s\o\Phi_s(\delta (P(X_{0})))\#dS_{s}.$$
We refer the reader to the main text for reminders about free stochastic integration. Here $\delta=(\delta_1,...,\delta_n)$ is the free difference quotient, $\Delta=\delta^*\delta$. This is  also the equation the author solved in a more recent paper \cite{Dab10}  in a much more general context but with more limited applications to microstate free entropy dimensions, because of a lack of control on the von Neumann algebra in which we build the process in this new approach.

Here we will thus rather solve the following dual equation :
\begin{align}\label{IntEq}X_{t}=X_{0}-\frac{1}{2}\int_{0}^{t}\Delta (X_{s})ds + \int_{0}^{t}\delta (X_{s})\#dS_{s},\end{align}

where $\delta$ will be an appropriate extension of the free difference quotient by zero on the free Brownian motion and a corresponding $\Delta=\delta^*\delta$. It is well known in quantum stochastic integration over symmetric Fock space that solving right Hudson Parthasarathy equations instead of left HP equations enables to solve equations in a mild sense (see e.g. \cite{FagnolaWills}) as in the classical Stochastic Differential equation case (see e.g. \cite{DaPrato}). {\gre It} is in order to use those techniques we considered this equation rather than the previous one.

Before describing the content of this paper, let us explain the relation of our work with classical stochastic partial differential equations. There are basically three main approaches to analysing SPDEs: the ``martingale
(or martingale measure) approach'' (cf. \cite{Wal86}), the ``semigroup (or mild
solution) approach'' (cf. \cite{DaPrato}) and the ``variational approach''
(cf. \cite{Roz}).  We will mainly refer to the above monographs instead of the original enormously rich literature. Beyond those mainstream approaches, one should also mention Krylov's $L^p$-theory \cite{KryLp}
and Kostelenez's methods \cite{Kos} using limits of particle systems, and also an old approach for more concrete SPDEs using SDEs in nuclear spaces and distributions (e.g. \cite{Ust}). Here we will adapt to the free SDE context a part of the semigroup approach using variational techniques. To compare with our work, we thus insist here on those two approaches.

To fix ideas a general SPDE considered in the classical literature is often of the form :
$$dX_t(\xi) = A(t,X_t(\xi),D_{\xi}X_t(\xi),D^2_{\xi} X_t(\xi))dt
+ B(t,X_t(\xi),D_{\xi}X_t(\xi))dW_t.$$
Since this will be our main interest, we will mainly focus on the linear time independent case where $A$ is thus a second order differential operator, $B$ a first order one, let say valued in Hilbert-Schmidt operators from the noise space $Y$(let say $W_t$ is a standard (with covariance $Id$) cylindrical brownian motion on a Hilbert space $Y$) to the space $H$ where $X_t$ lives. The linear case was also motivated in the early theory by filtering problems giving rise to linear equations suitable for the variational approach.  

Both approaches share the common features of considering SPDEs as SDEs valued in infinite dimensional spaces (usually Hilbert spaces of Sobolev types), using PDE techniques often in an abstract functional analytic setting.

Let us describe first the variational approach, originating from \cite{Par72},\cite{Par75}, \cite{KryRoz} (we refer to \cite{Roz}, and the recent introductory \cite{PreRoc} in the coercive case). Usually, solutions are built here by a Galerkin scheme, in first projecting the equation to finite dimensional sub-Hilbert spaces. After this transformation, the equation is an ordinary SDE solved by usual techniques. At this level, estimates (for this approximation) are proved, enabling to take a (weak) limit. The equation is first solved in a weak sense, avoiding to require $X_t\in D(A)$.
The standard assumption is the so called coercivity condition (also called superparabolic case in \cite{Roz} when considered for concrete differential operators).

This is roughly written : $$2\langle x, Ax \rangle +||B(x)||_{HS}^2 + \delta ||x||_U^2\leq  K ||x||_H^2,$$
where $U$ is another Hilbert space such that $U\subset D(B)$ continuously (often if $A$ is time independent self-adjoint, -A positive, $U=D((-A)^{1/2})$, for instance, to fix ideas). Having $\delta>0$ then enables to get a bound on $||X_t||_H$ and say $\int_0^t||X_s||_U^2ds$ giving sufficiently many regularity to get a weak limit, so that $B(X_t)$ makes sense, and to solve the equation weakly (i.e. after taking scalar products with $y\in D((-A)^{1/2})$ for instance in the self-adjoint case). Unfortunately, the case we are interested in is not coercive, it only satisfies the dissipative condition where $\delta=0$ above (sometimes called degenerate parabolic case). This equation is enough to guaranty a bound on $||X_t||_H$  but nothing more.  In this case, the usual method (for instance  used in \cite{Roz} Chapter 4 in a concrete differential operator setting), is to replace $B$ by $(1-\epsilon) B$ (or $A$ by $(1+\epsilon) A$) to get a coercive equation and get the bound on $||X_t||_U$ necessary to get a weak limit by another technique. In the coercive case, there are also standard ways of getting regularity results (for instance, we assume a dissipative inequality under an overall $(-A)^{1/2}$ for instance again in the self-adjoint positive case, i.e   $-2\langle Ax, Ax \rangle +||(-A)^{1/2}B(x)|| _{HS}^2\leq  K ||(-A)^{1/2}x||_H^2,$ and not surprisingly deduce from this a bound on $||(-A)^{1/2}X_t||_H$, the equation being ideal to apply Gronwall's lemma and get a bound in that way). One thus uses these standard ways of getting regularity via a priori estimates for the approximating equation to get a weak limit. 

The semigroup approach uses the semigroup $\phi_t$ generated by $A$ and rewrites the equation after "variation of constants", and thus looks for a so-called mild solution, i.e. a solution of~: 
$$X_t=\phi_t(X_0)+\int_0^t \phi_{t-s}(B(X_s)dW_s).$$
Then the goal is to use regularization properties of this semigroup to solve this equation. For instance, to solve non-linear equations with only continuous coefficients for $B$, one can use compactness of the semigroup and use compactness arguments (and get a stochastically weak solution, i.e. not adapted to the filtration of the Brownian motion. Note we use in this paper only the word weak in its PDE sense as in \cite{DaPrato}). In more standard assumptions, the semigroup is only assumed  analytic, or with generator a variational or a self-adjoint operator. We will be mainly interested in the semigroup approach under the same assumptions as in the variational approach. Indeed, in our free SDE setting, it is not quite clear what kind of Galerkin's method could make us recover an ordinary free SDE setting. Moreover as we will see, we will use extensively really weak notions of being a mild solution we will call ultramild as a crucial tool to get results {\gre under} really weak assumptions. Anyways, the interest of the semigroup approach for us lies in the fact it replaces Galerkin's method by a fixed point argument (for contractions) under the same coercivity assumption. Then, we can again prove a priori estimates to extend this to the degenerate parabolic case we will be interested in (since only this case can give stationary solutions in our examples).

Let us now finally describe the content 
of this article. In Section 1, we solve a really general stochastic partial differential equation (formally of the form (\ref{IntEq}))  with much less restrictive assumptions on $\delta$, $\Delta$. We find natural assumptions to get two kinds of solutions we will call mild and ultramild solutions, this second really weak sense of getting a solution has never been considered, to the best of our knowledge, in previously quoted contexts. These conditions are natural analogs of the dissipativity condition above (in case of ultramild solutions) and the dissipativity condition under $(-A)^{1/2}$ (to get regularity conditions and for us mild solutions). We have also to include in these conditions general compatibility assumption{\gre s} trivially checked in our main example.

In Section 2, we prove that we can check our assumptions to get mild solutions in the free difference quotient case with a Lipschitz conjugate variable type assumption as explained above. 
 The crucial issue here is to prove non-trivial domain properties of $\delta$, $\Delta$ which are usually checked classically using general regularity results of PDEs not available in our non-commutative context. One of the crucial tools here is 
 an easy boundedness criteria for $1\o\tau\circ\delta$  found by the author in \cite{nonGam} (cf. Lemma \ref{FPC2} infra. 
 coming from  \cite[lemma 10]{nonGam}).

In Section 3, we prove that as soon as we stick to our case of main interest of a derivation and the corresponding divergence form operator, it suffices to check $||X_t||_2=||X_0||_2$ in order to prove any ultramild solution to be stationary, as we want in order to get lower bounds on microstate free entropy dimension. 
Especially, this is always true if we can get a mild solution, and this is really likely why ultramild solutions {\gre were} never considered before. If we don't get an isometric map, solving those equations is  not such useful.

In Section 4, we explain our main application about computation of microstate free entropy dimension under Lipschitz conjugate variable assumption. In section {\red 4.3}, as we said above, we explain the concrete example of q-Gaussian variables{\red, after several general preliminaries gathered in section 4.2}. Here the proof of Lipschitz conjugate variable relies heavily on Bo\.zejko's analog of Haagerup's Inequalities \cite{Boz}.
 We {\gre will consider elsewhere how one could use a non-coassociative derivation to compute microstate free entropy dimension of q-Gaussian variables in a slightly less small range of q's}. Of course it is possible that a better understanding of combinatorial properties of those examples may give more extended ranges of q's with the same free SDE techniques. Finally, in section {\red 4.4}, we explain how hard it is to get stationary solutions in an example of derivations on group von Neumann algebras coming from group cocycles valued in the left regular representation, case also considered in \cite{S07}. Here coassociativity like assumptions are not available to get "easily" mild solutions, this is why we were motivated in being able to get solutions in a really general sense like ultramild solutions under somehow an automatically verified assumption. Indeed, in such a concrete example one can easily find a necessary and sufficient condition for getting $||X_t||_2=||X_0||_2$. However, it is expressed in terms of conservativity of a classical Markov process, well known to be hard to check. This is not such surprising since unitarity properties of left Hudson-Parthasarathy equations are also expressed in terms of conservativity of quantum Markov processes (see e.g. the survey \cite{FagnolaHP}). Of course, the occurrence of a classical process is only explained by our special example on groups, anyone interested in such a criteria for more general processes may be able to generalize this to a general case using conservativity of an appropriate quantum Markov process. However since any useful (easy to check) sufficient condition for proving this conservativity is not really available (even in HP case) beyond conditions really similar to those of our section 2 to get mild solutions (cf. \cite{ChebFag}), we don't enter in this general question here. Let us conclude with two remarks. Having in mind those similarities with questions of unitarity of solutions of Hudson-Parthasarathy equations, we can wonder whether a duality theory analogous to (Journe) duality of left and right HP equations could be developed in our context. In the other direction, one may wonder whether an ultramild like definition of a solution may be useful for right HP equations (e.g. in order to solve them under weaker conditions expressed in terms of conservativity assumptions similar to left HP equations) or whether the new approach of \cite{Dab10} could be translated in the context of left HP equations.  

\bigskip\textbf{Acknowledgments} 
 The author would like to thank D. Shlyakhtenko and P. Biane for plenty of useful
discussions,  A. Guionnet and an anonymous referee for plenty of useful comments helping him improve the exposition of previous versions of this text.
\section{A general Stochastic differential equation with unbounded coefficients}

Let $M_{0}$ be a $W^{*}$-probability space (with separable predual), $S_{t}^{(i)}$ ($i\in\N$) a family of 
free Brownian motions. Consider $M=M_{0}\star W^{*}(S_{t}^{(i)})$ the free product of $W^{*}$-probability spaces (so that $S_{t}^{(i)}$ are free with $M_{0}$ inside $M$) and consider finally the natural filtration $M_{s}=M_{0}\star W^{*}(S_{t}^{(i)},t\leq s)$. As a side remark, note we always use scalar products linear in the second variable.

In this part, we will be interested in the following equation~:
\begin{equation}\label{SPDES}X_{t}=X_{0}-\frac{1}{2}\int_{0}^{t}\Delta (X_{s})ds + \int_{0}^{t}\delta (X_{s})\#dS_{s},\end{equation}

where $\Delta~:L^{2}(M)\rightarrow L^{2}(M)$ and $\delta~:L^{2}(M)\rightarrow L^{2}(M)\otimes L^{2}(M)^{\bigoplus \N}$ are closed densely defined operators  and keeping invariant for $t\in [0,T]$ $L^{2}(M_{t})$ (resp. sending it to $L^{2}(M_{t}\otimes M_{t})^{\bigoplus \N}$ and with the analog property for its adjoint. {\gre By convention we say a closed densely defined unbounded operator keep invariant a subspace or send a closed subspace $S$ into another one $S'$ if its restriction to the intersection $D\cap S$ of its domain $D$ with $S$ is valued in $S'$ and the restricted operator is again a closed densely defined unbounded operator}. See subsection 1.2 for a definition of Stochastic integral){\red .} The sense in which we will solve this equation will be made precise in the 3 following sections~: the first will deal with some miscellaneous results about stochastic integration in our context, the second will introduce stochastic convolution, the key tool to define mild solutions and the third one will prove in the free Brownian case some well-known (in the classical Brownian motion case) relations between mild and strong solutions, and introduce ultramild solutions (the three kinds of solutions we will be interested in getting). Let us right now state the two class of assumptions we will need to get mild (resp. ultramild) solutions in the last subsection of this section\footnote{Note we will always write $A\circ B$ the closure of the composition of two closed operators if possible, and the usual composition if they are not closed, without risk of confusion. Sometimes we will even write $AB$ for the same object.}.
We also consider given another operator $\tilde{\delta}$ satisfying the assumptions for $\delta$, i.e $\tilde{\delta}~:L^{2}(M)\rightarrow L^{2}(M)\otimes L^{2}(M)^{\bigoplus \N}$ are closed densely defined operators keeping invariant the corresponding filtrations. {\gre (This will be useful for further applications in the $q$-Gaussian variable case. We will consider them elsewhere).}

We fix a few notation{\red 
} before stating the assumption. We will write (for $t>s$) $U\#(S_{t}-S_{s})=\sum_{i=0}^{\infty}U^{(i)}\#(S_{t}^{(i)}-S_{s}^{(i)})$ the Hilbert space isomorphism between the infinite direct sum of coarse correspondences $(L^{2}(M_{s})\otimes L^{2}(M_{s}))^{\bigoplus \N}$ {\gre and a corresponding subspace of $L^2(M_t)$}, where $U^{(i)}\#(S_{t}^{(i)}-S_{s}^{(i)})$ is the linear isomorphism extending $a\otimes b\#(S_{t}^{(i)}-S_{s}^{(i)})=a(S_{t}^{(i)}-S_{s}^{(i)})b$, for $a,b,c\in M_{s}$. Likewise, for 3-fold tensor products, we write 
$(a\otimes b\o c)\#_1(S_{t}^{(i)}-S_{s}^{(i)})=a(S_{t}^{(i)}-S_{s}^{(i)})b\o c, (a\otimes b\o c)\#_2(S_{t}^{(i)}-S_{s}^{(i)})=a\o b(S_{t}^{(i)}-S_{s}^{(i)}) c$ and their corresponding $L^2$ extensions.
 On a direct sum, we write $Diag(A_i)$ the operator acting diagonally, e.g. $Diag(A_i)(b_1\oplus b_2)=A_1b_1\oplus A_2b_2$.

 We will first use an assumption $\Gamma_0(\omega)$ to get really weak forms of solutions, we will call ultramild. 
 \begin{equation*}
\Gamma_{0}(\omega)\begin{cases}
a) & \text{$\Delta $ is a positive self-adjoint operator, $\eta_{\alpha}=\frac{\alpha}{\alpha+\Delta }$. }
\\
b.1) & D(\Delta ^{1/2})\subset D(\delta ),\\ 
c.1) & \text{for any $x\in D(\Delta ^{1/2})$  we have~:
$-||\Delta ^{1/2}x||_{2}^{2} + ||\delta (x)||_{2}^{2}\leq \omega ||x||_{2}^{2}$.}\\
d.1) &\text{There exists a closed densely defined positive operator}\\ & \text{$\Delta ^{\otimes}=:Diag(\Delta ^{\otimes}_i):(L^{2}(M)\otimes L^{2}(M))^{\bigoplus \N}\rightarrow (L^{2}(M)\otimes L^{2}(M))^{\bigoplus \N}$  }\\ & \text{(acting diagonally with respect to the direct sum and keeping invariant,}\\ &\text{for any $t$, $L^{2}(M_{t}\otimes M_{t})^{\bigoplus \N}$ and) such that $\forall U\in L^{2}(M_{s})\otimes L^{2}(M_{s})\cap D(\Delta _{i}^{\otimes})$:  }\\ & \text{$U\#(S_{t}^{(i)}-S_{s}^{(i)})\in D(\Delta )$ and $\Delta (U\#(S_{t}^{(i)}-S_{s}^{(i)}))=\Delta _{i}^{\otimes}(U)\#(S_{t}^{(i)}-S_{s}^{(i)})$.}
\\d.2) & \text{Moreover $\delta (U\#(S_{t}^{(i)}-S_{s}^{(i)}))$ is orthogonal to $L^{2}(M_s\o M_s)^{\gre \bigoplus \N}$.}
\end{cases}
\end{equation*}
{\gre We will use a variant $\Gamma_{0u}(\omega)$ with an extra Assumption $e$ added to get an extra uniqueness, technically provided by checking our solution will be also what we will call an ultraweak solution. For convenience, we write $((L^{2}(M_0))\o(L^{2}(M_0)\ominus \C)^{\o n-1}\o (L^{2}(M_0)))=V_n$. }{\gre We also consider an orthonormal basis $(e_n)_{n\in \N}$ of $L^{2}(W^{*}(S_{t}^{(i)}, t>0, i\in \N))\ominus  \C
$ such that for all $i$, $e_i\in W^{*}(S_{t}^{(j)}, t>0, j\in \N)$. We also write for any $a_1,...,a_{n+1}\in M_0$ with $\tau(a_i)=0$,($i\neq 1,n+1$) $(a_1\o...\o a_{n+1})\#(e_{i_1}\o...\o e_{i_n})=a_1e_{i_1}a_2...a_ne_{i_n}a_{n+1},$ and likewise $U\#(e_{i_1}\o...\o e_{i_n})$  the isometric extension to $U\in (L^{2}(M_0))\o(L^{2}(M_0)\ominus \C)^{\o n-1}\o (L^{2}(M_0))$.

Finally, we consider $\E_n=Span\{ \alpha_{0}(S_{v_{1}}^{(k_{1})}-S_{u_{1}}^{(k_{1})})...(S_{v_{n}}^{(k_{n})}-S_{u_{n}}^{(k_{n})})\alpha_{n} \ |\ \alpha_i\in M_0, k_i\in \N,  [u_{1},v_{1}]\times...\times[u_{n},v_{n}]\subset\R_{+}^{n}-D^{n}\}$, where $D^{n}$ is the full diagonal (the set of $n$-tuples having at least one pair of equal coordinates). We write $\bar{\E_n}$ the closure of $\E_n$ in $L^2(M)$.}

\begin{equation*}
\Gamma_{0u}(\omega)\begin{cases}
 & \Gamma_{0}(\omega)
\\e.1) &{\gre D(\Delta)\o_{alg}D(\Delta)\subset D(\delta^*).}
 \\e.2) &\text{\gre $\forall n\in \N^* \exists\Delta ^{\otimes (n+1)}:V_n \rightarrow V_n$  closed densely defined positive operator}\\ &\text{\gre 
  such that $\forall(i_1,...,i_n)\in \N^{n}$,  $\forall U\in D(\Delta^{\otimes (n+1)})$ :}  \\ &U\#(e_{i_1}\o...\o e_{i_n})\in D(\Delta ) 
\\ & \text{\gre and $\Delta (U\#(e_{i_1}\o...\o e_{i_n}))=\Delta^{\otimes (n+1)}(U)\#(e_{i_1}\o...\o e_{i_n})$.}\\ e.3) &\text{ \gre $\E_n\cap D(\delta)\subset \bar{\E_n}$ dense and $\delta(\E_n\cap D(\delta))$ orthogonal to $\oplus_{p+q< n}(\E_p\o \E_q)^{\gre \bigoplus \N}$.}
\\ e.4) &\text{\gre $\exists\D=\text{span}\D\subset L^{2}(M_0)\ominus \C$ a dense subspace such that:}\\  &
(\D\oplus \C)\otimes\D^{\o(n-1)}\o(\D\oplus \C)\subset D(\Delta ^{\otimes (n+1)}).
\end{cases}
\end{equation*}

The main assumption (useful to get mild solutions) will be called $\Gamma_{1}(\omega,C)$ (and $\Gamma_{1u}(\omega,C)$ if we add $\Gamma_{0u}(\omega)$) :
\begin{equation*}
\Gamma_{1}(\omega,C)\begin{cases}
 & \Gamma_{0}(\omega)\\ 
b.2) & D(\delta \circ\Delta ) \subset D(\delta ).\\
c.2) & \text{For any $x\in D(\Delta ^{1/2})$~:}\\ & -||\Delta\eta_{\alpha}^{1/2}x||_{2}^{2} + Re\langle\delta (\Delta\eta_{\alpha}(x)),\delta(x)\rangle\leq \omega ||\Delta^{1/2}\eta_{\alpha}^{1/2}x||_{2}^{2}.
 \\ d.3) &D(\Delta ^{1/2})\subset D(\tilde{\delta} ), D(\Delta ^{1/2})\text{ a core for } \tilde{\delta}  ,D(\tilde{\delta} \circ\Delta )\subset D(\tilde{\delta} )\\&\text{$\forall U\in L^{2}(M_s\o M_s)\cap D(\Delta^{\o}_i),\tilde{\delta} (U\#(S_{t}^{(i)}-S_{s}^{(i)}))$ orthogonal to  }\\&\text{$L^{2}(M_s\o M_s)^{\gre \bigoplus \N}$ and we assume we have a closed densely defined}\\ & \tilde{\delta}^{\otimes}:=\tilde{\delta}^{\otimes1}\oplus \tilde{\delta}^{\otimes2},\tilde{\delta}^{\otimes i}:(L^{2}(M^{\otimes 2}){\red )}^{\bigoplus \N}\rightarrow (L^{2}(M^{\otimes 3}){\red )}^{\bigoplus \N^2}\\ & \tilde{\delta}^{\otimes i}((x_j)_j)=:(\tilde{\delta}^{\otimes i}_{l,j}(x_j))_{(l,j)}, \tilde{\delta}^{\otimes i}_{l,j}:(L^{2}(M)^{\otimes 2})\rightarrow (L^{2}(M)^{\otimes 3}) \\ & \text{(keeping invariant, for any $t$, the filtration induced by $M_{t}$ and) such that}\\ &\text{ $\forall U\in L^{2}(M_{s})\otimes L^{2}(M_{s})$: $U\#(S_{t}^{(j)}-S_{s}^{(j)})\in D(\tilde{\delta_l })$ if $U\in D(\tilde{\delta}^{\otimes 1}_{l,j}\oplus \tilde{\delta}^{\otimes 2}_{l,j})$}\\ & \text{and $\tilde{\delta }_l(U\#(S_{t}^{(j)}-S_{s}^{(j)}))=\sum_{i=1}^2\tilde{\delta}^{\otimes i}_{l,j}(U)\#_i(S_{t}^{(j)}-S_{s}^{(j)})$.}
\\d.4) &  { \gre D(\Delta ^{\o 1/2})\subset D(\tilde{\delta}^{\o} ), D(\Delta ^{\o1/2})\text{ a core for } \tilde{\delta} ^{\o}}.
\\
f.1) & D(\tilde{\delta} \circ\Delta )\subset D(\Delta ^{\otimes}\circ\tilde{\delta} ).\\
f.2) &\text{There exists a bounded operator $\H$ on $L^{2}(M\otimes M)^{\bigoplus \N}$}\\ & \text{keeping invariant for any $s$, $L^{2}(M_{s}\otimes M_{s})^{\bigoplus \N}$ with $||\H||\leq C^{1/2}$ ($C\geq1$)}\\ & \text{such that for any $x\in D(\tilde{\delta} \circ\Delta )$:}\\ &
\Delta ^{\otimes}\circ\tilde{\delta}(x)-\tilde{\delta} \circ\Delta (x)=\H(\tilde{\delta} (x)).\\ 
g)& D(\tilde{\delta})= D(\delta)=:\tilde{\gre \D}, \forall x\in \tilde{\D},  C^{-1/2}||\tilde{\delta}(x)||_2\leq ||\delta(x)||_2\leq C||\tilde{\delta}(x)||_2,\\ &\text{ thus }D(\tilde{\delta} \circ\Delta )=D(\delta \circ\Delta ).\\ 
h)&D(\Delta)\subset D(\Delta ^{\otimes 1/2}\circ\tilde{\delta}\oplus \tilde{\delta}\oplus \tilde{\delta}^{\otimes}\delta)\text{ and for }x\in D(\Delta):
\\ & ||\tilde{\delta}^{\otimes}\delta(x)||_2^2\leq ||\Delta ^{\otimes 1/2}\circ\tilde{\delta}(x)||_2^2+C ||\tilde{\delta}(x)||_2^2.
\end{cases}
\end{equation*}

We will write $\phi_{t}$ the semigroup exponentiating $-1/2\Delta$ and $\phi_{t}^{\otimes}$ the semigroup associated to $-1/2\Delta^{\o}$.

In most cases we will be interested in the case $\tilde{\delta}=\delta$ in which case assumptions $g, h$ {\gre will be} automatic (using {\gre a variant of} $c.1$ {\gre for tensor variants } for {\gre the inequality in $h$ and $f$ for the domain assumption in $h$}). 

In the applications we have in mind, the strong assumption is $f$, the other ones being automatically verified and just important in this general setting.

\bigskip

\emph{General ideas and strategy}

With those notation fixed and before entering into technical details,  let us explain the intuition behind our results (in the case $\tilde{\delta}=\delta$, the general case is a slight extension following an idea of \cite{ChebFag}). 
In our general setting here, the proofs will follow closely the classical case, and therefore the intuition is basically the same, namely, since we want to solve SDEs with unbounded coefficients, with $\Delta$ a kind of divergence form operator (as we will consider in the next part) the corresponding semigroup is regularizing, and we want to use this. That's why we introduce mild solutions. As explained by various equivalences in section 2.1.3, the difference with strong solutions is only related to the domain in which we want to build the solution, we only require being in the domain of $\Delta^{1/2}$ (or even $\delta$) for mild solutions, and as soon as it is in the domain of $\Delta$, a mild solution is a strong solution, the converse being always true. The idea behind ultramild solutions is slightly trickier. Let us explain it in saying   $\phi_{t-s}^{\o}\circ \delta$ may have a much huger domain again than $\Delta^{1/2}$ and we want to use this regularization effect to have solutions with almost no conditions. Indeed, condition $f$ above will be really hard to check even with strong conditions (section 2.2), that's why we want to have solutions in a sense as general as possible.
We can also say that the current section somehow takes natural analogs of classical assumptions in the non-commutative case and check we can work with them.

 It is maybe also useful to have several ideas in mind, and first how those conditions will appear in a really natural way in the proof. To get an estimate on $||X_t||_2^2$, or $||\Delta^{1/2}(X_t)||_2^2$ (first on an approximation of the solution, in the spirit of moving from a degenerate parabolic case to a superparabolic case), the common idea is to differentiate, and try to apply Gronwall's lemma. Conditions $c.1$ (called dissipativity in the classical case) and $f.2$ above correspond exactly to what we want, in order to apply this lemma respectively in those cases.  The second idea is that if we replace $\delta$ by $(1-\epsilon)\delta$ the equation is much easier to solve, it is of superparabolic type (instead of degenerate parabolic type, said otherwise this gives a kind of coercivity, {\red see} \cite{Roz} for a presentation of this point in a more concrete setting but more clearly than in \cite{DaPrato}). First, in this case, there will be a Picard iteration argument to solve it, second we win something in terms of domains, assumption c is enough to bound $||\Delta^{1/2}(X_t^\epsilon)||_2^2$, assumption $f$ to get a bound on $||\Delta(X_t^\epsilon)||_2^2$ (those bounds diverging in $\epsilon$, of course). Anyways this will enable us to have respectively mild or strong  solutions of an approximating equation converging to a solution of our equation, even if the solution without $\epsilon$ will be only a mild or ultramild solution (note that in the case $\delta\neq \tilde{\delta}$ we will lose the strong solution property but keep mild solutions, hopefully we don't use this improvement in terms of getting a strong solution).  Somehow, to get later in section {\red 3.3} stationarity of the equation, this will be much more crucial to have an approximation by a mild solution than an ultramild solution, since Ito formula, already tricky to apply for mild solutions, seems to be completely unusable for ultramild solutions. Moreover there is the general idea that if you get a bound on $||\Delta^{1/2}(X_t^\epsilon)||_2^2$ (uniform in $\epsilon$), you can get a Cauchy condition in  $||.||_2$ norm and thus norm convergence, but however in general we will work only with weak convergence. Finally, we will also show in section 1.3 that mild solutions are also weak solutions, in a usual duality sense of weak solutions, however, we don't have an analog for ultramild solutions. Thus we will also introduce a notion of ultraweak solution, mainly to get uniqueness results in applying Laplace transform techniques, our general result will be "there exists a unique ultraweak solution which is also an ultramild solution", and limit of mild solutions of  approximating equation{\gre s}.
 
 Before starting, we sum up the content of the next sections. In section 1.1, we give miscellaneous  definitions and results moving almost commutation properties of our assumptions to stochastic integrals. In section {\red 1.2 }, we prove a{\gre n} integration by parts formula for a stochastic convolution we introduce. We avoid proving a free variant of the usually used stochastic Fubini Theorem in using ad hoc proofs in our really special case. In section 1.3 we introduce our different kinds of solutions and prove relations between them (explained above). Section 1.4 contains our general theorem. 

Let us prove right now an easy consequence of our main assumptions. Note we often use the bound $||\Delta\eta_{\alpha}||\leq 2\alpha$ coming from $\Delta\eta_{\alpha}=\alpha(1-\eta_{\alpha})$.
\begin{lemma}\label{Gamma1}
Assume $\Gamma_{1}(\omega,C)$.Then, there exists for any $\alpha\in (0,\infty)$ bounded operators $\H_{\alpha}$  from the graph of $\Delta ^{1/2}$:$G(\Delta ^{1/2})\subset L^{2}(M)^{\oplus2}$ to $L^{2}(M\otimes M)^{\bigoplus \N}$ sending, for $s$, $L^{2}(M_{s})^{\oplus2}$ to $L^{2}(M_{s}\otimes M_{s})^{\bigoplus \N}$ with $||\H_{\alpha}||\leq \max(1,\sqrt{\omega})C$  such that
for any $x\in D(\Delta ^{1/2})$  (if we write $\eta_{\alpha}=\frac{\alpha}{\alpha+\Delta }$ and the analog $\eta_{\alpha}^{\otimes}=\frac{\alpha}{\alpha+\Delta ^{\otimes}}$):
$$\Delta^{\otimes}\eta_{\alpha}^{\otimes}\circ\tilde{\delta} (x)-\tilde{\delta} \circ\Delta \eta_{\alpha}(x)=\H_{\alpha}(x\oplus\Delta ^{1/2}(x)).$$
Moreover, for each  $x\in D(\Delta ^{1/2})$, $\H_{\alpha}(x\oplus\Delta ^{1/2}(x))$ converges in $L^{2}$ to $\H(\tilde{\delta} (x))$ when $\alpha\rightarrow\infty$.

Finally, there is also a bounded $\tilde{\H}_{\alpha}: L^{2}(M) \to L^{2}(M\otimes M)^{\bigoplus \N}$, with $||\tilde{\H}_{\alpha}||\leq \frac{C}{\alpha}\sqrt{\omega+2\alpha}$ and the same invariance of filtration properties, such that for any $x\in D(\tilde{\delta})$ 
 :$$\tilde{\delta} \eta_{\alpha}(x)-\eta_{\alpha}^{\o}\tilde{\delta} (x)=\tilde{\H}_{\alpha}(x),$$
\end{lemma}
\begin{proof}[Proof]
Let $\H$ be given by assumption $f.2$. Let us define $\H_{\alpha}$. For $x\in D(\Delta^{1/2} )$, $\eta_{\alpha}(x)\in D(\Delta^{3/2} )\subset D(\tilde{\delta} \circ\Delta )$, thus applying the equation for $\H$ in $f.2$, we get~: $$\eta_{\alpha}^{\o}\Delta ^{\otimes}\circ\tilde{\delta} \eta_{\alpha}(x)-\eta_{\alpha}^{\o}\tilde{\delta} \circ\Delta \eta_{\alpha}(x)=\eta_{\alpha}^{\o}\H(\tilde{\delta} \eta_{\alpha}(x)).$$
Thus multiplying by $\frac{1}{\alpha}$ 
and using $\alpha(1-\eta_{\alpha})=\Delta\eta_{\alpha}$, we get~:
$$\tilde{\delta} \eta_{\alpha}(x)-\eta_{\alpha}^{\o}\tilde{\delta} (x)=\frac{1}{\alpha}\eta_{\alpha}^{\o}\H(\tilde{\delta} \eta_{\alpha}(x)).$$
Especially, defining  $\tilde{\H}_{\alpha}=\frac{1}{\alpha}\eta_{\alpha}^{\o}\H\tilde{\delta} \eta_{\alpha}$, we get the last statement since (by assumptions $f.2,g,c.1$) $||\tilde{\H}_{\alpha}||\leq \frac{C^{1/2}}{\alpha}||\tilde{\delta} \eta_{\alpha}||\leq \frac{C}{\alpha}\sqrt{\omega+2\alpha}$ and moreover, by $d.3$, $D(\Delta^{1/2})$ is a core for $\tilde{\delta}$.

Moreover we also deduce : 
$$\Delta^{\o}\eta_{\alpha}^{\o}\tilde{\delta} (x)=-\frac{1}{\alpha}\Delta^{\o}\eta_{\alpha}^{\o}\H(\tilde{\delta} \eta_{\alpha}(x))+ \tilde{\delta}\Delta\eta_{\alpha}(x)+\H(\tilde{\delta} \eta_{\alpha}(x)),$$
thus equivalently $$\Delta^{\o}\eta_{\alpha}^{\o}\tilde{\delta} (x)-\tilde{\delta} \Delta \eta_{\alpha}(x)=\eta_{\alpha}^{\o}\H(\tilde{\delta }\eta_{\alpha}(x)).$$

This suggests  $\H_{\alpha}(x\oplus\Delta ^{1/2}(x))=\eta_{\alpha}^{\o}\H(\tilde{\delta} \eta_{\alpha}(x))$. In that way, the equation is verified, the stability properties come from the assumptions and using properties {\gre $f.2,g,c.1$ again}, we get~:
$$||\H_{\alpha}(x\oplus\Delta ^{1/2}(x))||_{2}^{2}\leq C||\tilde{\delta} \eta_{\alpha}(x)||_{2}^{2}\leq C^{2}\left(||\Delta ^{1/2}(x)||_{2}^{2}+\omega ||x||_{2}^{2}\right).$$
Thus we get the bound on $||\H_{\alpha}||$, and \begin{align*}||\H_{\alpha}(x\oplus\Delta ^{1/2}(x))-\H(\tilde{\delta} (x))||_{2}&\leq ||\eta_{\alpha}^{\o}\H(\tilde{\delta} (\eta_{\alpha}-1)(x))||_{2}+||(\eta_{\alpha}^{\o}-1)\H(\tilde{\delta} (x))||_{2}\\ &\leq C^{1/2}||\tilde{\delta} (\eta_{\alpha}-1)(x)||_{2}+||(\eta_{\alpha}^{\o}-1)\H(\tilde{\delta} (x))||_{2},\end{align*}and the right hand side goes to zero using again assumption $c.1$.
\end{proof}

\subsection{Stochastic integration in presence of $\delta $ and $\Delta $}

Following \cite{BS98}, except for the value in $L^{2}(M\otimes M)^{\bigoplus \N}$ instead of $L^{2}(M\otimes M)$
of bi-processes, we write  $\B_{2}^{a}([0,T])$  the completion of the space of
simple bi-processes on $[0,T]$, adapted 
with respect to the algebraic direct sum $(M_{t}\otimes M_{t})^{\bigoplus \N}$,  in the following norm~:
$$||U||_{\B_{2}^{a}([0,T])}=\left(\int_{0}^{T}||U_{s}||_{L^{2}(\tau\otimes\tau)^{\oplus \tiny\N}}^{2}ds\right)^{1/2}.$$
{\gre We may also write  later $\B_{2}^{a}$ instead of $\B_{2}^{a}([0,T])$ when $T$ is clear from the context and similarly for variants later.}

Let us remark that this space can also be seen as a subspace of $L^{2}([0,T],L^{2}(\tau\otimes\tau)^{\bigoplus \N})$
 (defined, say, in Bochner's sense) and we will always see it as such a subspace.

Then, recall that the map $U \mapsto \int_{0}^{T}U_{s}\#dS_{s}=\sum_{j=0}^{\infty} \int_{0}^{T}U_{s}^{(j)}\#dS_{s}^{(j)}$ is an
isometric linear extension from  $\B_{2}^{a}([0,T])$ to $L^{2}(M,\tau)$ of the usual map, sending, for $a,b\in M_s$, $a\o b1_{[s,t)}$ seen in the $i$-th component of the direct sum to $a(S_t^{(i)}-S_s^{(i)})b$.
Thus, we can remark for further use that weak convergence in $L^{2}(\tau)$ of a sequence of
stochastic integrals $\int_{0}^{T}U_{s}^{n}\#dS_{s}$ is equivalent to
weak convergence of its integrand $U_{s}^{n}$ in
$L^{2}([0,T],L^{2}(\tau\otimes\tau)^{\gre\bigoplus \N})$. 

Analogously, one can consider $\B_{2}^{a(\o 3)}([0,T])$ for processes adapted to the filtration $(M_{t}\otimes M_{t}\o M_t)^{\bigoplus (\{1,2\}\times \N^2)}$. We define $\int_{0}^{T}U_{s}{\gre\#
}dS_{s}\in L^2(M\o M)^{\bigoplus \N}$ 
 in extending the definition for $a,b,c\in M_s$, $U_u=a\o b\o c1_{[s,t)}(u)$ seen in the $1,j,i$-th component  of the direct sum  $\int_{0}^{T}a\o b\o c1_{[s,t)}(u)\#dS_{u}=a(S_t^{(i)}-S_s^{(i)})b\o c$ in the $j$-th component, and when seen in the $2,{\gre j,i}$-th component$\int_{0}^{T}a\o b\o c1_{[s,t)}(u)\#dS_{u}=a\o b(S_t^{(\gre i)}-S_s^{(i)})c$ in the $j$-th component.

We will write $\B_{2,\delta \circ\Delta ^{\beta}}^{a}([0,T])$ (for $\beta\in\{0,1/2,1\}$, resp. $\B_{2,\Delta ^{\beta}}^{a}([0,T])$ for $\beta\in\{1/2,1,3/2\}$) the completion with respect to the
following norms of what we will call $\delta \circ\Delta ^{\beta}$-simple
 adapted processes (resp. $\Delta ^{\beta}$-simple
 adapted processes), i.e. processes of the form $X=\sum_{j=1}^{M}X_{j}1_{[t_{j},t_{j+1})}$ with
$X_{j}\in D(\Delta^{\beta})\cap\bigcap_{b=0}^{2\beta}D(\delta\circ\Delta^{b/2})$ (resp. $X_{j}\in D(\Delta^{\beta}))$~:
\begin{align*}&||X||_{\B_{2,\delta\circ\Delta ^{\beta}}^{a}}=\left(\int_{0}^{T}\sum_{b=0}^{2\beta}||\delta \circ
  \Delta ^{b/2}X_{s}||_{L^{2}(\tau\otimes\tau)^{\bigoplus \tiny\N}}^{2}+\sum_{b=0}^{2\beta}||\Delta ^{b/2}X_{s}||_{L^{2}(\tau)}^{2}
  ds\right)^{1/2}.\\
&(resp.\  ||X||_{\B_{2,\Delta ^{\beta}}^{a}}=\left(\int_{0}^{T}\sum_{b=0}^{2\beta}||\Delta ^{b/2}X_{s}||_{L^{2}(\tau)}^{2}
  ds\right)^{1/2}.)\end{align*}

Of course, using $g$, one gets the same spaces if we replace $\delta$ by $\tilde{\delta}$.
Assuming $\Gamma_{0}(\omega)$ (especially condition $b$), we have clearly continuous embeddings $\B_{2,\delta \circ\Delta ^{\beta}}^{a}([0,T])\rightarrow L^{2}_{a}([0,T],L^{2}(M))$ (space of adapted processes), $\B_{2,\Delta ^{\beta}}^{a}([0,T])\rightarrow L^{2}_{a}([0,T],L^{2}(M))$, for $\beta'\leq\beta,\beta,\beta'\in\{0,1/2,1\}$,
$\B_{2,\delta \circ\Delta ^{\beta}}^{a}([0,T])\rightarrow\B_{2,\delta \circ\Delta ^{\beta'}}^{a}([0,T])$  and for $\beta'\leq\beta,\beta,\beta'\in\{1/2,1\}$, $\B_{2,\delta \circ\Delta ^{\beta}}^{a}([0,T])\rightarrow\B_{2,\Delta ^{\beta'}}^{a}$ $\B_{2,\Delta ^{\beta+1/2}}^{a}([0,T])\rightarrow\B_{2,\delta \circ\Delta ^{\beta}}^{a}([0,T])$ for $\beta\in\{0,1/2\}$ (using assumption $c$).

 From the assumptions on $\delta $ and $\Delta $, we remark that we can see for any $X_{s}\in \B_{2,\delta \circ\Delta ^{\beta}}^{a}([0,T])$, $\delta \circ\Delta ^{\beta}X_{s}$ as an element
of $\B_{2}^{a}$. Finally, let us note that if $B$  bounded 
operator on $L^{2}(M_{s}\otimes M_{s})^{\bigoplus \N}$, keeping invariant, for any $t$, $L^{2}(M_{t}\otimes M_{t})^{\bigoplus \N}$, and if $U_{s}\in \B_{2}^{a}([0,T])$, then $B(U_{s})\in \B_{2}^{a}([0,T])$.

 The following lemma is the goal of these definitions~:

\begin{lemma}\label{deltaInt}
Assume $\Gamma_{1}(\omega,C)$ for (i) and $\Gamma_{0}(\omega)$ for (ii) and (iii).
\begin{enumerate}\item[(i)] Let $X_{s}\in \B_{2,\Delta ^{1/2}}^{a}([0,T])$ then we have $\eta_{\alpha}(X_{s})\in \B_{2,\delta \circ\Delta }^{a}([0,T])$, $\Delta ^{\o}\eta_{\alpha}^{\o}\tilde{\delta} (X_{s}),\H_{\alpha}(X_{s}\oplus\Delta ^{1/2} (X_{s}))\in \B_{2}^{a}([0,T])$ and for $t\leq T$~:
 $$\int_{0}^{t}\Delta ^{\o}\eta_{\alpha}^{\o}\tilde{\delta} (X_{s})\#dS_{s}=\int_{0}^{t}\tilde{\delta} \circ\Delta (\eta_{\alpha}(X_{s}))\#dS_{s}+\int_{0}^{t}\H_{\alpha}(X_{s}\oplus\Delta ^{1/2}(X_{s}))\#dS_{s}.$$
 If $X_{s}\in \B_{2,\Delta ^{3/2}}^{a}([0,T])$, then $\tilde{\delta}^{\o }\delta(X_s)=(\tilde{\delta}^{\o k}_{i,j}\delta_j(X_s))_{k,i,j}\in \B_{2}^{a(\o 3)}([0,T])$,
 
 \noindent  $\int_{0}^{t}\delta (X_{s})\#dS_{s}\in D(\tilde{\delta})$ and we have the equation~: $$\tilde{\delta}\int_{0}^{t}\delta (X_{s})\#dS_{s}=\int_{0}^{t}\tilde{\delta}^{\o}\delta (X_{s})\#dS_{s}.$$
 
\item[(ii)]Likewise, for any $U_{s}\in \B_{2}^{a}([0,T])$ and $t,\alpha,\beta> 0$, $i\in \{1/2,1\}$:  $\eta_{\alpha}^{\o i}(U_{s}),\phi_{t}^{\o}(U_{s})\in \B_{2}^{a}([0,T]),$ $\Delta ^{\o}\eta_{\alpha}^{\o1/2}\eta_{\beta}^{\o1/2}(U_{s}),\Delta ^{\o}\phi_{t}^{\o}(U_{s})\in \B_{2}^{a}([0,T])$(and assuming $d.3,d.4,g$ we have also{\red,}

\n $\ \tilde{\delta}^{\o}\eta_{\alpha}^{\o i}(U_{s})=(\tilde{\delta}_{i,j}^{\o k}\eta_{\alpha}^{\o i}(U_{s}^j))_{k,i,j}\in \B_{2}^{a (\o 3)}([0,T])$),  and we have for $\tau\leq T$~:
\begin{align*}&\eta_{\alpha}^i(\int_{0}^{\tau}U_{s}\#dS_{s})=\int_{0}^{\tau}\eta_{\alpha}^{\o i}(U_{s})\#dS_{s},
\ \ \Delta \eta_{\alpha}^{1/2}\eta_{\beta}^{1/2}(\int_{0}^{\tau}U_{s}\#dS_{s})=\int_{0}^{\tau}\Delta ^{\o}\eta_{\alpha}^{\o1/2}\eta_{\beta}^{\o1/2}(U_{s})\#dS_{s},\\
&\phi_{t}(\int_{0}^{\tau}U_{s}\#dS_{s})=\int_{0}^{\tau}\phi_{t}^{\o}(U_{s})\#dS_{s}
,\ \ \Delta \phi_{t}(\int_{0}^{\tau}U_{s}\#dS_{s})=\int_{0}^{\tau}\Delta ^{\o}\phi_{t}^{\o}(U_{s})\#dS_{s},\\
&\tilde{\delta}\eta_{\alpha}^i(\int_{0}^{\tau}U_{s}\#dS_{s})=\int_{0}^{\tau}\tilde{\delta}^{\o}\eta_{\alpha}^{\o i}(U_{s})\#dS_{s}
 \ \ \ \text{if $d.3,d.4,g$ also hold}.
\end{align*}
Finally, for any $W\in L^{2}(M_t\o M_t)^{\oplus \N}$, any $V=(\int_{t}^{\tau}U_{s}\#dS_{s})$, $t\leq \tau$, with $V\in D(\delta)$, $\langle W,\delta(V)\rangle=0$.
\item[(iii)]
For $U_s\in \B_{2}^{a}([0,T])$, $\int_0^TU_s\#dS_s\in D(\Delta^{1/2})$ if and only if $U_s\in D(\Delta^{\o 1/2})$ for almost every $s$ and $\int _0^Tds ||\Delta^{\o 1/2}U_s||_2^2<\infty $. In this case $\Delta^{1/2}\int_0^TU_s\#dS_s=\int_0^T\Delta^{\o 1/2}(U_s
)\#dS_s$. {\red If $d.3,d.4,g$ also hold, then  for any $U_s$ with $\tilde{\delta}^{\o}(U_s)\in \B_{2}^{a (\o 3)}([0,T])$ (e.g. for $U_s=\delta(X_s)$, for $X_s\in\B_{2,\Delta}^{a}([0,T])$ if $h$ holds), we have $\int_0^TU_s\#dS_s\in D(\tilde{\delta})$ and } $\tilde{\delta}\int_0^TU_s\#dS_s=\int_0^T\tilde{\delta}^{\o}(U_s
)\#dS_s.$
\end{enumerate}
\end{lemma}

\begin{proof}[Proof]
First of all, the statements about $\eta_{\alpha}^{\o i}(U_{s}),\phi_{t}^{\o}(U_{s})\in \B_{2}^{a}([0,T])$,

\n $\Delta ^{\o}\eta_{\alpha}^{\o1/2}\eta_{\beta}^{\o1/2}(U_{s}),\Delta ^{\o}\phi_{t}^{\o}(U_{s})\in \B_{2}^{a}([0,T])$ and $\Delta ^{\o}\eta_{\alpha}^{\o}\tilde{\delta} (X_{s})\in \B_{2}^{a}$ follow from the remark before the lemma, since e.g. $||\Delta ^{\o}\eta_{\alpha}^{\o}||\leq 2\alpha$. Assuming $d.3,d.4,g$, the same is true for $\tilde{\delta}^{\o}\eta_{\alpha}^{\o i}(U_{s})\in \B_{2}^{a(\o 3)}([0,T])$.

If $X_{s}$ is a $\Delta ^{1/2}$-simple  process. By linearity, we can suppose   
$X_{s}=x1_{[t_1,t_2)}(s)$, with $x\in D(\Delta ^{1/2})\cap L^{2}(M_{t_1})$. In that case,
we have clearly $\eta_{\alpha}(X_{s})\in D(\tilde{\delta} \circ\Delta )$ (using assumption $d.3$) and the equality
stated is nothing but the one of lemma \ref{Gamma1}. In the general
case $X_{s}\in \B_{2,\Delta ^{1/2}}^{a}([0,T])$, take by density
$X_{s}^{n}$ $\Delta ^{1/2}$-simple processes converging to $X_{s}$ in
$\B_{2,\Delta ^{1/2}}^{a}([0,T])$. Then, since $\delta \circ\Delta ^{\beta}\eta_{\alpha}(X_{s}^{n})=\delta \circ\alpha^{\beta}(id-\eta_{\alpha})^{\beta}\eta_{\alpha}^{1-\beta}(X_{s}^{n})$ (and using assumption $c.1$),

\n $\sum_{b=0}^{2}||\delta \circ\Delta ^{b/2}\eta_{\alpha}(X_{s}^{n}-X_{s}^{m})||_{2}^{2}\leq
(1+2\alpha+(2\alpha)^{2})
\left(||\Delta ^{1/2}(X_{s}^{n}-X_{s}^{m})||_{2}^{2}+\omega||X_{s}^{n}-X_{s}^{m}||_{2}^{2}\right) $.
Likewise $$\sum_{b=0}^{2}||\Delta ^{b/2}\eta_{\alpha}(X_{s}^{n}-X_{s}^{m})||_{2}^{2}\leq
(1+2\alpha+(2\alpha)^{2})
||X_{s}^{n}-X_{s}^{m}||_{2}^{2}.$$
 As a consequence, $(\eta_{\alpha}(X_{s}^{n}))$ converges in
$\B_{2,\delta \circ\Delta }^{a}([0,T])$, and by the embedding in $\B_{2,\Delta ^{1/2}}^{a}([0,T])$,
it converges to $\eta_{\alpha}(X_{s})$ which is thus in
$\B_{2,\delta \circ\Delta }^{a}$. Now, $\Delta ^{\o}\eta_{\alpha}^{\o}\tilde{\delta} (X_{s}),\tilde{\delta} (\Delta\eta_{\alpha}X_{s})\in \B_{2}^{a}([0,T])$. Therefore, applying the equation of lemma \ref{Gamma1}, $H_{\alpha}(X_{s}\oplus\Delta ^{1/2}(X_{s}))\in \B_{2}^{a}([0,T])$ and we have our equality after taking the isometric map of stochastic integration.
The second statement of (i) is proved in a similar way using $d.3$ for the equation, $g,h,f.1,c.1$ for boundedness results.
For (ii), the boundedness has already been discussed, and this explains the definition of the right hand sides of the
equations. The equalities are clear for simple processes (easy consequences of assumptions $d$, and $a$ for the semigroup. Note for $\eta_{\alpha}^{1/2}$ one can use $\eta_{\alpha}^{1/2}=\int_0^{\infty}\pi^{-1}\frac{t^{-1/2}}{1+t}{\red \eta_{\alpha(1+t)/t}dt}$ from \cite[lemma 2.2]{Pe06}), this concludes by density.
%
%

For the last statement of (ii) about orthogonality, since $\delta\eta_{\alpha}(V)\to\delta(V)$, we can assume by the beginning of (ii) (putting $\eta_{\alpha}$ in the stochastic integral), $U_s\in D(\Delta^\o)$. Again, it suffices to prove the simple process case, and this reduces to assumption $d.2$.

Finally, for the equivalence of (iii), note that the two first equalities of {\gre (ii) give}

\n$||\Delta^{1/2}\eta_{\alpha}(\int_0^TU_s\#dS_s)||_2^2=\int_0^T||\Delta^{\o 1/2}\eta_{\alpha}^{\o}(U_s
)||_2^2ds$. From this, letting $\alpha\to \infty$, the direct implication follows from monotone convergence theorem with $\Delta^{1/2}\eta_{\alpha}=\eta_{\alpha}\Delta^{1/2}$ and the reverse implication using also $\Delta^{1/2}$ is a closed operator. Let us check first for any $U\in \B_{2}^{a}([0,T])$, $\Delta^{1/2}\eta_{\alpha}(\int_0^TU_s\#dS_s)=\int_0^T\Delta^{\o1/2}\eta_{\alpha}^{\o}U_s\#dS_s$. Again it suffices to check it on simple processes, on which this comes from  $\Delta^{1/2}=\int_0^{\infty}\pi^{-1}t^{-1/2}{\red (id-\eta_t)dt}$ (cf. e.g. \cite{S2} or \cite{Kato}). Now, the stated result comes from $\alpha\to \infty$, the statement for $\tilde{\delta}$ is analogous.
\end{proof}

\bigskip

\subsection{A definition of free Stochastic convolution}

In this subsection, we assume $\Gamma_{0}(\omega)$. We want to give sense to the following kind of integral, for $U_{s}\in\B_{2}^{a}$~:$\int_{0}^{t}\phi_{t-s}(U_{s}\#dS_{s}).$
We will define it by $$\int_{0}^{t}\phi_{t-s}(U_{s}\#dS_{s})=\int_{0}^{t}\phi_{t-s}^{\o}(U_{s})\#dS_{s},$$ and we want to verify the usual properties of stochastic convolution.

For this, we have to verify that $\phi_{t-s}^{\o}(U_{s})1_{{\red [0,t]}}(s)\in
\B_{2}^{a}([0,t])$, and since $\phi^{\o}_{t-s}$ is a contraction, it is sufficient to
show this for $U_{s}$ a simple process, and thus even for $U1_{[u,v)}(s)$,
$U\in L^{2}(M_u\o M_u)$. But consider $u_{i,n}=u+i(v-u)/n$, then
$U^{n}=\sum_{i=0}^{n-1}\phi_{t-u_{i,n}}(U)1_{[u_{i,n},u_{i+1,n})}$ is
easily shown to converge in $L^{2}([0,t],L^{2}(M\o M))$ to
$\phi_{t-s}^{\o}(U)$ using strong continuity of $\phi^{\o}$, this concludes
the preliminaries for the definition. 

Let us define a variant of the spaces of the previous subsection useful to define a really weak form of solutions we will call in the next subsection : ``ultramild'' solutions.
We will write $\B_{2,\phi\delta}^{a}([0,T])$ for the completion with respect to the
following norm of $\delta$-simple
 adapted processes, i.e. recall this means processes of the form $X=\sum_{j=1}^{M}X_{j}1_{[t_{j},t_{j+1})}$ with
$X_{j}\in D(\delta)$ ~:
\begin{align*}&||X||_{\B_{2,\phi\delta}^{a}}=\left(\int_{0}^{T}\left[\int_{0}^{t}||\phi_{t-s}^{\o}\delta(X_{s})||_{L^{2}(\tau\otimes\tau)^{\bigoplus \tiny\N}}^{2}ds+||X_{t}||_{L^{2}(\tau)}^{2}
  \right]dt\right)^{1/2}.\end{align*}

We have clearly a continuous embedding $\B_{2,\delta}^{a}([0,T])\rightarrow \B_{2,\phi\delta}^{a}([0,T])$ using subsection 1.1 and the above remark defining stochastic convolution (the first space being clearly dense in the second by definition).
We get thus a map $\gamma:B_{2,\delta}^{a}([0,T])\rightarrow L^{2}_{a}([0,T],L^{2}(M))$
such that $\gamma(X)_{t}=\int_{0}^{t}\phi_{t-s}(\delta(X_{s})\#dS_{s})$ and clearly $||\gamma(X)||_{L^{2}_{a}([0,T],L^{2}(M))}\leq ||X||_{\B_{2,\phi\delta}^{a}}$ so that $\gamma$ extends to a continuous map (also called) $\gamma:B_{2,\phi\delta}^{a}([0,T])\rightarrow L^{2}_{a}([0,T],L^{2}(M))$.

We also want to show that
$t\mapsto\langle\int_{0}^{t}U_{s}\#dS_{s},\zeta\rangle$ {\gre is}
of bounded variation so that  we can define something like
$\int_{0}^{t}\langle U_{s}\#dS_{s},\zeta\rangle$
(with the same value)
and see $\langle U_{s}\#dS_{s},\zeta\rangle$ as
a measure on $\R_{+}$.  But since stochastic integration is an isometric map onto
its image we can project $\zeta$ on this space thus write its projection
$\int_{0}^{t}V_ {s}\#dS_{s}$, and the result is a consequence of the
isometry property.

Finally, we want to define for $\zeta(.)\in C^{1}([0,T],L^{2}(M))$~:
$\int_{0}^{t}\langle U_ {s}\#dS_{s},\zeta(s)\rangle$ and show a relation
with stochastic convolution in a special case.
For this, first note that the family of functions of the form
$\varphi(.)\zeta_{0}$, for $\varphi(.)\in C^{1}([0,T],\C)$ and $\zeta_{0}\in
L^{2}(M)$ linearly spans a dense subset of $C^{1}([0,T],L^{2}(M))$, thus
consider also first $\zeta(.)$ in this linear span. Consider also $\mathcal{U}_{t}=\int_{0}^{t} U_{s}\#dS_{s}$.

Define $\int_{0}^{t}\langle U_
{s}\#dS_{s},\varphi(s)\zeta_{0}\rangle=\langle\int_{0}^{t} \overline{\varphi(s)}U_
{s}\#dS_{s},\zeta_{0}\rangle$ using the previous paragraph, and consider {\gre as}
in this paragraph the projection of $\zeta_{0}$ on the space of stochastic
integrals $\int_{0}^{t}V_{s}\#dS_{s}$. Then compute using integration by
parts~:
\begin{align*}\int_{0}^{t}\langle U_
{s}\#dS_{s},\varphi(s)\zeta_{0}\rangle&=\int_{0}^{t}\varphi(s)\langle U_
{s},V_{s}\rangle ds=\varphi(t)\int_{0}^{t}\langle U_
{s},V_{s}\rangle ds -
\int_{0}^{t}\varphi'(s)\langle\mathcal{U}_{s},\zeta_{0}\rangle ds \\ &=\langle\mathcal{U}_{t},\zeta(t)\rangle -
\int_{0}^{t}\langle\mathcal{U}_{s},\zeta'(s)\rangle ds,\end{align*}
which extends by linearity on the above mentioned linear span.

But now, we get the bound~:
\begin{align*}\left|\int_{0}^{t}\langle U_
{s}\#dS_{s},\varphi(s)\zeta_{0}\rangle\right|&\leq ||\mathcal{U}_{t}||_{2}\left(||\zeta(t)||_{2} +
t\ \sup_{s}||\zeta'(s)||_{2}\right)
 ,\end{align*}

using $||\mathcal{U}_{t}||_{2}$ is increasing with t. 
 We can thus extend our linear map by continuity to $\zeta(.)\in
 C^{1}([0,T],L^{2}(M))$ and we have also the equality~:
\begin{equation}\label{STCONV}\int_{0}^{t}\langle U_
{s}\#dS_{s},\zeta(s)\rangle=\langle\mathcal{U}_{t},\zeta(t)\rangle -
\int_{0}^{t}\langle\mathcal{U}_{s},\zeta'(s)\rangle ds.\end{equation}


Consider finally $\zeta(s)=\phi_{t-s}(\zeta)$, with $\zeta\in D(\Delta )$, writing as before $\int_0^\infty V_{s}\#dS_s$ the projection of $\zeta$ on the space of stochastic integrals. 
Using this last equality, we get ~:
\begin{align*}\int_{0}^{t}\langle &U_{s}\#dS_{s},\phi_{t-s}(\zeta)\rangle\\&=\langle\mathcal{U}_{t},\zeta\rangle  -
\int_{0}^{t}\langle\frac{1}{2}\Delta \phi_{t-s}(\mathcal{U}_{s}),\zeta\rangle
ds=\langle\mathcal{U}_{t},\zeta\rangle  -
\int_{0}^{t}\int_{0}^{s}\langle\frac{1}{2}\Delta ^{\o}\phi_{t-s}^{\o}(U_{u}),V_{u}\rangle
du ds\\ &=\int_{0}^{t}\langle U_{s},V_{s}\rangle ds  -
\int_{0}^{t}\langle\int_{u}^{t}\frac{1}{2}\Delta ^{\o}\phi_{t-s}^{\o}(U_{u})ds
,V_{u}\rangle du\\ &=\int_{0}^{t}\langle U_{s},V_{s}\rangle ds  -
\int_{0}^{t}\langle U_{u}-\phi_{t-u}^{\o}(U_{u}),V_{u}\rangle du=\langle\int_{0}^{t}\phi_{t-s}(U_{s}\#dS_{s}),\zeta\rangle.\end{align*}
In the first line we used our identity since $\phi_{{\gre t}-s}(\zeta) \in C^{1}([0,T],L^{2}(M))$ for $\zeta\in D(\Delta )$ and then lemma \ref{deltaInt} (ii).

Line 3 is only a computation, first with the differential equation, second, with the definition of stochastic convolution after simplification.

In line 2, we have to justify application of Fubini theorem. 
Note  that
$\zeta=\eta_{\alpha}(z)$ (since by Hille-Yosida theory $Range(\eta_\alpha)=D(\Delta)$, see e.g. (1.3) in the proof of Chapter 1 proposition 1.5 in \cite{MaR}), if the projection of $z$ is written
$\int_{0}^{T}W_{s}\#dS_{s}$, then $V_{s}=\eta_{\alpha}^{\o}(W_{s})$ a.e. by lemma \ref{deltaInt} (ii). Thus $V_{s}$ is a.e. in
$D(\Delta ^{\o})$. We can now use Cauchy-Schwarz inequality :
\begin{align*}\int_{0}^{t}ds\int_{0}^{s}du|\langle\frac{1}{2}\Delta ^{\o}\phi_{t-s}^{\o}(U_{u}),V_{u}\rangle|&\leq \int_{0}^{t}ds\left(\int_{0}^{s}du||\frac{1}{2}\phi_{t-s}^{\o}(U_{u})||_2^2\right)^{1/2}(\int_{0}^{s}du||\Delta ^{\o}V_{u}||_2^2)^{1/2}\\ &\leq t\left(\int_{0}^{t}du||U_{u}||_2^2\right)^{1/2}||\Delta(\zeta)||_2<\infty.\end{align*}

Starting from the second line above, applying Fubini to go upwards after a change of variable {\gre $t-s=s'-u$}, we also obtain :
\begin{align*}\langle\int_{0}^{t}\phi_{t-s}(U_{s}\#dS_{s}),\zeta\rangle&=\int_{0}^{t}\langle U_{s},V_{s}\rangle ds  -
\int_{0}^{t}\langle\int_{u}^{t}\frac{1}{2}\Delta ^{\o}\phi_{s-u}^{\o}(U_{u})ds
,V_{u}\rangle du\\ &=\langle\mathcal{U}_{t},\zeta\rangle  -
\int_{0}^{t}\langle\frac{1}{2}\Delta \int_{0}^{s}\phi_{s-u}(U_{u}\#dS_{u}),\zeta\rangle
ds.\end{align*}


\begin{proposition}\label{IPPSC}
(Integration by parts for stochastic convolution) For $\zeta\in D(\Delta)$, $U\in B_{2}^a([0,t])$, we have~:
\begin{align*}\int_{0}^{t}\phi_{t-s}(U_{s}\#dS_{s})&=\int_{0}^{t} U_{s}\#dS_{s}  -
\Delta \int_{0}^{t}ds\frac{1}{2}\phi_{t-s}(\int_{0}^{s} U_{v}\#dS_{v})
\\ &=\int_{0}^{t} U_{s}\#dS_{s}  -\frac{1}{2}\Delta\int_{0}^{t}ds\int_{0}^{s}\phi_{s-u}( U_u\#dS_{u}),\end{align*}

\begin{align*}\int_{0}^{t}\langle U_{s}\#dS_{s},\phi_{t-s}(\zeta)\rangle =\langle\int_{0}^{t}\phi_{t-s}(U_{s}\#dS_{s}),\zeta\rangle.
\end{align*}

\end{proposition}

Note the following useful formula we will use often later :
\begin{equation}\label{basic}||\phi_t(x)||_2^2=||x||_2^2-\int_0^t||\Delta^{1/2}\phi_{s}(x)||_2^2ds.
\end{equation}
\begin{proposition}\label{Cruc}
For $Y\in \B_{2,\delta}^{a}([0,T])$, define $\gamma(Y)_t=\int_0^t\phi_{t-s}(\delta(Y_s))\#dS_s)$, then $\gamma(Y)_t\in D(\Delta^{1/2})$ for a.e. $t\leq T$ and morever : 
\begin{equation}\label{crucial}\int _0^T||\Delta^{1/2}\gamma(Y)_t||_2^2dt=-||\gamma(Y)_T||_2^2+\int_0^Tdt||\delta(Y_t)||_2^2.\end{equation}
Moreover, assume $\Gamma_1(\omega,C)$ for any $B$ among $\tilde{\delta}$,$\Delta^{1/2}$, $\alpha,\alpha'>0$, then :
\begin{equation}\label{crucialsupp}\begin{split}||B\eta_{\alpha}^{ 1/2}\eta_{\alpha'}^{1/2}(\gamma(Y)_{T})||_{2}^{2}&=||B\eta_{\alpha}^{ 1/2}\eta_{\alpha'}^{1/2}\int_0^T\delta(Y_s)\#dS_s||_2^2\\ &-\int_{0}^{T}dt\Re\langle B\Delta\eta_{\alpha}^{ 1/2}\eta_{\alpha'}^{1/2}\gamma(Y)_{t},B\eta_{\alpha}^{ 1/2}\eta_{\alpha'}^{1/2}\gamma(Y)_{t}\rangle.\end{split}
\end{equation}
\end{proposition}
\begin{proof}
By Fubini-Tonneli Theorem and the remark before the proposition, we deduce $$\int _0^Tdt\int_0^tds||\Delta^{\o1/2}\phi_{t-s}^{\gre\o}(\delta(Y_s))||_2^2=\int _0^Tds||(\delta(Y_s))||_2^2-||\phi_{T-s}^{\gre\o}(\delta(Y_s))||_2^2.$$
Thus lemma  \ref{deltaInt} (iii) concludes to the first statement. Since $B\eta_{\alpha}$ is a bounded operator, and from the first statement and $\Gamma_{1}(\omega,C)\ g,c.1$ for the last term, all the terms in (\ref{crucialsupp}) are continuous on  $\B_{2,\delta}^{a}([0,T])$. As a consequence, it suffices to prove it for simple processes of even $Y=X1_{[s,t)}$, $X\in D(\delta)$.

From lemma \ref{deltaInt} (and with an obvious notation $B^{\o}$), this reduces the statement to \begin{align*}&\int_s^{\gre T} du||B^{\o}\eta_{\alpha}^{\o 1/2}\eta_{\alpha'}^{\o 1/2}\phi_{T-u}^{\o}\delta(X)||_2^2 =||B^{\o}\eta_{\alpha}^{\o 1/2}\eta_{\alpha'}^{\o 1/2}(\delta(X))||_2^2 (T-s)\\&-\int_s^Tdv \Re \int_s^vdu\langle B^{\o}\Delta^{\o}\eta_{\alpha}^{\o 1/2}\eta_{\alpha'}^{\o 1/2}\phi_{v-u}^{\o}\delta(X),B^{\o}\eta_{\alpha}^{\o 1/2}\eta_{\alpha'}^{\o 1/2}\phi_{v-u}^{\o}\delta(X)\rangle. \end{align*}

But this is obvious after applying Fubini on the last integral and integrating along $v$ {\gre (using $||\Delta^{\o 1/2}\phi_{v-u}^{\o}||\leq\frac{cst}{\sqrt{v-u}}$ and $d.4$ in the $\tilde{\delta}$ case to bound it by the corresponding $\Delta^{1/2}$ case)}.
\end{proof}

\subsection{Useful links between mild solutions and strong solutions}

In this part, we will also work under assumption $\Gamma_{0}(\omega)$.
Let us define four kinds of solutions.
\begin{definition}
 We will call a \textbf{strong solution} an element $X_{t} \in\B_{2,\Delta}^{a}$ satisfying (\ref{SPDES}). A \textbf{mild solution} will be an $X_{t} \in\B_{2,\delta}^{a}$ satisfying~:
\begin{equation}\label{SPDEM}X_{t}=\phi_{t}(X_{0}) + \int_{0}^{t}\phi_{t-s}(\delta(X_{s})\#dS_{s}).\end{equation}
We call \textbf{ultramild solutions},  solutions of (\ref{SPDEM}) in $\B_{2,\phi\delta}^{a}$.
We call a\textbf{ weak solution} an $X_{t} \in\B_{2,\delta}^{a}$ such that, for any $\zeta\in D(\Delta )$:
$$\langle X_{t}, \zeta \rangle = \langle X_{0}, \zeta
\rangle  -\frac{1}{2}\int_{0}^{t}\langle X_{s}, \Delta (\zeta)\rangle ds +
\int_{0}^{t}\langle \delta(X_{s})\#dS_{s},\zeta\rangle.$$
\end{definition}

We will first recall analogs of usual results (in classical SPDE theory)
concerning the link between strong solutions and mild solutions.
We mainly follow here the proofs (for a
classical Brownian motion and a classical SPDE) of
\cite{DaPrato} Chapter 6.

\begin{proposition}\label{SImpM}
A strong solution of (\ref{SPDES}) (in $\B_{2,\Delta}^{a}$) is also a mild solution (even a solution of (\ref{SPDEM}) in $\B_{2,\Delta^{1/2}}^{a}$.)
\end{proposition}

\begin{proof}[Proof]
 First, 
 note that for any $\zeta(.)\in C^{1}([0,T];D(\Delta ))$, and any
$t\in [0,T]$, we have~:
$$\langle X_{t}, \zeta(t) \rangle = \langle X_{0}, \zeta(0)
\rangle + \int_{0}^{t}\langle X_{s}, -\frac{1}{2}\Delta (\zeta(s)) + \zeta'(s)\rangle ds +
\int_{0}^{t}\langle \delta (X_{s})\#dS_{s},\zeta(s)\rangle.$$

(To prove this, use (\ref{STCONV}) to compute the stochastic part and use an integration by parts to get the other term).

Finally, consider $\zeta(s)=\phi_{t-s}(\zeta)$, which is in
$C^{1}([0,T];D(\Delta ))$, if say $\zeta\in D(\Delta^2)$. The terms inside the usual integral cancel out and you get~:
$$\langle X_{t}, \zeta \rangle = \langle X_{0}, \phi_{t}(\zeta)
\rangle  +
\int_{0}^{t}\langle \delta(X_{s})\#dS_{s},\phi_{t-s}(\zeta)\rangle.$$
Inasmuch as you can take any $\zeta\in D(\Delta ^2)$ and $D(\Delta ^2)$ is dense (even a core for $\Delta$ by a standard result Theorem 3.24 p275 in Chapter V of Kato's book \cite{Kato}),
you get the result (using proposition \ref{IPPSC}).
\end{proof}

\begin{proposition}\label{MImpS}
A mild solution $X_{t}$ is always a weak solution, and if it is also in $\B_{2,\Delta}^{a}$, then it is in fact a strong solution.
\end{proposition}
\begin{proof}[Proof]
Once we have proved that our mild solution is in fact a weak solution, 
%
we are in fact done,
since {\gre under} our assumption $\Delta (X_{s})$ is (Lebesgue-almost surely) well
  defined and in $L^{2}([0,T],L^{2}(M))$, showing that the wanted equation
  (\ref{SPDES}) under $\langle ., \zeta\rangle$, which concludes by
  density.

To show that we have the desired weak solution, we will merely use that the
solution is in $\B_{2,\delta}^{a}$, as required for a mild solution. Consider thus $\zeta\in D(\Delta ^2)$
\begin{align*}-\frac{1}{2}&\int_{0}^{t}\langle X_{s}, \Delta (\zeta)\rangle ds \\ &= 
-\frac{1}{2}\int_{0}^{t}\langle X_{0}, \phi_{s}(\Delta (\zeta))\rangle ds -
\frac{1}{2}\int_{0}^{t}ds \int_{0}^{s}\langle \delta(X_{u})\#dS_{u},\phi_{s-u}(
\Delta (\zeta))\rangle \\ &= 
-\frac{1}{2}\langle X_{0}, \int_{0}^{t}\phi_{s}(\Delta (\zeta))ds \rangle +\left(\langle\int_{0}^{t}\phi_{t-s}(\delta (X_{s})\#dS_{s}),\zeta\rangle -\langle\int_{0}^{t}\delta (X_{s})\#dS_{s},\zeta\rangle \right)\\ &= 
\langle \phi_{t}(X_{0})+
\int_{0}^{t}\phi_{t-s}(\delta (X_{s})\#dS_{s}), \zeta\rangle - \langle X_{0}, \zeta\rangle -
\int_{0}^{t}\langle \delta (X_{s})\#dS_{s},\zeta\rangle\\ &= 
\langle X_{t}, \zeta\rangle - \langle X_{0}, \zeta\rangle -
\int_{0}^{t}\langle
\delta (X_{s})\#dS_{s},\zeta\rangle.\end{align*}

The first line has been justified in proposition \ref{IPPSC} applied to the definition of mild solutions.  The last line, clearly concluding to what we wanted to prove, 
uses nothing but  again the definition of a mild solution. Of course the third line uses again the differential equation for $\phi$. The second line reduces to the second equation in proposition \ref{IPPSC}. 

\end{proof}

Finally, to get uniqueness results 
assuming only really weak conditions, we want to introduce a notion of ultraweak solution for which uniqueness will be easy to prove so that we will build unique ultraweak solutions which are also ultramild solutions. We thus assume $\Gamma_{0u}(\omega)$. This also  needs some results on chaotic decomposition very similar {\red to} those of section 5.3 in \cite{BS98} but not only for the free Fock space $F(H)$ with $H=L^{2}(\R_{+})$ but also for $H=L^{2}(\R_{+})^{\oplus \N}$ and moreover with an initial condition space $L^{2}(M_{0})$ i.e. we want to see a multiple stochastic integral variant of $L^{2}(M)=L^{2}(M_{0}\star \mathcal{SC}(H))\simeq L^{2}(M_{0})\star F(H)$. Since this requires a little bit of notation with nothing new, we merely state the results after introduction of notation.

For $f\in L^{2}(\R_{+}^{n}\times \N^{n},L^{2}(M_{0})^{n+1})$, we want to define a multiple stochastic integral 

$I(f)=\int f(t_{1},...,t_{n})\# dS_{t_{1}}...dS_{t_{n}}$. Of course, we extend it linearly and isometrically {\gre as} in \cite{BS98} after defining it on appropriate multiple of characteristic function $f=1_{\A}\delta_{k_{1},...,k_{n}}\alpha_{0}\otimes...\o \alpha_{n}$, $\alpha_{i}\in M_{0}$, $\delta_{k_{1},...,k_{n}}$ the function on $\N^{n}$ taking non zero value 1 only on the indicated support, $\A=[u_{1},v_{1}]\times...\times[u_{n},v_{n}]$ with $\A\subset \R_{+}^{n}-D^{n}$ ($D^{n}$ the usual full diagonal e.g. definition 5.3.1 of \cite{BS98}). {\gre We will call later step function any linear combination of such $f$'s with maybe $\alpha_{0}\otimes...\o \alpha_{n}$ replaced by $U\in L^{2}(M_0)^{\o(n+1)}$. We thus define :}
$$I(f):=\alpha_{0}(S_{v_{1}}^{(k_{1})}-S_{u_{1}}^{(k_{1})})\alpha_{1}...(S_{v_{n}}^{(k_{n})}-S_{u_{n}}^{(k_{n})})\alpha_{n}.$$

Recall $I(f)\in \E_n$ according to the notation before assumption $\Gamma_{0u}(\omega)$.

Then, we can write $f=\sum_{n={0}}^{\infty}f_{n}\in L^{2}(M_{0})\star F(H)$  so that $I(f)=\sum_{i=0}^{\infty}I(f_{n})$ define an isometric map $I:L^{2}(M_{0})\star F(H)\rightarrow L^{2}(M)$ determined by $I(f)\Omega=f$ ($\Omega$ the usual cyclic empty vector in Fock space), {\gre as } in proposition 5.3.2 of \cite{BS98}. Recall $P_{\Gamma}$ is the projection on adapted bi-processes. It is defined (as $\Gamma$) in proposition 5.3.12 in \cite{BS98} before free Bismut-Clark-Ocone formula. Note that this formula is also valid mutatis mutandis in our context, recall it involves $\nabla_s$ the gradient operator from definition 5.1.1 in \cite{BS98}. For instance in the really elementary case of $Y\in  \E_n\cap L^2(M_s)$,  it gives :
$$Y=E_{M_0}(Y)+\int_0^s(P_\Gamma\nabla_uY)\#dS_u.$$

{\gre Consider a step function $f=1_{\A}\delta_{k_{1},...,k_{n}}\alpha_{0}\otimes...\o \alpha_{n}$ as above such that $\alpha_0\o...\o\alpha_n=\sum_{\{j_0<...<j_k\}\subset\llbracket0,n\rrbracket} \J_{j_0,...,j_k}(U_{j_0,...,j_k}) $, the sum running over (maybe empty) subsets  of $\llbracket0,n\rrbracket$, with $1^{\o 1_{\{j_0\neq 0\}}}\o U_{j_0,...,j_k}\o 1^{\o1_{\{j_k\neq n\}}}\in D(\Delta^{\o (k+1+1_{\{j_0\neq 0\}}+1_{\{j_k\neq n\}})})\subset L^{2}(M_0)^{\o1_{\{j_0\neq 0\}}}\o  (L^{2}(M_0)\ominus \C)^{\o (k+1)}\o L^{2}(M_0)^{\o1_{\{j_k\neq n\}}}$  (with the notation $L^2(M_0)^{\o 0}=\C$, $1^{\o 0}$ the unit in this $\C$) and, for $J=\{j_0<...<j_k\}$ $\J_{j_0,...,j_k}=\J_J$ the unique isometric linear map $L^{2}(M_0)^{\o k+1}\to L^{2}(M_0)^{\o n+1}$ extended from $\J_{j_0,...,j_k}(a_0\o...\o a_k)=1^{\o j_0}\o a_0\o 1^{\o (j_1-j_0-1)}\o a_1...\o a_k\o1^{\o (n-j_k)}$.
We will later write $\D_n$ the space of linear combinations of such $\alpha_0\o...\o\alpha_n$'s. Since the images of $\J_{J\cup\{0,n\}}$ for different $J\cup\{0,n\}$ are orthogonal, there is obviously a closed densely defined positive operator $\Delta^{\o [n]}$ on the closure of $\D_n$ in $L^{2}(M_0)^{\o(n+1)}$ defined  on $\D_n$ by  \begin{align*}\Delta^{\o [n]}&(\alpha_0\o...\o\alpha_n)\\&=\sum_{J=\{j_0<...<j_k\}\subset\llbracket0,n\rrbracket} \J_{J\cup\{0,n\}}(\Delta^{\o (k+1+1_{\{j_0\neq 0\}}+1_{\{j_k\neq n\}})}(1^{\o 1_{\{j_0\neq 0\}}}\o U_{j_0,...,j_k}\o 1^{\o1_{\{j_k\neq n\}}})).
\end{align*}
Note that by assumption $e.4$ $D(\Delta^{\o [n]})\supset (\D\oplus\C)^{\o n+1}$, this explains $\Delta^{\o [n]}$ densely defined. Note also that the formula $\Delta(I(f))=I(\Delta^{\o [n]} (f))$ proved bellow first in case $f\in \D_n$ explains why $\Delta^{\o [n]}$ is positive on $\D_n$.
\begin{lemma}\label{ultraweakLemma}
Assume $\Gamma_{0u}(\omega)$. For any step function $f\in L^{2}(\R_{+}^{n}\times \N^{n})\o_{alg}D(\Delta^{\o [n]})$ as above $I(f)\in D(\Delta)$ and $\Delta(I(f))=I(\Delta^{\o [n]} (f))$. Moreover, if $f\in L^{2}(\R_{+}^{n}\times \N^{n})\o_{alg}\D_n$,  $P_{\Gamma}\nabla_{s}I(f)\in D(\delta^*)$ and $\delta^*P_{\Gamma}\nabla_{s}I(f)\in \overline{Span}\{h_k\in \E_k, k\leq n-1\}$.
\end{lemma}
\begin{proof} It suffices to consider $f\in L^{2}(\R_{+}^{n}\times \N^{n})\o_{alg}\D_n$ step function, e.g. $I(f):=\alpha_{0}(S_{v_{1}}^{(k_{1})}-S_{u_{1}}^{(k_{1})})...(S_{v_{n}}^{(k_{n})}-S_{u_{n}}^{(k_{n})})\alpha_{n}$ as before. Now for any $J=\{j_0,...,j_k\}\subset\llbracket0,n\rrbracket$
 we can write \begin{align*}&(S_{v_{1}}^{(k_{1})}-S_{u_{1}}^{(k_{1})})..(S_{v_{j_0}}^{(k_{j_0})}-S_{u_{j_0}}^{(k_{j_0})})\o(S_{v_{j_0+1}}^{(k_{j_0+1})}-S_{u_{j_0+1}}^{(k_{j_0+1})})...(S_{v_{j_1}}^{(k_{j_1})}-S_{u_{j_1}}^{(k_{j_1})})\o... (S_{v_{n}}^{(k_{n})}-S_{u_{n}}^{(k_{n})})\\&=\sum_{i_1,...,i_k\in \N, i_0,i_{k+1}\in \N\cup\{-1\}}\lambda_{i_0,...,i_{k+1}}^{(J)}(f) e_{i_0}\o...\o e_{i_{k+1}}.\end{align*}
with by convention $e_{-1}=1$ occurring only if $j_0=0$ and/or $j_k=n$ and in these cases the only non-vanishing $\lambda_i$ is respectively for $i$ with  $i_0=-1$ ( $i_{k+1}=-1$). (Note that if $J$ empty we have $k=-1$ and no tensor product). This uses only the orthonormal basis introduced before $\Gamma_0(\omega)$ since any $(S_{v_{j_i+1}}^{(k_{j_i+1})}-S_{u_{j_i+1}}^{(k_{j_i+1})})...(S_{v_{j_{i+1}}}^{(k_{j_{i+1}})}-S_{u_{j_{i+1}}}^{(k_{j_{i+1}})})$ is orthogonal to $\C$ when not $1$.

Especially, if we write $e_i=(e_{i_0}\o e_{i_1}...\o e_{i_{k+1}})$, $$I(f)=\sum_{J=\{j_0,...,j_k\}\subset\llbracket0,n\rrbracket}\sum_{i_1,...,i_k\in \N, i_0,i_{k+1}\in \N\cup\{-1\}}\lambda_{i_0,...,i_{k+1}}^{(J)}(f) (1\o U_{j_0,...,j_k}\o 1)\#e_{i},$$
so that using assumption $e.2$,  writing 
 $\lambda_{i}^{(J)}(f)=\lambda_{i_0,...,i_{k+1}}^{(J)}(f)$, $\N^{-}=\N\cup\{-1\}$, one gets : 
\begin{align*}&\Delta(I(f))=\sum_{J\subset\llbracket0,n\rrbracket}\sum_{i\in\N^{-}\times \N^k\times  \N^{-}}\\ & \lambda_{i}^{(J)}(f) (1^{\o 1_{\{j_0=0\}}}\o(\Delta^{\o (k+1+1_{\{j_0\neq 0\}}+1_{\{j_k\neq n\}})}(1^{\o 1_{\{j_0\neq 0\}}}\o U_{j_0,...,j_k}\o 1^{\o1_{\{j_k\neq n\}}})\o1^{\o1_{\{j_k= n\}}})\# e_{i}
\\ &=\sum_{J\subset\llbracket0,n\rrbracket}I((\J_{J\cup\{0,n\}}(\Delta^{\o (k+1+1_{\{j_0\neq 0\}}+1_{\{j_k\neq n\}})}(1^{\o 1_{\{j_0\neq 0\}}}\o U_{j_0,...,j_k}\o 1^{\o1_{\{j_k\neq n\}}}) )))\\&=I(\Delta^{\o [n]} (f)).\end{align*}

Recall $(P_{\Gamma}\nabla_{s}I(f))_{k_i}=1_{(u_i,v_i]}(s)(\alpha_{0}(S_{v_{1}}^{(k_{1})}-S_{u_{1}}^{(k_{1})})... \alpha_{i-1}\o \alpha_i...(S_{v_{n}}^{(k_{n})}-S_{u_{1}}^{(k_{n})})\alpha_{n},$ where $v_i$ is the hugest of all $v_k$'s ($(P_{\Gamma}\nabla_{s}I(f))_j=0$ if $j\neq k_i$) .
As before, one gets $P_{\Gamma}\nabla_{s}I(f)=I\o I(f_{1,s}\o f_{2,s})\in D(\Delta)\o_{alg}D(\Delta)$ so that assumption $e.1$ concludes to the domain statement. Note that $P_{\Gamma}\nabla_{s}I(f)\in (\E_{i-1}\o \E_{n-i})^{\oplus\N}$, and thus by assumption $e.3$ is orthogonal to $\delta(I(g_p))$ for any $p\geq n$ if $g_p\in L^{2}(\R_{+}^{n}\times \N^{n})\o_{alg}D(\Delta^{\o [p]})$. Thus $\delta^*P_{\Gamma}\nabla_{s}I(f)$ is orthogonal to all such $I(g_p)$, and by density for all $h_p\in \E_p,p\geq n$. This gives the last statement.
\end{proof}
}

We can now define :
\begin{definition}\label{ultraweak}
An \textbf{ultraweak solution} (of (\ref{SPDES})) is a{\gre n $L^2$}- weakly continuous adapted process $X_{t}$ in $L^{2}_{a,loc}(\R_{+},L^{2}(M))$ 
such that, for some $C$ and $\omega$, $||X_{t}||_{2}\leq C e^{\omega t}$  and for all  finite sums $g=\sum_{n}g_{n}$,  $g_{n} \in L^{2}(\R_{+}^{n}\times \N^{n})\o_{alg}\D_n$ step functions as above, then a.e. in $t\in\R_{+} $ :
\begin{align*}\langle I(g),X_{t}\rangle = \langle I(g),X_{0}\rangle -\frac{1}{2}\int_{0}^{t}ds\langle \Delta I(g),X_{s}\rangle + \int_{0}^{t}ds \langle\delta^{*}P_{\Gamma}\nabla_{s}I(g),X_{s}\rangle.\end{align*}
\end{definition}

Note that a weak solution satisfying $||X_{t}||_{2}\leq C e^{\omega t}$  is an ultraweak solution. Indeed from the free Bismut-Clark-Ocone formula and the previous lemma for domain issues, one gets : $$\langle I(g),\int_0^t\delta(X_s)\#dS_s \rangle=\langle \int_0^t(P_\Gamma\nabla_uI(g))\#dS_u,\int_0^t\delta(X_s)\#dS_s\rangle=\int_0^tds \langle\delta^{*}P_{\Gamma}\nabla_{s}I(g),X_{s}\rangle.$$
\subsection{Mild and Ultramild solutions}

Here is the main theorem in the general setting.
\begin{theorem}\label{main}
\begin{enumerate}\item[(i)] Let us assume $\Gamma_{0u}(\omega)$ and that $X_{0}\in L^{2}(M_{0})$, then equation
(\ref{SPDES}) has a unique ultraweak solution. This solution is also an Ultramild Solution $X_{t}$ 
 and we have, for every $T$ and   a.e. in t~:
\begin{align*}&||X_{t}||_{2}^{2}\leq e^{\omega t}||X_{0}||_{2}^{2},\\ &||X||_{\B_{2,\phi\delta}^{a}([0,T])}^{2}\leq
 ||X_{0}||_{2}^{2}2\frac{e^{\omega T}-1}{\omega}\ \  (\mathrm{or}\ \ 2T||X_{0}||_{2}^{2}\  \mathrm{if}\ \omega=0)
.\end{align*}
Furthermore, if we write $X_{s}^{\epsilon}$ a solution for $\delta$ replaced by $(1-\epsilon)\delta$ ($\epsilon\in (0,1]$), then $X_{s}^{\epsilon}$ is a  unique mild solution of this variant equation, i.e. a solution of (\ref{SPDEM}$\epsilon$) in $\B_{2,\Delta^{1/2}}^{a}$, and the solution built above $X_{t}$ is, for every $T$, 
 a weak
limit ($\epsilon\rightarrow 0$) in $\B_{2,\phi\delta}^{a}([0,T])$ and  strong limit in $C^{0}([0,T],(L^{2}(M),\sigma(L^{2}(M),L^{2}(M))))$ of the
solutions $X_{t}^{\epsilon}$. 
Finally, if we assume 
$X_{0}\in
D(\Delta^{\o1/2}\delta)\cap D(\Delta)\cap L^{2}(M_{0})$ then the solution satisfies a.e.: 
\begin{align*}
||X_{t}&-X_{0}-\delta (X_{0})\#S_{t}||_{2}^{2} \leq \frac{t^{2}}{4}||\Delta(X_{0})||_{2}^{2}+(e^{\omega t}-1)||X_{0}||_{2}^{2}\\ &+\frac{t^{2}}{2}\left( ||\Delta^{\o1/2}(\delta(X_{0}))||_{2}^{2}+\frac{\pi}{4}(||\Delta(X_{0})||_{2}^2+\omega||\Delta^{1/2}(X_{0})||_{2}^2)^{1/2}||\Delta^{\o1/2}(\delta(X_{0}))||_{2}\right)\\ &+t \sup_{s\in[0,t]} \left(||\Delta^{1/2}(\phi_{s}X_{0})||_{2}^{2}-||\delta(\phi_{s}X_{0})||_{2}^{2}\right)+t^2\omega/4||\Delta^{1/2}(X_0)||_2^2.
\end{align*}

\item[(ii)] Let us assume $\Gamma_{1u}(\omega,C)$ and that $X_{0}\in
D(\Delta^{1/2})\cap L^{2}(M_{0})$, then equation
(\ref{SPDES}) has a unique \textbf{Mild Solution} $X_{t}$
. Moreover, we have the following inequalities a.e:
\begin{align*}&||X_{t}||_{2}^{2}\leq e^{\omega t}||X_{0}||_{2}^{2},\ \  
\\ &||\tilde{\delta}(X_{t})||_{2}^{2}\leq ||\tilde{\delta}(X_{0})||_{2}^{2}e^{(6+2\omega)C^4t}.\end{align*}

If we write $X_{s}^{\epsilon}$ a solution for $\delta$ replaced by $(1-\epsilon)\delta$ ($\epsilon\in (0,1]$), then, if $\delta=\tilde{\delta}$, $X_{s}^{\epsilon}$ is a strong solution, i.e. a solution of (\ref{SPDEM}$\epsilon$) in $\B_{2,\Delta}^{a}([0,T])$, for every $T$, and otherwise, if $\delta\neq\tilde{\delta}$ a mild solution by (i). Furthermore, $X_{t}$ is, for every $T$, the weak
limit ($\epsilon\rightarrow 0$) in $\B_{2,\delta }^{a}([0,T])$ and strong limit in  $\B([0,T],L^{2}(M))$ (the space of bounded functions with uniform convergence) of the
solution $X_{t}^{\epsilon}$. 
\end{enumerate}
\end{theorem}

\begin{proof2}[Proof]

Let us sketch the plan of the proof. Step 0 proves uniqueness of ultraweak solutions, which is a useful preliminary. We will first find unique mild (resp strong in case (ii) ) solutions after replacing $\delta$ by $(1-\epsilon)\delta$ with $\epsilon>0$ [step 1]. Then, we will prove that when
$\epsilon\rightarrow 0$ we can get some weak convergence to an ultramild (resp a mild, in case (ii)) solution
of (\ref{SPDEM}), mainly by showing several inequalities like the ones stated in
the theorem [step 2 for part (i), step 3 for part (ii)]

\begin{step}
Uniqueness of ultraweak solutions in case (i). 
\end{step}
We have to show that an ultraweak solution with $X_{0}=0$ vanishes. The proof is in the spirit of Theorem 5.6 in \cite{FagnolaHP} in the symmetric Fock space context. For $g_{n}$ {\gre as} in definition \ref{ultraweak}, we prove by induction on $n$ $\langle I(g_{n}),X_{s}\rangle=0$. By density this gives the same for a step function $g_n\in L^{2}(\R_{+}^{n}\times \N^{n})\o_{alg}D(\Delta^{\o [n]})$. The induction hypothesis {\gre (and lemma \ref{ultraweakLemma})} or only the definition of $\nabla_{s}$ at initialization, gives the last integral in the definition of ultraweak solution vanishes, so that for a step function  $g=g_{n}\in L^{2}(\R_{+}^{n}\times \N^{n})\o_{alg}D(\Delta^{\o [n]})$: 
\begin{align*}\langle I(g),X_{t}\rangle =  -\frac{1}{2}\int_{0}^{t}ds\langle \Delta I(g),X_{s}\rangle .\end{align*}

Since $||X_{t}||_{2}\leq C \exp(\omega t)$, we can consider the Laplace transform for $\lambda > \omega$ so that we get~:
\begin{align*}\lambda \int_{0}^{\infty}dt \exp(-\lambda t)\langle I(g),X_{t}\rangle &=-\frac{\lambda}{2}\int_{0}^{\infty}dt \exp(-\lambda t) \int_{0}^{t}ds\langle \Delta I(g),X_{s}\rangle\\ &=-\frac{\lambda}{2}\int_{0}^{\infty}ds  \langle \Delta I(g),X_{s}\rangle\int_{s}^{\infty}dt \exp(-\lambda t)\\ &=-\frac{1}{2}\int_{0}^{\infty}dt \exp(-\lambda t) \langle \Delta I(g),X_{t}\rangle.
\end{align*}

Thus ${\gre \int_{0}^{\infty}dt \exp(-\lambda t)}\langle (\lambda +\Delta/2) I(g),X_{t}\rangle=0$ but {\gre (using  lemma \ref{ultraweakLemma})} $$I((\lambda+\Delta^{\o [n]}/2)^{-1}(g))=(\lambda+\Delta/2)^{-1}I(g),$$ thus applying the result above to  $g_{\lambda}=(\lambda+\Delta^{\o [n]}/2)^{-1}(g)$ gives 
$${\gre \int_{0}^{\infty}dt \exp(-\lambda t)\langle I(g),X_{t}\rangle=0}.$$ The result of the next inductive step  {\gre ($\langle I(g),X_{t}\rangle=0$) follows from uniqueness of Laplace transform}. Now by density of the functions of the form $I(\sum g_{n})$ (as in definition of ultraweak solutions) we get $X_{t}=0$.

\begin{step}
\textit{Assume $\Gamma_{0}(\omega)$. For any $\epsilon\in(0,1]$ and  $X_{0}=X_{0}^{\epsilon}\in L^2(M_0)$ there exists a unique mild solution (even in $\B_{2,\Delta^{1/2}}^a([0,T])$ for any T) to $X_t^{\epsilon}=\phi_t(X_0)+(1-\epsilon)\gamma(X^{\epsilon})_t$.}

\textit{Assume now $\Gamma_{1}(\omega,C)$ and $\delta=\tilde{\delta}$. For any $\epsilon\in(0,1]$ and  $X_{0}=X_{0}^{\epsilon}\in D(\Delta^{1/2})\cap L^{2}(M_0)$ there exists a unique strong solution (i.e. in $\B_{2,\Delta}^a([0,T])$ for any T) to $X_t^{\epsilon}=\phi_t(X_0)+(1-\epsilon)\gamma(X^{\epsilon})_t.$}
\end{step}
For each statement we can be content with proving for a small $T>0$ to be fixed later. Then, using the fact that $\Gamma_0,\Gamma_1$ are translation invariant in time, if we consider the same problem starting at $kT$, this gives the same {\gre result} on any $[0,T]$. 

The first statement is easy and a consequence of (\ref{crucial}) in  Proposition \ref{Cruc}.
If $Y\in \B_{2,\Delta^{1/2}}^{a}([0,T])$
, define an element 
at least in $L^{2}(M_t)\cap D(\Delta^{1/2})$ (for a.e. $t\in (0,T]$, by Proposition \ref{Cruc} and since $\B_{2,\Delta^{1/2}}^{a}([0,T])\hookrightarrow \B_{2,\delta}^{a}([0,T])$)~: 
$$\Gamma(Y)_{t}=\phi_{t}(X_{0}) + (1-\epsilon) \int_{0}^{t}\phi_{t-s}(\delta (Y_{s})\#dS_{s}).$$

First of all, $\Gamma(Y)$ is  in
$\B_{2,\Delta^{1/2}}^{a}$ for $Y$ in this space.
Indeed, first  $\phi_{t}(X_{0})$ is in this space, as a limit (coming from (\ref{basic})) of $\phi_{t}(\eta_{\alpha}(X_{0}))$,  continuous
function in $C^{0}([0,T],D(\Delta^{1/2}))\hookrightarrow\B_{2,\Delta^{1/2}}^{a}([0,T])$ (a usual $\Delta^{1/2} $-simple-process approximation
giving this). Second, since, if $Y_{n}$ is a $\Delta^{1/2} $-simple process
converging to $Y$, $\gamma(Y_{n})$ converge to $\gamma(Y)$ (a priori in
$L^{2}([0,T],D(\Delta^{1/2} ))$ from Proposition \ref{Cruc} (\ref{crucial})), {\red it suffices to note $\gamma(Y_{n})$ is itself 
 in $\B_{2,\Delta^{1/2}}^{a}([0,T])$}.

Finally it suffices to check $\Gamma$ is a contraction  (after moving to an equivalent norm) on 
$\B_{2,\Delta^{1/2}}^{a}([0,T]).$ 
Indeed note from proposition \ref{Cruc} and the definition:
 \begin{equation}\label{deltFo}\begin{split}&\int_0^Tds||\Delta^{1/2}(\Gamma(Y)_s-\Gamma(Z)_s)||_2^{2}\leq (1-\epsilon)^2\int_0^Tds||\delta(Y-Z)_s||_2^{2},
 \\&\int_0^Tds||\Gamma(Y)_s-\Gamma(Z)_s||_2^{2}\leq (1-\epsilon)^2T\int_0^Tds||\delta(Y-Z)_s||_2^{2}.\end{split}\end{equation}

Thus, fix $0<T<\epsilon/(2\max(1,\omega))$, so that one can take $K=\epsilon/2T $ to get $(1-\epsilon+KT)\max(1,\omega)<(1-\epsilon/2) K$ and define the equivalent norm on $\B_{2,\Delta^{1/2}}^{a}$ : 
$$||Y||_{\B_{2,\Delta^{1/2}}^{a},K}^2=\int_0^Tds||\Delta^{1/2}(Y)_s||_2^{2}+K||Y_s||_2^{2}.$$  We deduce from $\Gamma_{0}(\omega)$ $c.1$ and the previous inequalities that 
\begin{align*}||&\Gamma(Y)-\Gamma(Z)||_{\B_{2,\Delta^{1/2}}^{a},K}^2\\&\leq ((1-\epsilon)^2+ KT)\int_0^Tds(||\Delta^{1/2}(Y-Z)_s||_2^{2}+ \max(1,\omega)||(Y-Z)_s||_2^{2})\\&\leq (1-\epsilon/2)||Y-Z||_{\B_{2,\Delta^{1/2}}^{a},K}^2.
\end{align*}
This concludes to the first statement.


For the second statement, we want to show $\Gamma$ is a contraction on $\B_{2,\Delta}^{a}([0,T])$ after taking an equivalent norm again. We thus now take $Y\in \B_{2,\Delta}^{a}([0,T])$.




We can  apply Proposition \ref{Cruc} (\ref{crucialsupp}) to get : 
\begin{align*}\int_{0}^{T}dt||\Delta \eta_{\alpha}^{1/2}\eta_{\alpha'}^{1/2}\gamma(Y)_{t}||_{2}^{2}&= - ||\Delta ^{1/2}\eta_{\alpha}^{1/2}\eta_{\alpha'}^{1/2}(\gamma(Y)_{T})||_{2}^{2} +
\int_{0}^{T}dt||\Delta ^{\otimes
  1/2}\eta_{\alpha}^{\otimes1/2}\eta_{\alpha'}^{\o1/2}(\delta (Y_{t}))||_{2}^{2} 
\\&\leq \langle\Delta \eta_{\alpha}\int_{0}^{T}\delta (Y_{t})\#dS_{t},\int_{0}^{T}\delta (Y_{t})\#dS_{t}\rangle 
,\end{align*}
where we have used in the second line lemma \ref{deltaInt} (ii) and contractivity of $\eta_{\alpha'}^{1/2}$. But now, (in the case we assume $\Gamma_1$ and $\delta=\tilde{\delta}$), we can use lemma \ref{deltaInt} (i) and then $\Gamma_1$ $c.2$ and the bound in lemma \ref{Gamma1} for $\H_{\alpha}$ to get :
\begin{align*}\langle\Delta &\eta_{\alpha}\int_{0}^{T}\delta (Y_{t})\#dS_{t},\int_{0}^{T}\delta (Y_{t})\#dS_{t}\rangle 
\\ &=\langle\int_{0}^{T}\delta \circ\Delta (\eta_{\alpha}(Y_{s}))\#dS_{s}+\int_{0}^{T}\H_{\alpha}(Y_{s}\oplus\Delta ^{1/2}(Y_{s}))\#dS_{s},\int_{0}^{T}\delta (Y_{t})\#dS_{t}\rangle 
\\ &\leq \int_{0}^{T}||\Delta (Y_{t})||_{2}^{2}dt
+(\max(1,\omega)C+\omega)\int_{0}^{T}||Y_{t}||_{2}^{2}+||\Delta^{1/2} (Y_{t})||_{2}^{2}dt.\end{align*}

But better, we can write $||\Delta \eta_{\alpha}^{1/2}\eta_{\alpha'}^{1/2}\gamma(Y)_{t}||_{2}^{2}=\langle \Delta \eta_{\alpha}(\Delta )^{1/2}\eta_{\alpha'}^{1/2}\gamma(Y)_{t},(\Delta )^{1/2}\eta_{\alpha'}^{1/2}\gamma(Y)_{t}\rangle$
 to show that this increases to
 $||\Delta\eta_{\alpha'}^{1/2} \gamma(Y)_{t}||_{2}^{2}$ in $\alpha$ and then to $||\Delta\gamma(Y)_{t}||_{2}^{2}$  in $\alpha'$, with the inequality bellow and as a
 consequence (recall $C\geq 1$) $\gamma(Y)_{t}\in D(\Delta )$ a.e. and we got~:
\begin{align}\label{DeltFo}\int_{0}^{T}dt||\Delta \gamma(Y)_{t}||_{2}^{2}\leq \int_{0}^{T}||\Delta (Y_{t})||_{2}^{2}dt
+2\max(1,\omega)C\int_{0}^{T}||Y_{t}||_{2}^{2}+||\Delta^{1/2} (Y_{t})||_{2}^{2}\ dt.
\end{align}

Time has gone to choose $T$ small enough and introduce the equivalent norm on
$\B_{2,\Delta}^{a}([0,T])$ for which $\Gamma$ will be a contraction under the assumptions of (ii).

First choose $T$ such that $T\omega<1-(1-\epsilon)^{2}$ so that $T\omega<\frac{1}{(1-\epsilon)^{2}}-1$. Second, let $L>\frac{2C\max(1,\omega)(1+ T)}{1-(1-\epsilon)^{2}-\omega T}>0$, and $K=L\omega+2C\max(1,\omega)>0$ thus~:
\begin{align*}L>L\eta&:= L(1-\epsilon)^{2}+ (2C\max(1,\omega)(1+ T)+\omega TL)(1-\epsilon)^{2}\\&=L(1-\epsilon)^{2}+ (2C\max(1,\omega)+TK)(1-\epsilon)^{2}.\end{align*}
We get also : $$K>K\eta':=(1-\epsilon)^{2}K(1+T\omega)=(1-\epsilon)^{2}(L\omega + 2C\max(1,\omega)+KT\omega).$$

Finally, define the clearly equivalent norm~:
$||X||_{L,K,T}^{2}=\int_{0}^{T}(L||\Delta^{1/2}(X_{s})||_{L^{2}(\tau)}^{2}+K||X_{s}||_{L^{2}(\tau)}^{2}+||\Delta(X_{s})||_{L^{2}(\tau)}^{2})ds.$
We get, using (\ref{deltFo}) and (\ref{DeltFo}) in the first line, and then assumption $c.1$ in the second line~:
\begin{align*}(1&-\epsilon)^{2}||\gamma(Y)||_{L,K,T}^{2}\leq
  (1-\epsilon)^{2}\int_{0}^{T}||\Delta (Y_{t})||_{2}^{2}dt
  \\ &+(1-\epsilon)^{2}2C\max(1,\omega)\int_{0}^{T}||\Delta^{1/2}(Y_{t})||_{2}^{2}+||Y_{t}||_{2}^{2}dt+
  (L+KT)(1-\epsilon)^{2}\int_{0}^{T}||\delta (Y_{t})||_{2}^{2}dt \\ &\leq (1-\epsilon)^{2}\int_{0}^{T}||\Delta (Y_{t})||_{2}^{2}dt+
 L\eta \int_{0}^{T}||\Delta^{1/2}(Y_{t})||_{2}^{2}dt
  +
 K\eta'
 \int_{0}^{T}||Y_{t}||_{2}^{2}dt \\&\leq
  \max\left((1-\epsilon)^{2},\eta,\eta'
  \right) ||Y||_{L,K,T}^{2}.\end{align*}

First of all, this shows that $\Gamma(Y)$ is indeed in
$\B_{2,\Delta}^{a}$ for $Y$ in this space :
first, since  $\phi_{t}(X_{0})$ is in this space as before 
 and second, since, if $Y_{n}$ is a $\Delta $-simple process
converging to $Y$, $\gamma(Y_{n})$ converge to $\gamma(Y)$ (a priori in
$L^{2}([0,T],D(\Delta ))$, {\red and $\gamma(Y_{n})$ is itself 
 in $\B_{2,\Delta}^{a}$}.

Then, we can say that $\Gamma$ is a contraction on
$\B_{2,\Delta}^{a}([0,T])$ equipped of the norm
$||.||_{L,K,T}$, this concludes.

\begin{step}
Conclusion of the proof of (i).
\end{step}

Applying orthogonality (via lemma \ref{deltaInt} (iii)) and equation 
(\ref{crucial}) {\gre to} $(1-\epsilon)\gamma(X^{\epsilon})_{t}=X_{t}^{\epsilon}-\phi_{t}(X_{0})$,  we know that for any $T$ :

\begin{align}\begin{split}\label{eqstep2}\int_{0}^{T}dt||\Delta ^{1/2}X_{t}^{\epsilon}||_{2}^{2}&= \int_{0}^{T}dt||\Delta ^{1/2}\phi_{t}(X_{0})||_{2}^{2}- (1-\epsilon)^{2}||(\gamma(X^{\epsilon})_{T})||_{2}^{2} +
(1-\epsilon)^{2}\int_{0}^{T}dt||(\delta (X_{t}^{\epsilon}))||_{2}^{2}.\end{split}
\end{align}

Using  equation (\ref{basic}) and orthogonality and then assumption $c.1$  we deduce :

\begin{align*}&(1-\epsilon)^{2}||\gamma(X^{\epsilon})_{t}||_{2}^{2}\leq ||X_{0}||_{2}^{2}-||\phi_{t}(X_{0})||_{2}^{2}+\omega\int_{0}^{t}||X_{s}^{\epsilon}||_{2}^{2}ds.
\\ &||X_{t}^{\epsilon}||_{2}^{2}=||X_{0}^{\epsilon}||_{2}^{2}+
(1-\epsilon)^{2}||\int_{0}^{t}\delta (X_{s}^{\epsilon})\#dS_{s}||_{2}^{2}
- \int_{0}^{t}||\Delta^{1/2}(X_{s}^{\epsilon})||_{2}^{2}ds\leq ||X_{0}^{\epsilon}||_{2}^{2}+
\omega \int_{0}^{t}||X_{s}^{\epsilon}||_{2}^{2}ds.\end{align*}
Note this second inequality works for $\epsilon=0$ as soon as we have a solution in this case.
 We can use Gronwall's lemma on this second inequality. It proves the first inequality of the theorem (for $X^{\epsilon}$ instead of $X$). Combining this with the first inequality, we get, after integration, the second inequality in part (i), 
showing that $X^{\epsilon}$ is bounded in $\B_{2,\phi\delta}^{a}$.
\begin{align*}||X^\epsilon||_{\B_{2,\phi\delta}^{a}}&=\left(\int_{0}^{T}\int_{0}^{t}||\phi_{t-s}^{\o}\delta(X_{s})||_{L^{2}(\tau\otimes\tau)^{\bigoplus \N}}^{2}ds+||X_{t}||_{L^{2}(\tau)}^{2}
  dt\right)^{1/2}\\ &\leq\frac{1}{(1-\epsilon)}\left(\int_{0}^{T}\left(||X_{0}||_{L^{2}(\tau)}^{2}+\omega\int_0^tdse^{\omega s}||X_{0}||_{L^{2}(\tau)}^{2}\right)+e^{\omega t}||X_{0}||_{L^{2}(\tau)}^{2}
  dt\right)^{1/2}\\ &\leq\frac{1}{(1-\epsilon)}\left(\int_{0}^{T}2e^{\omega t}||X_{0}||_{L^{2}(\tau)}^{2}
  dt\right)^{1/2}.\end{align*}

 Modulo extraction, we get a *-weak limit in $\B_{2,\phi\delta}^{a}([0,T])$ 
  by compactness. As a consequence, since $\gamma$ is a linear continuous map as recalled in the part on stochastic convolution, $\gamma(X^{\epsilon})$ (or at least the image of the previous extraction) converges in $L^{2}([0,T],L^{2}(M))$ weakly. Since $\phi_{t}(X_{0})$ is a constant in this space we can take the limit and verify the equation in this space, thus a.e., we especially get an ultramild solution.  Since we deduce any such *-weak limit point is also an ultraweak solution (since $X^{\epsilon}$ is a mild thus weak thus ultraweak solution of the $\epsilon$ variant) we get *-weak convergence from uniqueness proved in step 0.

Moreover, taking $\xi\in L^{2}(M)$, with, say, the projection of $\xi$ on the space of stochastic integrals given by $\int_{0}^{T}\eta_{s}\#dS_{s}$, let us prove that $\langle\xi, X_{t}^{\epsilon}\rangle$ is an equicontinuous and uniformly bounded family (for $\epsilon\in (0,1]$) on [0,T]. From what we obtained above, only equicontinuity need to be proved, but (for $t\leq \tau$) we have (using the equation for $X_{\tau}^{\epsilon}$ and Cauchy-Schwarz):
\begin{align*}\langle\xi, &X_{\tau}^{\epsilon}- X_{t}^{\epsilon}\rangle\leq ||\xi||_{2} ||\phi_{\tau-t}(X_{0})-X_{0}||_{2}\\ &+(1-\epsilon)\int_{t}^{\tau}ds\langle\eta_{s},\phi_{\tau-s}^{\o}(\delta(X_{s}^{\epsilon}))\rangle + (1-\epsilon)\int_{0}^{t}ds\langle\phi_{\tau-t}^{\o}\eta_{s}-\eta_{s},\phi_{t-s}^{\o}(\delta(X_{s}^{\epsilon}))\rangle\\ &\leq ||\xi||_{2} ||\phi_{\tau-t}(X_{0})-X_{0}||_{2}+(\int_{t}^{\tau}ds||\eta_{s}||_{2}^{2})^{1/2}||(1-\epsilon)\gamma_{\tau}(X^{\epsilon})||_{2} \\ &+ (\int_{0}^{t}||\phi_{\tau-t}^{\o}\eta_{s}-\eta_{s}||_{2}^{2})^{1/2}||(1-\epsilon)\gamma_{t}(X^{\epsilon})||_{2}\\ &\leq ||\xi||_{2} ||\phi_{\tau-t}(X_{0})-X_{0}||_{2}+e^{\omega \tau/2}||X_{0}||_{2}\left((\int_{t}^{\tau}ds||\eta_{s}||_{2}^{2})^{1/2} + ||\phi_{\tau-t}\xi-\xi||_{2}\right).\end{align*}

This concludes using strong continuity of $\phi_{t}$ (and using Heine-Cantor Theorem). As a consequence, using Arzela-Ascoli Theorem (and separability assumption on $L^{2}(M)$), we get via diagonal extraction, $X_{t}$ is weakly continuous, and limit of a subsequence of $X_{t}^{\epsilon}$ in $C^{0}([0,T],(L^{2}(M),\sigma(L^{2}(M),L^{2}(M))))$. As a consequence, this easily enables us to pass to the limit $\epsilon\rightarrow 0$ in the first inequality of the theorem. From this we get also that any limit point is an ultraweak solution, so that from uniqueness we get the stated limit without extraction.

We now establish the supplementary inequality.

First by orthogonality and assumption $d.2$ of $\Gamma_0(\omega)$, we have~:
\begin{align*}&
||X_{T}^{\epsilon}-X_{0}-(1-\epsilon)\delta (X_{0})\#S_{T}||_{2}^{2}\\&=||\phi_{T}(X_{0})-X_{0}||_{2}^{2}+ (1-\epsilon)^{2}\int_{0}^{T}dt||\phi_{T-t}^{\o}\delta(\phi_t(X_{0}))-\delta(X_{0})||_{2}^{2}+(1-\epsilon)^4||\gamma\circ\gamma(X^{\epsilon})_T||_2^2.
\end{align*}

Morover,  the same kind of orthogonality and relations (\ref{crucial}) and (\ref{basic}) imply that :
\begin{align*}(&1-\epsilon)^4||\gamma\circ\gamma(X^{\epsilon})_T||_2^2=
\int_{0}^{T}dt((1-\epsilon)^2||\delta(X^{\epsilon}_t)||_2^2-||\Delta^{1/2}(X^{\epsilon}_t)||_2^2)\\ &+\int_{0}^{T}dt(||\Delta^{1/2}(\phi_t(X_0))||_2^2-(1-\epsilon)^2||\delta(\phi_t(X_0))||_2^2)+(1-\epsilon)^2||\Delta^{1/2}(\gamma(\phi_{.}(X_0))_t)||_2^2
\\& \leq (e^{\omega T}-1)||X_0||_2^2 + T \sup_{[0,T]} \left(||\Delta^{1/2}(\phi_{t}X_{0})||_{2}^{2}-(1-\epsilon)^2||\delta(\phi_{t}X_{0})||_{2}^{2}\right)\\ & +(1-\epsilon)^2\int_{0}^{T}dt||\delta(\phi_t(X_0))||_2^2-||\phi_{T-t}^{\otimes}\delta(\phi_t(X_0))||_2^2,
\end{align*}

where we used, in the inequality, the first inequality  of our theorem.
 
 Our first line is in our estimate (once added a $(1-\epsilon)$ where needed for our $\epsilon$ variant). It only remains to get the other terms by several elementary computations, only involving $X_0$.
 \begin{align*}\int_{0}^{T}&dt||\phi_{T-t}^{\o}\delta(\phi_t(X_{0}))-\delta(X_{0})||_{2}^{2}+ ||\delta(\phi_t(X_0))||_2^2-||\phi_{T-t}^{\otimes}\delta(\phi_t(X_0))||_2^2\\&=\int_{0}^{T}dt||\delta(\phi_{t}X_{0}-X_{0})||_{2}^{2}+
2\Re\int_{0}^{T}dt\  \langle\delta(X_{0})-\phi_{T-t}^{\o}\delta(X_{0}),\delta(X_{0})\rangle\\&+2\Re\int_{0}^{T}dt\ \langle(\phi_{T-t}^{\o}-id)\delta(X_{0}),\delta(X_{0}-\phi_{t}(X_{0}))\rangle\\ & \leq \int_{0}^{T}dt\ ||\Delta^{1/2}(\phi_{t}X_{0}-X_{0})||_{2}^{2}+T^2\omega/4||\Delta^{1/2}(X_0)||_2^2\\ &+\int_{0}^{T}dt\ (T-t)||\Delta^{\o1/2}(\delta(X_{0}))||_{2}^{2}\\ &+\int_{0}^{T}dt\ \sqrt{t(T-t)}(||\Delta(X_{0})||_{2}^2+\omega||\Delta^{1/2}(X_{0})||_{2}^2)^{1/2}||\Delta^{\o1/2}(\delta(X_{0}))||_{2}
\\ &= \int_{0}^{T}dt\ ||\Delta^{1/2}(\phi_{t}X_{0}-X_{0})||_{2}^{2}+T^2\omega/4||\Delta^{1/2}(X_0)||_2^2\\ & +\frac{T^{2}}{2}\left( ||\Delta^{\o1/2}(\delta(X_{0}))||_{2}^{2}+\frac{\pi}{4}(||\Delta(X_{0})||_{2}^2+\omega||\Delta^{1/2}(X_{0})||_{2}^2)^{1/2}||\Delta^{\o1/2}(\delta(X_{0}))||_{2}\right).
\end{align*}

The inequality comes from $\Gamma_0(\omega) c.1$ and several uses of the spectral theorem applied in the form $\langle(id-\phi_{t})^{i}x,x\rangle\leq \langle\frac{t}{2}\Delta x,x\rangle$ ($i=1$ or $2$).

Finally, it remains to compute the last integral using the spectral theorem for $\Delta$ :
\begin{align*}\int_{0}^{T}dt||\Delta^{1/2}(\phi_{t}(X_{0})-X_{0})||_{2}^{2}&=4||\phi_{T/2}(X_{0})||_{2}^{2}-3||X_{0}||^{2}-||\phi_{T}(X_{0})||_{2}^{2}+T||\Delta^{1/2}(X_{0})||_{2}^{2}\\ &=2\langle(\phi_{T}-id+T\Delta/2)(X_{0}),X_{0}\rangle-||\phi_{T}(X_{0})-X_{0}||_{2}^{2}\\ &\leq \frac{T^{2}}{4}||\Delta(X_{0})||_{2}^{2}-||\phi_{T}(X_{0})-X_{0}||_{2}^{2},\end{align*}

Putting everything together this concludes to : 
\begin{align*}
||X_{t}&^{\epsilon}-X_{0}-(1-\epsilon)\delta (X_{0})\#S_{t}||_{2}^{2} \\&\leq \frac{t^{2}}{4}||\Delta(X_{0})||_{2}^{2}+(e^{\omega t}-1)||X_{0}||_{2}^{2}+(1-(1-\epsilon)^2)||\phi_{t}(X_{0})-X_{0}||_{2}^{2}\\ &+\frac{t^{2}}{2}\left( ||\Delta^{\o1/2}(\delta(X_{0}))||_{2}^{2}+\frac{\pi}{4}(||\Delta(X_{0})||_{2}^2+\omega||\Delta^{1/2}(X_{0})||_{2}^2)^{1/2}||\Delta^{\o1/2}(\delta(X_{0}))||_{2}\right)\\ &+t \sup_{[0,t]} \left(||\Delta^{1/2}(\phi_{s}X_{0})||_{2}^{2}-(1-\epsilon)^{2}||\delta(\phi_{s}X_{0})||_{2}^{2}\right)+t^2\omega/4||\Delta^{1/2}(X_0)||_2^2.
\end{align*}

\n We easily obtain the limit case $\epsilon=0$ using the limit in $C^{0}([0,T],(L^{2}(M),\sigma(L^{2}(M),L^{2}(M))))$.

\begin{step}
Under the assumptions of (ii), with $\B_{2,\delta }^{a}$ depending on our fixed $T>0$, and
for $X_{0}\in D(\Delta)$, there
exists a unique mild solution $X_{t}$ of (\ref{SPDES}) which is the weak
limit in $\B_{2,\delta }^{a}$ and strong limit in $\B([0,T],L^{2}(M))$ of the
solution $X_{t}^{\epsilon}$ of step one. Moreover, this solution satisfies
the two first inequalities of (ii) in the theorem. 
\end{step}
Consider $\epsilon>0$ {\gre like } in step 1.
In case $\delta\neq\tilde{\delta}$, we don't know $X_{t}^{\epsilon}\in D(\Delta)$, since we have only a mild solution, we have to circumvent this trouble for computational purposes.

Applying the first part of step 1 with $\delta$ replaced by $\eta_{\beta}^{\o}\delta$, we get a solution $X_{t}^{\epsilon,\beta}$ in $\B_{2,\Delta^{1/2} }^{a}$
 and since by proposition \ref{Cruc} (\ref{crucial}) and the argument in step one, $\gamma(X_{t}^{\epsilon,\beta})\in \B_{2,\Delta^{1/2} }^{a}$ we deduce $\eta_{\beta}\gamma(X_{t}^{\epsilon,\beta})\in \B_{2,\Delta^{3/2} }^{a}$. As a consequence, if $X_0\in D(\Delta)$ we get as in step 1, $X_{t}^{\epsilon,\beta}\in \B_{2,\Delta^{3/2} }^{a}$.



We can now compute for our solution $X_{t}^{\epsilon,\beta}$. We can apply Proposition \ref{Cruc} (\ref{crucialsupp}) in case $B=\tilde{\delta}$, $\alpha=\alpha'=\beta$ and the variant of (\ref{basic}) valid for $x=X_0\in D(\Delta^{1/2})$ : $||\tilde{\delta}\phi_t(x)||_2^2=||\tilde{\delta}x||_2^2-\int_0^t\Re\langle\tilde{\delta}\Delta\phi_{s}(x),\tilde{\delta}\phi_{s}(x)\rangle ds.$ Using also orthogonality from lemma \ref{deltaInt} but for $\tilde{\delta}$, we get~:
\begin{align*}||&\tilde{\delta}
(X_{t}^{\epsilon,\beta})||_{2}^{2}=||\tilde{\delta}
(X_{0}^{\epsilon})||_{2}^{2}+
(1-\epsilon)^{2}||\tilde{\delta}
\eta_{\beta}\int_{0}^{t}\delta (X_{s}^{\epsilon,\beta})\#dS_{s}||_{2}^{2}
- \int_{0}^{t}\Re\langle\tilde{\delta}\Delta 
X_{s}^{\epsilon,\beta},\tilde{\delta} 
X_{s}^{\epsilon,\beta}\rangle ds.\end{align*}


We have thus shown :
\begin{align*}&||\tilde{\delta}(X_{t}^{\epsilon,\beta})||_{2}^{2}=||\tilde{\delta}(X_{0}^{\epsilon})||_{2}^{2}\\ &+
(1-\epsilon)^{2}||(\eta_{\beta}^{\o}\tilde{\delta}+\tilde{\H}_{\beta})\int_{0}^{t}\delta (X_{s}^{\epsilon,\beta})\#dS_{s}||_{2}^{2}
- \int_{0}^{t}\Re\langle\Delta ^{\otimes}\circ\tilde{\delta}(X_{s}^{\epsilon,\beta})-\H(\tilde{\delta} ( X_{s}^{\epsilon,\beta})),\tilde{\delta} X_{s}^{\epsilon,\beta}\rangle ds.\end{align*} We used the identities of assumption $f.2$ and of lemma \ref{Gamma1} about $\tilde{\delta}\Delta$ and $\tilde{\delta}\eta_{\beta}$ justified since almost surely in $s$ $X_{s}^{\epsilon,\beta}\in D(\Delta^{3/2})$ and because via lemma \ref{deltaInt} (i) we know $\int_{0}^{t}\delta (X_{s}^{\epsilon,\beta})\#dS_{s}\in D(\tilde{\delta})$. We deduce :
\begin{align*} ||&\tilde{\delta}(X_{t}^{\epsilon,\beta})||_{2}^{2}\\ &\leq ||\tilde{\delta}(X_{0}^{\epsilon})||_{2}^{2}+
(1-\epsilon)^{2}\int_{0}^{t}||\tilde{\delta}^{\o}\delta (X_{s}^{\epsilon,\beta})||_{2}^{2}ds+ 2\Re \langle \tilde{\delta}\eta_{\beta}\int_{0}^{t}\delta (X_{s}^{\epsilon,\beta})\#dS_{s}, \tilde{\H}_{\beta}\int_{0}^{t}\delta (X_{s}^{\epsilon,\beta})\#dS_{s}\rangle\\ &
- \int_{0}^{t}\Re\langle\Delta ^{\otimes}\circ\tilde{\delta}(X_{s}^{\epsilon,\beta})-\H(\tilde{\delta} ( X_{s}^{\epsilon,\beta})),\tilde{\delta} X_{s}^{\epsilon,\beta}\rangle ds\\ &\leq ||\tilde{\delta}(X_{0}^{\epsilon})||_{2}^{2}+
\int_{0}^{t} 2C ||\tilde{\delta}(X_{s}^{\epsilon,\beta})||_2^2ds+ 2\Re \langle \tilde{\delta}\eta_{\beta}\int_{0}^{t}\delta (X_{s}^{\epsilon,\beta})\#dS_{s}, \tilde{\H}_{\beta}\int_{0}^{t}\delta (X_{s}^{\epsilon,\beta})\#dS_{s}\rangle
\\ &\leq ||\tilde{\delta}(X_{0}^{\epsilon})||_{2}^{2}+
\int_{0}^{t} (2C+2\frac{C^4}{\beta}(\omega+2\beta)) ||\tilde{\delta}(X_{s}^{\epsilon,\beta})||_2^2ds
.\end{align*}
In the first line, we used $\eta_{\beta}$ contractive after computing the first scalar product. 
In the second line, we used assumption $h$ to cancel one term and the bound $||\H||\leq C$. In the last line, we used  assumption $g$, the bound on $\tilde{\H}_{\beta}$ from 
lemma \ref{Gamma1} and $||\tilde{\delta}\eta_{\beta}||\leq C\sqrt{\omega+2\beta}$ already used there.

Applying Gronwall's lemma, we got (for $\beta\geq 1$):
$$||\tilde{\delta}(X_{t}^{\epsilon,\beta})||_{2}^{2}\leq ||\tilde{\delta}(X_{0}^{\epsilon,\beta})||_{2}^{2}e^{(6+2\omega)C^4t}.$$

As a consequence, we get a weak limit point $X_{t}^{\epsilon,\infty}$ in $\B_{2,\delta }^{a}$. 
Let us show such a limit point is a solution of (\ref{SPDES}$\epsilon$) in $\B_{2,\delta }^{a}$. 
This gives by uniqueness 
$X_{t}^{\epsilon,\infty}=X_{t}^{\epsilon}$, and the fact that the weak limit point is a limit. Of course, it suffices to show the  equation  weakly, the only non trivial limit is the stochastic integral, but since $\delta X_{t}^{\epsilon,\beta}$ is bounded it is easy to remove $\eta_\beta$ on the other side of the scalar product, and then to use weak convergence of $X_{t}^{\epsilon,\beta}$ in $\B_{2,\delta }^{a}$.
We also get a corresponding inequality a.e. for the limit by seeing the inequality weakly in $L^2([0,T])$.

As is usual, if we are able to prove bounds in $D(\delta)$, we can also deduce $||.||_2$ Cauchy property.
Using (\ref{crucial}) 
after using the SDE and the common initial conditions,  we also get (for $0<\epsilon,\eta<1$)~:
\begin{align*}&||X_{t}^{\epsilon}-X_{t}^{\eta}||_{2}^{2}=||\gamma((1-\epsilon)X_{t}^{\epsilon}-(1-\eta)X_{t}^{\eta})||_{2}^{2}\\
  &=\
 -\int_{0}^{t}||\Delta ^{1/2}\gamma((1-\epsilon)X_{s}^{\epsilon}-(1-\eta)X_{s}^{\eta})||_{2}^{2}ds +
\int_{0}^{t}||\delta ((1-\epsilon)X_{s}^{\epsilon}-(1-\eta)X_{s}^{\eta})||_{2}^{2}ds\\
&\leq \
 -\int_{0}^{t}||\Delta ^{1/2}(X_{s}^{\epsilon}-X_{s}^{\eta})||_{2}^{2}ds +
\int_{0}^{t}||\delta(X_{s}^{\epsilon}-X_{s}^{\eta})||_{2}^{2} \\&+12\max(\epsilon, \eta)\max(||\delta(X_{s}^{\epsilon})||_{2},||\delta (X_{s}^{\eta})||_{2})^{2}ds.
\end{align*}
In the last line, we used an elementary bound on the second integral expanding the scalar products with  $(1-\epsilon)X_{s}^{\epsilon}-(1-\eta)X_{s}^{\eta}=(X_{s}^{\epsilon}-X_{s}^{\eta}) +(\eta X_{s}^{\eta}-\epsilon X_{s}^{\epsilon})$ and again the SDE with same initial condition on the first integral.
Using assumption $c.1,g$ and our bound on $||\delta(X_{s}^{\epsilon})||_{2}$, one gets :
\begin{align*}&||X_{t}^{\epsilon}-X_{t}^{\eta}||_{2}^{2}\leq \
 \int_{0}^{t}\omega||(X_{s}^{\epsilon}-X_{s}^{\eta})||_{2}^{2} +12\max(\epsilon, \eta)C^2||\tilde{\delta}(X_{0}^{\epsilon,\beta})||_{2}^{2}e^{(6+2\omega)C^4s}ds\\ &\leq \
 12\max(\epsilon, \eta)C^2||\tilde{\delta}(X_{0}^{\epsilon,\beta})||_{2}^{2}\frac{e^{(6+2\omega)C^4t+\omega t}}{(6+2\omega)C^4}.
\end{align*}

As noted at the beginning of step {\gre 2}, 
we know any (mild) solution of the case $\epsilon=0$, if it
exists satisfies~: $||X_{t}||_{2}^{2}\leq e^{\omega t}||X_{0}||_{2}^{2},$ giving especially
uniqueness.

We have thus obtained strong convergence on $X_{t}^{\epsilon}$ in $\B([0,T],L^{2}(M))$ by Cauchy property. We have also boundedness of $X_{t}^{\epsilon}$ in $B_{2,\delta}^{a}$, which gives by weak compactness a limit up to
extraction when $\epsilon\rightarrow 0$. Once we will have proved that any
such limit point is a mild solution with $\epsilon=0$, uniqueness (of the
solution thus of the limit point) will get
that in fact $X_{t}^{\epsilon}$ weakly converges in
$\B_{2,\delta }^{a}$ to the newly found solution $X_{t}$. Since
we have already noticed weak continuity of Stochastic convolution, we are
in fact done for proving that any limit point is a mild solution.


Finally, we conclude the proof of the part (ii) of our theorem, by considering $\eta_{\alpha}(X_{0})$ as initial condition of a solution $X_{t,\alpha}$, in case we have only $X_{0}\in D(\Delta^{1/2})$  (and not anymore $D(\Delta)$) and letting go $\alpha\rightarrow\infty$. With the same weak limit arguments, we show that $X_{t,\alpha}$ converges weakly in $\B_{2,\delta }^{a}$ to $X_{t}$. Moreover, note for further use that we have also strong convergence of $X_{t,\alpha}$ to $X_{t}$ in $\B([0,T],L^{2}(M))$ by the following inequality (proved as above for the Cauchy property, except we don't have the same initial conditions anymore, but , however, more cancellations)~:
\begin{align*}||&X_{t,\alpha}-X_{t,\beta}||_{2}^{2}=||\phi_{t}(X_{0,\alpha}-X_{0,\beta})||_{2}^{2}+||\gamma(X_{t,\alpha}-X_{t,\beta})||_{2}^{2}\\
  &=||\phi_{t}(X_{0,\alpha}-X_{0,\beta})||_{2}^{2}\
 -\int_{0}^{t}||\Delta ^{1/2}\gamma(X_{s,\alpha}-X_{s,\beta})||_{2}^{2}ds +
\int_{0}^{t}||\delta (X_{s,\alpha}-X_{s,\beta})||_{2}^{2}ds\\
&\leq ||X_{0,\alpha}-X_{0,\beta}||_{2}^{2}\
 -\int_{0}^{t}||\Delta ^{1/2}(X_{s,\alpha}-X_{s,\beta})||_{2}^{2}-||\Delta ^{1/2}\phi_{s}(X_{0,\alpha}-X_{0,\beta})||_{2}^{2}ds \\ &+
\int_{0}^{t}||\Delta^{1/2}(X_{s,\alpha}-X_{s,\beta}))||_{2}^{2}+\omega||X_{s,\alpha}-X_{s,\beta}||_{2}^{2}ds
\\&\leq e^{\omega t}\left(||\eta_{\alpha}(X_{0})-\eta_{\beta}(X_{0})||_{2}^{2}\
 +T||\eta_{\alpha}(\Delta ^{1/2}X_{0})-\eta_{\beta}(\Delta ^{1/2}X_{0})||_{2}^{2}\right).\ \ \ \ \ \ \ \ \ \ \ \ \ \ \ \ \ \ \ \ \ensuremath{\Box}
\end{align*}\end{proof2}

\section{Our Main example : Derivation-Generator of a Dirichlet form}

As explained in the introduction, our main case of interest will be when $\delta$ is a derivation and $\Delta=\delta^*\delta$ the corresponding generator of a Dirichlet form. Note that in that case it is well known (cf e.g. \cite{CiS}) $\phi_t$ and $\eta_{\alpha}$ are completely positive contractions on $M$. 

\subsection{Preliminaries and notation around zero extensions of a derivation on free Brownian motions}

\subsubsection{Setting and extension}
Recall $M=W^{*}(M_{0};S_{s}^{(j)}, 0\leq s\leq \infty, 0\leq j\leq N)$ (we will consider only here the case of finitely many derivations and thus free Brownian motions) and $M_{t}=W^{*}(M_{0};S_{s}^{(j)}, 0\leq s\leq t, 0\leq j\leq N)$.

Let us assume we are given $\partial:
D(\partial)\rightarrow 
 HS(M_{0})^{N}\simeq (L^2(M_0)\o L^2(M_0))^N\simeq_{1\o O} (L^2(M_0)\o L^2(M_0^{op}))^N$ a derivation 
valued in a direct sum of Hilbert-Schmidt operators over $L^2(M_0)$. As usual the identification of $L^2(M_0)\o L^2(M_0)$ with Hilbert-Schmidt operators sends $a\o b$ to the finite rank operator $x\mapsto a\tau(bx)$. As real bimodules, they are considered with bimodule structure induced by $a(b\o c)d=ab\o cd$, and real structure $\J(a\o b)=b^*\o a^*$ corresponding to adjointness of Hilbert-Schmidt operators. We will emphasize the isomorphism $1\o O$ with $L^2(M_0)\o L^2(M_0^{op})$ (coming from traciality) with corresponding bimodule structure when necessary (it is induced by the identity map for $a,b\in M$ $(1\o O)(a\o b)=a\o b$ with $b$ seen in $M^{op}$).


We write $Z_{j}=(0,...,0,1\otimes 1,0,...,0)$ in $HS(M_{0})^{N}$ the non-zero term
lying on the $j^{th}$ component
. We also write $\partial_{j}$ for the $j^{th}$
component in $HS(M_{0})^{N}$ (and we will use freely later this kind of notation).
For $U\in L^{2}(M)\o L^2(M^{op})$, $K\in M\overline{\o } M^{op}$, we write consistently with {\red our} previous notation $U\# K$ the map induced by multiplication in $M\overline{\o } M^{op}$. If $U\in L^{2}(M)\o L^2(M)$, we write in this way the map induced by the previous isomorphism : $(1\o O)(U\# K):=((1\o O)(U))\# K$.

  Domains of closures will be considered in this $L^2$ setting,
 $D(\partial)\subset M_{0}$ is a weakly dense *-subalgebra. We will really soon impose conditions making $\partial$ closable as an
unbounded operator from $L^{2}(M_{0},\tau)\rightarrow HS(M_{0})^{N}$, and real
(i.e. we have the relation $\partial(x)^*=\partial(x^*)$ with the adjoint of Hilbert-Schmidt operators in each component and as a consequence $\langle \partial(x),y\partial(z)\rangle= \langle
\partial(z^{*})y^{*},\partial(x^{*})\rangle, \forall x,y,z\in D(\partial)$).
After extending it to a closed derivation $\bar{\delta}$ on $M$ we will be interested in the corresponding generator of a Dirichlet form $\Delta=\delta^{*}\bar{\delta}$. This part will find realistic assumptions on $\partial$ to get $\Gamma_{1u}(\omega,C)$ 
 and thus to be able to apply our general theory.

Suppose also that $\J_{j}:=\partial^{*}(Z_{j})\in L^{2}(M_{0})$ is well
defined for all $j \in [1,N]$. We have a well-known lemma (identical to 
Proposition 4.1 in \cite{Vo5} which is valid for any real derivation of the kind
considered above, as pointed
out after Proposition 6.2 in \cite{Vo6})~:

\begin{lemmas}\label{triv}
Consider $\partial$ real densely defined derivation with $\J_{j}:=\partial^{*}(Z_{j})\in L^{2}(M_{0})$, then $(D(\partial)\otimes_{alg} D(\partial))^{N}$ is
contained in $D(\partial^{*})$ {\gre (as a consequence assumption $e.1$ is satisfied)} and~:
$$\partial_{j}^{*}(a\otimes b):= \partial^{*}(aZ_{j}b)= a\J_{j}b -
(1\otimes\tau)[\partial_{j}(a)]b-a(\tau\otimes 1)[\partial_{j}(b)].$$
Moreover (see e.g. \cite{nonGam}
 Remark 7, using mainly \cite{DL}, 
 ), $\bar{\partial}|_{M_0\cap D(\bar{\partial})}$ defines a derivation (noted $\partial^{\infty}$ on the $*$-algebra  $M_0\cap D(\bar{\partial})$), closed as an unbounded operator $M_0\rightarrow HS(M_0)^{N}$. Finally (see e.g. proposition 6 in 
 \cite{nonGam}
), for any $Z\in D(\bar{\partial})\cap M_{0}$, there exists a sequence $Z_{n}\in D(\partial)$ with $||Z_{n}||\leq ||Z||$, $||Z_{n}-Z||_{2},||\partial(Z_{n})-\bar{\partial}(Z)||_{2}\rightarrow 0$.
\end{lemmas}

 Consider also $D(\delta)=D(\partial)*\C\langle S_{s}^{(j)}
, 0\leq j\leq N, 0\leq s\leq \infty \rangle \subset M$, the algebra generated by
$S_{s}^{(j)}$ and $D(\partial)$ (thus $D(\delta )$ is a weakly dense
*-subalgebra of $M$). Define ${\red \delta :D(\delta)
\rightarrow HS(L^{2}(M))^{N}}$ the
unique derivation such that $\delta (x)=\partial(x)$ if $x\in D(\partial)$
and $\delta(S_{t}^{(j)})=0$ for all $t$. 
Then, clearly $\J_{j}=\delta^{*}(Z_{j})\in L^{2}(M_{0})\subset L^{2}(M)$ (see e.g. \cite[Example 2.4]{Shly03}), and
using the lemma above, $\delta$ is also closable (since $\delta ^{*}$ is
densely defined).  $\delta$ is thus a closable real derivation, like
$\partial$, {\gre satisfying $e.1$}. 
 We may sometimes write $\delta^{\infty}:M\cap D(\bar{\delta})\rightarrow HS(L^{2}(M))^{N}$ the analog derivation defined in the previous lemma (when we want to emphasize the domain). 
We will write $\Delta=\delta^{*}\bar{\delta}$ the associated generator of a
completely Dirichlet form, $\phi_{t}$ the semigroup
generated by
$-1/2 \Delta$, $\eta_{\alpha}=\frac{\alpha}{\alpha+\Delta}$ the ``resolvent map" associated, as before.  As we already pointed out, {\gre $\eta_{\alpha}$ and $\phi_{t}$} {\red induce} completely positive contraction{\gre s} on $M$.

We thus only assumed {\gre in this section } assumption $0$ :

\begin{minipage}{15,5cm}\textbf{Assumption 0}\textit{~:(a) $\partial:
D(\partial)\rightarrow  HS(M_{0})^{N}$ real derivation $D(\partial)\subset M_0$ weakly dense *-subalgebra}

\textit{(b) $\J_{j}:=\partial^{*}(Z_{j})\in L^{2}(M_{0})$ is well
defined for all $j \in [1,N]$, and $\delta$ is an extension by 0 on free Brownian motions : $\delta (x)=\partial(x)$ if $x\in D(\partial)$
and $\delta(S_{t}^{(j)})=0$ for all $t$.}
 \end{minipage}
 
  This subsection will mainly develop general consequences of this assumption 0, giving at the end $\Gamma_{0u}$.

\subsubsection{Useful $L^1$-closures}
Here we assume assumption $0$.

We will also define following \cite{Pe04}  1.4, an analog of $\Delta$, $\Delta^{1}:M\rightarrow L^{1}(M,\tau)$ (there noted $\Psi$), by
 \begin{align}\label{Delta1def}\begin{split}D(\Delta^{1})=\{ x\in& D(\bar{\delta})\cap M\  | \ y\mapsto\langle\bar{\delta}(x),\bar{\delta}(y)\rangle\\ &\text{extends to a normal linear functional on $M$}\}\end{split}.\end{align} 
  $\Delta^{1}(x)$ is defined as the adjoint of the Radon-Nikodym derivative of the preceding linear functional $y\mapsto\langle\bar{\delta}(x),\bar{\delta}(y)\rangle$, i.e. $\langle\Delta^{1}(x),y\rangle:=\tau(\Delta^{1}(x)^*y)=\langle\bar{\delta}(x),\bar{\delta}(y)\rangle$ (note the anti-linear duality bracket consistent with scalar products).
  
   Likewise, we can define $\delta^{*1}:(L^{2}(M)\o L^{2}(M))^{N}\rightarrow L^{1}(M,\tau)$, by \begin{align*}D(\delta^{*1})=\{ U\in &(L^{2}(M)\o L^{2}(M))^{N}\  | \ y\mapsto\langle U,\bar{\delta}(y)\rangle\\ &\text{extends to a normal linear functional on $M$}\}.\end{align*} $\delta^{*1}(U)$ is defined as the adjoint of the Radon-Nikodym derivative of the preceding linear functional  $y\mapsto\langle U,\bar{\delta}(y)\rangle$. By the very definition, we see that  for any $x\in D(\Delta^{1})$, $\bar{\delta}(x)\in D(\delta^{*1})$ and $\Delta^{1}(x)=\delta^{*1}\bar{\delta}(x)$. Moreover, we see obviously that $\delta^{*1}$ is a closed densely defined  operator (using $D(\delta^{*})\subset D(\delta^{*1})$ and $\bar{\delta}|_{M\o D(\bar{\delta})}$ is a densely defined formal adjoint.).  Note the following elementary lemma, using mainly the fact that $\delta^{\infty}$ is a derivation~:

\begin{lemma}\label{delta1}
$D(\Delta^{1})$ is a $*$-subalgebra of $M$ containing $D(\Delta)\cap M$, and for any $x,y\in D(\Delta^{1})$:
$$\Delta^{1}(xy)=\Delta^{1}(x)y+x\Delta^{1}(y)-2\sum_{i=1}^{N}m\circ(1\o\tm\o1)(\d(x)\o\d(y)),$$
where $m$ denote the multiplication map $L^{2}(M)\hat{\o}L^{2}(M)\rightarrow L^{1}(M)$. Finally, for any $x,y\in D(\Delta^{1})$: $\langle\Delta^{1}(x),y\rangle=\langle x,\Delta^{1}(y)\rangle.$ \end{lemma}

\begin{proof}[Proof]
 Take $x,y\in D(\Delta^{1})$, $z\in M\cap D(\delta)$,
thus \begin{align*}\langle\delta(xy),\delta(z)\rangle&=\langle\delta(x)y,\delta(z)\rangle+\langle x\delta(y),\delta(z)\rangle
\\ &=\langle\delta(x),\delta(z)y^*\rangle+\langle\delta(y),x^*\delta(z)\rangle
\\ &=\langle\delta(x),\delta(zy^*)\rangle+\langle\delta(y),\delta(x^*z)\rangle
\ \ -\langle\delta(x),z\delta(y^*)\rangle-\langle\delta(y),\delta(x^*)z\rangle
\\ &=\langle\Delta^{1}(x)y,z\rangle+\langle x\Delta^{1}(y),z)\rangle
-\langle\delta(y)z^*,\delta(x^*)\rangle-\langle\delta(y),\delta(x^*)z\rangle
\\ &=\langle\Delta^{1}(x)y,z\rangle+\langle x\Delta^{1}(y),z)\rangle
-2\sum_i Tr\left(\delta_i(y)^*\circ\delta_i(x)^*z\right).
\\ &=\langle\Delta^{1}(x)y,z\rangle+\langle x\Delta^{1}(y),z)\rangle
- 2\tau(zm(\sum_i \delta_i(x)\circ\delta_i(y))^*).
\end{align*}
In the fourth line, we used the definition of $\Delta^{1}$ and the fact $\delta$ is a real derivation. 
We used at the next to last line the identification of $L^2\o L^2$ with Hilbert Schmidt operators and the Trace on trace class, and the relation $\delta_i(x)^*=\delta_i(x^*)$ with the adjoint of Hilbert-Schmidt operators coming from the fact we have a real derivation.
At the last line, we used the multiplication map to $L^1(M)$, induced by $m(a\otimes b)=ab$.

This proves the domain property and the equation. 
\end{proof}

We will also need an extension $\overline{\Delta^{1}}:L^{2}(M)\rightarrow L^{1}(M)$. But the last equality of the previous lemma especially shows that $\Delta|_{D(\Delta)\cap M}:M\rightarrow L^{2}(M)$ is a ($\sigma$-weakly) densely defined formal adjoint of $\Delta^{1}:L^{2}(M)\rightarrow L^{1}(M)$, thus this operator is closable. 
And moreover, for any $x\in D(\Delta)\cap M, y\in D(\overline{\Delta^{1}})$, $\langle\Delta(x),y\rangle=\langle x,\overline{\Delta^{1}}(y)\rangle.$ 


Note the following elementary lemma, using $M\cap D(\Delta)$ is a core for $\Delta$ (thanks to stability of $M$ by $\phi_{t}$) :

\begin{lemmas}\label{bardelta1}
For any $x,y\in D(\Delta)$ with either $x$ or $y$ in $M$, then $xy\in D(\overline{\Delta^{1}})$:
$$\overline{\Delta^{1}}(xy)=\Delta(x)y+x\Delta(y)-2\sum_{i=1}^{N}m\circ(1\o\tau\circ m\o1)(\d(x)\o\d(y)),$$
where $m$ denotes the multiplication map $L^{2}(M)\hat{\o}L^{2}(M)\rightarrow L^{1}(M)$.\end{lemmas}

\subsubsection{Lemmas about the extension}
Here again we only assume $0$.

We can consider $(\delta\o1)\oplus(1\o\delta):L^{2}(M)\o L^{2}(M)\rightarrow(L^{2}(M)\o L^{2}(M)\o L^{2}(M))^{2N}$, (later abbreviated $\delta\o1\oplus1\o\delta$ or $\delta^{\o}$) which is easily seen to be densely defined on $D(\delta)\o_{alg}D(\delta)$, and closable (with an explicit densely defined adjoint coming from lemma \ref{triv} in case of assumption 0). We will write $\Delta^{\o}:=(\delta\o1\oplus1\o\delta)^{*}(\delta\o1\oplus1\o\delta)=\overline{\Delta\o1+1\o\Delta}$, which is thus a densely defined closed self-adjoint positive operator. It can be seen, as stated above, to be equal to the closure of $\Delta\o1+1\o\Delta$ (defined on $D(\Delta)\o_{alg}D(\Delta)$, using the stability of this space by $\phi_{t}\o\phi_{t}$, or rather more the regularization effect, implying this is a core of the previous closed operator).

{\gre Likewise, we define $\Delta^{\o(n+1)}$ on $D(\Delta)^{\o_{alg}(n+1)}\cap V_n$ (with the notation before $\Gamma_{0u}$), i.e. for $a_i\in D(\Delta)\cap L^2(M_0)$, by $$\Delta^{\o(n+1)}(a_0\o ...\o a_n)=\sum_{i=0}^na_0\o...\o a_{i-1}\o\Delta(a_i)\o a_{i+1}\o...\o a_n.$$

It clearly extends to a positive densely defined self-adjoint operator on $V_n$.
Assumption $e.4$ is obvious with $\D=D(\Delta)\ominus \C$. 
}

Recall $.\#(S_{t}^i-S_{s}^i)/\sqrt{t-s}:L^{2}(M_{s})\o L^{2}(M_{s})\rightarrow L^{2}(M)$ is the standard isometric map extending $(a\o b)\#(S_{t}^i-S_{s}^i)=a(S_{t}^i-S_{s}^i)b$. Likewise, we define $\#_j$ extending $(a\o b\o c)\#_1(S_{t}^i-S_{s}^i)=a(S_{t}^i-S_{s}^i)b\o c$ and $(a\o b\o c)\#_2(S_{t}^i-S_{s}^i)=a\o b(S_{t}^i-S_{s}^i)c$, $a,b,c\in M_s$.

\begin{corollary}\label{tens}
For any $U\in D(\Delta^{\o})\cap L^{2}(M_{s})\o L^{2}(M_{s})$, then $U\#(S_{t}^{i}-S_{s}^{i})\in D(\Delta)$ $(t\geq s)$ and :
 $$\Delta (U\#(S_{t}^{i}-S_{s}^{i}))=\Delta^{\o}(U)\#(S_{t}^{i}-S_{s}^{i}).$$
Moreover, for any $U\in L^{2}(M_{s})\o L^{2}(M_{s})$, $U\#(S_{t}^{i}-S_{s}^{i})\in D(\bar{\delta})$ if and only if $U\in D(\overline{\delta\o 1\oplus 1\o \delta})$, and  we also have $$\bar{\delta}(U\#(S_{t}^{i}-S_{s}^{i}))=\overline{\delta\o1}(U)\#_2(S_{t}^{i}-S_{s}^{i})+\overline{1\o\delta}(U)\#_1(S_{t}^{i}-S_{s}^{i}).$$
As a consequence, such an element is orthogonal to any $L^{2}(M_0\o M_0)$ (as claimed in assumption $d.2$). {\gre Finally, assumptions $e.2,e.4$ are verified by $\Delta^{\o (n+1)}$, and $e.3$ by $\delta$.}
\end{corollary}

\begin{proof}[Proof]
Consider $U \in (D(\Delta)\cap M_{s})\o_{alg}(D(\Delta)\cap M_{s})$, and by linearity even $U=a\o b$, then by lemma \ref{delta1}, we have $U\#(S_{t}^{i}-S_{s}^{i})\in D(\Delta)$ and the formula comes from the formula there (applied twice and using freeness to cancel the other terms). The density remark before the proof and the isometric map $.\#(S_{t}^i-S_{s}^i)/\sqrt{t-s}:L^{2}(M_{s})\o L^{2}(M_{s})\rightarrow L^{2}(M)$ conclude the general case. {\gre Assumption $e.2$ follows similarly.}

The second property comes from $\delta$ a derivation starting with the case $U \in D(\delta)\o_{alg}D(\delta)\cap M_{s}\o M_{s}$, and using  $\delta$ closed for the if part and in order to extend the formula. The only if part uses $\delta$ is defined first on $D(\partial)* \C\langle S_{s}^{(j)}
, 0\leq j\leq N, 0\leq s\leq \infty \rangle,$ and the fact we can take the approximation of $U\#(S_{t}^{i}-S_{s}^{i})$ in the image of $M_{s}\o_{alg} M_{s}$ by  $.\#(S_{t}^{i}-S_{s}^{i})$ (using freeness and the derivation property on the free product above to get the projection of a first approximation on the set above is dominated for the norm of 
$\delta$ by the first one). {\gre Assumption $e.3$ is also checked using the derivation property (the density statement is obvious). On $\E_n'=\E_n\cap( D(\partial)*\C\langle S_{s}^{(j)}
, 0\leq j\leq N, 0\leq s\leq \infty \rangle)$, we can apply the derivation property to show  $\delta(\E_n')\subset \overline{\oplus_{p+q=n}\E_p\o \E_q}$ (closure in $L^2$), implying the orthogonality statement. }
\end{proof}

\subsubsection{Summary of results under assumption 0}

We summarize the easy results obtained at this stage~:
\begin{lemmas}\label{main2-0}
With this assumption 0
, $\bar{\delta}$ and $\Delta$ satisfy the stability of filtration properties and  $\Gamma_{0u}(\omega=0)$ and also $b.2, c.2$(i.e. assumptions $a,b, c,e$ and $d.1,d.2$ of $\Gamma_{1u}(\omega=0,C)$, and also $d.3,d.4$ in case $\delta=\tilde{\delta}$).
\end{lemmas}

\subsection{Sufficient conditions for the main Assumption}
\subsubsection{Statement of result }
Let us sum up right now the assumptions we will {\gre use }
and our result. {\gre We consider here an exact coassociativity assumption even if an almost coassociativity (considered in a previous preprint version of this paper) would be enough. This will limit the applications of this section essentially to free difference quotients.}  
{\gre We will also consider the case} $\delta=\tilde{\delta}$ {\gre and consider elsewhere the case where we need and use two derivations}. 




\bigskip
{\gre 
\begin{minipage}{15,5cm}\textbf{Assumption 1}\textit{~: (a) $\partial:D(\partial)\to (D(\partial)\o_{alg}D(\partial))^{N}\subset (L^2(M_0)\o L^2(M_0^{op}))^{N}$  is 
 coassociative i.e. $\forall i,j\forall x\in D(\partial)$ :}
  $$({\red \partial_j\o 1})\circ \partial_i(x)-({\red 1\o {\partial}_i})\circ \partial_j(x)=0.$$
\textit{(a') $\partial$ satisfy assumption $0$ and $\partial_{j}^{*}1\o1\in M_{0}$.}

\medskip
\textit{Moreover, we suppose that (b)
$\partial_{j}^{*}1\o1\in D(\overline{\partial})$ 
and $$(1\o O)\bar{\partial}_i\partial_{j}^{*}1\o1\in M_{0}\overline{\otimes}
M_{0}^{op}.$$}
 \end{minipage}
 }

 \begin{theorem}\label{main2}
Under assumption 1, $\bar{\delta}$ and $\Delta$ satisfy the stability of filtration properties and assumption $\Gamma_{1u}(\omega=0,C)$ for some finite constant $C$ in the context $\delta=\tilde{\delta}$.
\end{theorem}

\subsubsection{Boundedness for $(1\o \tau)\circ\delta_k$ under assumption 1}
We first recall lemma 10 
in \cite{nonGam}, which is stated there for the free difference quotient, but the coassociative case is identical. We can and will also extend it elsewhere to an almost coassociative case.

\begin{lemmas}\label{FPC2}
Assume Assumption 1. Let $Z \in M\cap
D(\bar{\delta})$, 
then
 the following inequality holds~:
  \begin{align*}||(1\otimes\tau)(\bar{\delta_{i}}(Z))||_{2}&\leq  ||Z||_2\left[\left(2||\delta_{i}^{*}(1\o 1)||\right)+\left(||\delta_{i}^{*}(1\o 1)||^2+||\overline{\delta}_i\delta_{i}^{*}(1\o 1)||_{M\overline{\o} M^{\red op}}\right)^{1/2}\right].
\end{align*}
As a consequence, $(1\otimes\tau)\circ\bar{\delta_{i}}$ extends as a bounded map $L^{2}(M,\tau)\rightarrow L^{2}(M,\tau)$. 
\end{lemmas}
\subsubsection{Almost commutation of $\delta$ and $\Delta$ on an extended domain}

We are now ready to solve our main domain issues (to get $f$) in the next~:

\begin{lemma}\label{in2deltaDelta}Assume assumption 1 
\begin{enumerate}
\item[(i)]
For any $x\in D(\delta)$ we have 
$x\in D(\Delta_j)$, 
$\delta_i(x)\in D(\Delta_j\o1+1\o \Delta_j)$, $x\in D(\Delta^{3/2})$ and~:
 \begin{align*} 
&{\delta}_{i}\Delta_{j}(x)= (1 \o \Delta_j+\Delta_{j}\o 1){\delta}_i(x)
+\delta_{j}(x)\#((1\o O)\bar{\partial}_{i}\partial_{j}^{*}(1\o1) ).
\end{align*}
\item[(ii)]
 If $x\in D(\bar{\delta})$ (resp. $x\in D(\Delta)$)  then so is $1\otimes\tau(\bar{\delta_{i}}(x))$. 
\item[(iii)]
$D(\Delta^{3/2})\subset D(\overline{\Delta \o1+1\o \Delta }\circ\bar{\delta} )$ and moreover we have for any $x\in D(\Delta^{3/2})$ \begin{align*} 
&\bar{\delta}_{i}\Delta(x)= \Delta^{\o}\bar{\delta}_i(x)
+\sum_{j=1}^N\bar{\delta}_{j}(x)\#((1\o O)\bar{\partial}_{i}\partial_{j}^{*}(1\o1) ).
\end{align*}
\end{enumerate}
\end{lemma}
\begin{proof}[Proof]

\begin{enumerate}\item[(i)]
Consider $x\in D(\delta)$, by assumption 1 $\delta(x)\in (D(\delta)\o_{alg} D(\delta))^{N}\subset (M\o_{alg} M)^N$ (the extension from $\partial$ to $\delta$ is easy), thus using lemma \ref{triv}, $\delta_j(x)\in D(\delta^*_j)$, i.e. $x\in D(\Delta_j)$ for all $j$ and  
$$\Delta_j(x)=\delta_j(x)\#\delta_j^*(1\o 1)  -
m\circ(1\otimes\tau\o 1)\circ[\delta_j\o 1]\circ\delta_j(x)-m\circ(1\o \tau\otimes 1)\circ[1\o \delta_{j}]\circ\delta_j(x).$$
Recall $\delta_j^*(1\o 1)=\partial_j^*1\o 1\in D(\delta)$, and for any $j,k$ $(\delta_j\o 1)\delta_k(x)\in D(\delta)\o_{alg}D(\delta)\o_{alg}D(\delta)$
 so that one gets $\Delta_j(x)\in D(\delta_i)$ (and also the statement $\delta_i(x)\in D(\Delta_j\o 1)$ using again lemma \ref{triv}) and applying the derivation property for $\delta_i$, we get (recall the notation  for $\#$ before lemma \ref{triv} and $\#_i$ similar to the one used before corollary \ref{tens}, $(a\o b\o c )\#_2d=a\o bdc$, $(a\o b\o c )\#_1d=ad b\o c$):
  \begin{align*}\overline{\delta}_i\Delta_j(x)&=((\delta_i\o1)\delta_j(x))\#_2\delta_j^*(1\o 1) +((1\o\delta_i)\delta_j(x))\#_1\delta_j^*(1\o 1) \\&+\delta_j(x)\#((1\o O)\overline{\partial}_i\partial_j^*(1\o 1))\\&-
(1\o m\circ(1\otimes\tau\o 1))\circ[\delta_i\o 1\o1]\circ[\delta_j\o 1]\circ\delta_j(x)\\ &-
( m\circ(1\otimes\tau\o 1)\o 1)\circ[ 1\o1\o \delta_i]\circ[\delta_j\o 1]\circ\delta_j(x)\\&-
(1\o m\circ(1\otimes\tau\o 1))\circ[\delta_i\o 1\o1]\circ[1\o\delta_j]\circ\delta_j(x)\\ &-
( m\circ(1\otimes\tau\o 1)\o 1)\circ[ 1\o1\o \delta_i]\circ[1\o\delta_j]\circ\delta_j(x).\end{align*}
 
Now, one can easily extend coassociativity to $\delta$ (the coassociator $({\red \delta_j\o 1})\circ \delta_i-({\red 1\o {\delta}_i})\circ \delta_j$ being a derivation, coassociativity is checked on generators).
Thus one can rewrite :
 \begin{align*}&-((\delta_i\o 1)\delta_j(x))\#_2\delta_j^*(1\o 1)=-((1\o\delta_j)\delta_i(x))\#_2\delta_j^*(1\o 1),\\ 
&(1\o m\circ(1\otimes\tau\o 1))\circ[\delta_i\o 1\o1]\circ[\delta_j\o 1]\circ\delta_j(x)\\&=(1\o [m\circ(1\otimes\tau\o 1)\circ[\delta_j\o 1]\circ\delta_j])\circ\delta_i(x),\\
&(1\o m\circ(1\otimes\tau\o 1))\circ[\delta_i\o 1\o1]\circ[1\o\delta_j]\circ\delta_j(x)\\&=(1\o [m\circ(1\otimes\tau\o 1)\circ[1\o\delta_j]\circ\delta_j])\circ\delta_i(x),\end{align*}
 
and similar results for other lines in our previous sum. Using the formula in lemma \ref{triv}, the three previous lines sum up to $-(1\o\Delta_j)\circ\delta_i(x)$. Doing the same for the other lines, we thus proved the expected formula.

\item[(ii)] 
Let $x\in D(\bar{\delta})$ and take $x_{n}\in D(\delta)$ converging to $x$ in $D(\bar{\delta})$. 
 We can compute (using coassociativity 
  again): $$\bar{\delta_{j}}(1\o\tau)(\bar{\delta_{i}}(x_{n}))= 
   (1\o[(1\o\tau)(\bar{\delta_{i}})])(\bar{\delta_{j}}(x_{n})).$$
By the boundedness result of lemma \ref{FPC2}, the right hand side converges and this gives the result since $\bar{\delta}$ is closed.

For the second statement, consider 
the equation of (i) applied via scalar product to $U\in (D(\bar{\delta})\cap M)\o_{alg}(D(\bar{\delta})\cap M)$ :
\begin{align} \label{eqDeldel} 
&\langle \Delta_{j}(x_n),\delta_i^*(U)\rangle= \langle(1 \o \Delta_j+\Delta_{j}\o 1){\delta}_i(x_n),U\rangle
+\langle\delta_{j}(x_n)\#((1\o O)\bar{\partial}_{i}\partial_{j}^{*}(1\o1) ),U\rangle
\end{align}

with $U=V\otimes1,V\in D(\delta)$, $x_n$ above (with $x\in D(\Delta)$). Note that using lemma \ref{triv}, assumption 1 and our first result in (ii), $ {\delta}_{i}^{*} (V\o 1)\in D(\bar{\delta})$ with $\bar{\delta}_{j}\delta_{i}^{*}( V\o 1)={\delta}_{j}(V)\delta_{i}^{*}1\o 1+V\bar{\delta}_{j}\delta_{i}^{*}(1\o 1)-\delta_{j}1\o\tau\delta_{i}(V)\in L^{2}(M\o M)$. Thus
 $\langle\Delta_{j}(x_n),{\delta}_{i}^{*} U\rangle=\langle\delta_{j}(x_n),\overline{\delta}_{j}\delta_{i}^{*}( V\o 1)\rangle$. Note also that $\langle1\o \tau\delta_{i}(x_n),\Delta V\rangle =\langle\overline{\delta}(1\o \tau\delta_{i}(x_n)),\overline{\delta}(V)\rangle$  converges to the analog with $x$ by what we have just proved. 
Since the resulting terms in (\ref{eqDeldel}) are bounded with respect to $||\delta(x_n)||_2$, we can get the equation at the limit $x_n\to x$.

We thus got :
$$\langle(1\o\tau\delta_i)\Delta(x), V\rangle
=\langle\overline{\delta}(1\o \tau\delta_{i}(x)),\overline{\delta}(V)\rangle +\sum_ {j=1}^N\langle\bar{\delta}_{j}(x)\#((1\o O)\bar{\partial}_{i}\partial_{j}^{*}(1\o1) ),V\o 1\rangle.$$
Now we can extend this from $V\in D(\delta)$ to $V\in D(\bar{\delta})$ and thus we obtain our result by definition of $\Delta$.

\item[(iii)]

Consider again  this time the variant of equation (\ref{eqDeldel}) 
 with $U\in (M\cap D(\Delta))\o_{alg}(M\cap D(\Delta))$, and $x\in D(\delta)$. 
Everything reduces to $U=a\o b$. Using lemma \ref{triv},  we have $\delta_{i}^{*}(U)=a\delta_{i}^{*}(1\o1)b-(1\o\tau)\delta_{i}(a)b-a(\tau\o1)\delta_{i}(b)$. But now, $a,b\in M\cap D(\Delta)$,$(1\o\tau)\delta_{i}(a),(\tau\o1)\delta_{i}(b)\in D(\Delta)$ by (ii) thus lemma \ref{bardelta1} proves $-(1\o\tau)\delta_{i}(a)b-a(\tau\o1)\delta_{i}(b)\in D(\overline{\Delta^{1}})$.  Then, let us write, for any $U\in (M\cap D(\Delta))\o_{alg}(M\cap D(\Delta))$ (with the notation $(a\o b)\#c=acb$), $\delta_{i}^{*}(U)=U\#\delta_{i}^{*}(1\o1)-V$ with $V\in D(\overline{\Delta^{1}})$. We can now rewrite our equation (using $\d$ is a derivation on $M\cap D(\d)$ to see $U\#\delta_{i}^{*}(1\o1)\in D(\d)$):
\begin{align*}\sum_ {j=1}^N\langle\delta_j(x)&,\bar{\delta}_j(U\#\delta_{i}^{*}(1\o1))\rangle-\langle\bar{\delta}_{j}(x)\#((1\o O)\bar{\partial}_{i}\partial_{j}^{*}(1\o1) ),U\rangle\\ &=\langle x,\overline{\Delta^{1}}(V)\rangle+
  \langle\d(x),[1\o\Delta+\Delta\o1](U)\rangle
\end{align*}
Now, once again using the second part of lemma \ref{triv}, we get this for any $x\in D(\Delta^{3/2})\cap M\subset D(\bar{\delta})\cap M$. Using the remark before lemma \ref{bardelta1}, we can rewrite $\langle x,\overline{\Delta^{1}}(V)\rangle=\langle \Delta(x),V\rangle$, and thus, finally coming back to our original notation~:
\begin{align*}\langle\d\Delta(x),U\rangle&=  \langle\d(x),[1\o\Delta +\Delta\o1](U)\rangle+\sum_ {j=1}^N\langle\bar{\delta}_{j}(x)\#((1\o O)\bar{\partial}_{i}\partial_{j}^{*}(1\o1) ),U\rangle
\end{align*}
Finally, (using stability by $\phi_{t}\o\phi_{t}$) it is easily seen that $(M\cap D(\Delta))\o_{alg}(M\cap D(\Delta))$ is a core for $\overline{\Delta\o 1+1\o\Delta}$, and thus we can take $U$ in the domain of that operator 
and finally, since this operator is closed, we get our result. The extension from $x\in D(\Delta^{3/2})\cap M$ to $x\in D(\Delta^{3/2})$ is easy.

\end{enumerate}
\end{proof}





\subsubsection{Proof of Theorem \ref{main2}}
Using  lemma \ref{main2-0}, it only remains to check assumption $f,h$. Lemma \ref{in2deltaDelta} (iii) proves $f.1$ and $f.2$ with $\H:L^2(M\o M)^N\to L^2(M\o M)^N$ given by $(\H(C))_i=\sum_{j=1}^NC_{j}\#((1\o O)\bar{\partial}_{i}\partial_{j}^{*}(1\o1) )$ so that $||\H||\leq ||((1\o O)\bar{\partial}_{i}\partial_{j}^{*}(1\o1) )_{(i,j)}||_{M_N(M\o M^{op})}$.

It remains to check $h$. Consider $x\in D(\Delta)$ then $\eta_{\alpha}(x)\in D(\Delta^{3/2})$, thus we can apply lemma  \ref{in2deltaDelta} (iii) to get :
\begin{align*} 
&\langle\Delta(\eta_{\alpha}(x)),\Delta(\eta_{\alpha}(x))\rangle  =\sum_{i=1}^N\langle\bar{\delta}_{i}\Delta(\eta_{\alpha}(x)),\bar{\delta}_{i}(\eta_{\alpha}(x))\rangle \\&=\sum_{i=1}^N \langle\Delta^{\o}\bar{\delta}_i(\eta_{\alpha}(x)),\bar{\delta}_{i}(\eta_{\alpha}(x))\rangle 
+\langle\H(\bar{\delta}(\eta_{\alpha}(x)))_i,\bar{\delta}_{i}(\eta_{\alpha}(x))\rangle\\ &=\sum_{i,j=1}^N\langle\overline{\delta_j\o1\oplus1\o\delta_j}\bar{\delta}_i(\eta_{\alpha}(x)),\overline{\delta_j\o1\oplus1\o\delta_j}\bar{\delta}_{i}(\eta_{\alpha}(x))\rangle+\sum_{i=1}^N\langle\H(\bar{\delta}(\eta_{\alpha}(x)))_i,\bar{\delta}_{i}(\eta_{\alpha}(x))\rangle.
\end{align*}

Since we assume $x\in D(\Delta)$, the left hand side and the second term in the right hand side converge when $\alpha\to\infty$ showing that $x\in D(\overline{\delta\o1\oplus1\o\delta}\circ\bar{\delta})$ as expected.
Now, the inequality stated in $h$ is a  tensor variant of the one stated in $c$ and already checked. \ \hfill \ensuremath{\Box}

\section{Complementary properties of our main example}

\subsection{An Ito Formula for resolvent operators under weak assumptions}

Let us consider an integral of the form~:
$$X_{t}=X_{0}+\int_{0}^{t}K_{s}ds+\int_{0}^{t}U_{s}\#dS_{s},$$
where $X_{0}\in M_{0}=W$, $s\mapsto K_{s}$ weakly measurable with $K_{s}\in
L^{1}(M_{s})$, $\int_{0}^{T}||K_{s}||_{1}ds<\infty \forall T>0$ and
$U\in \B_{2}^{a}$. We also assume
$K_{s}=K_{s}^{*}$,$U_{s}=U_{s}^{*}$ (in this section we use the involution
induced by $HS(M)$, i.e. $(a\o b)^{*}=b^{*}\o a^{*}$),$X_{0}=X_{0}^{*}$ so that $X_{t}=X_{t}^{*}$.

We would like to find a formula for $(z+X_{t})^{-1}, z\in \C, \Im z >0$, to
compute the Cauchy-transform of $X_{t}$ with this unbounded $X_{t}\in
L^{1}(M_{t})$. If we supposed $K_{s}\in M_{s}$, $U_{s}\in \B_{\infty}^{a}$
Proposition 4.3.4 of \cite{BS98} would conclude (see this article for
the notation, the case with N free Brownian motions as in our case is similar
to their case, especially we write in this section also $\#$ for multiplication in $M\o M^{op}\o M$ without confusion with the previous notation for multiplication in $M\o M^{op}$) since $f(x)=\frac{1}{z+x}=\int_{\R}e^{ixy}\mu(dy)$ with
$\mu(dy)=-i1_{[0,\infty)}e^{izy}dy$ (which satisfy $\II_{2}(f)<\infty$, and thus their results apply).

But we are not in such a bad position because all the terms of their
expression in the Ito Formula for $(z+X_{t})^{-1}$ make sense, this almost only require{\gre s} applying a standard density argument left to the reader.

\begin{proposition}\label{Ito-res}
With the previous assumptions we have~:
\begin{align*}&(z+X_{t})^{-1}  = (z+X_{0})^{-1} -
\int_{0}^{t}\left[(z+X_{s})^{-1}\otimes(z+X_{s})^{-1}\right]\#U_{s}\#dS_{s}
\\ &- \int_{0}^{t}\left[(z+X_{s})^{-1}\otimes(z+X_{s})^{-1}\right]\#K_{s}ds
\\ &+\sum_{i=1}^{N} \int_{0}^{t}m\circ(1\otimes\tau\otimes1)((1\otimes U_{s}^{(i)})\#((z+X_{s})^{-1}\otimes(z+X_{s})^{-1}\otimes(z+X_{s})^{-1})\#(U_{s}^{(i)}\otimes1))ds.\end{align*}
\end{proposition}

The two next lemmas are also left to the reader.
\begin{lemma}\label{Kaplansky}
Let\begin{align*}&X_{t}=X_{0}+\int_{0}^{t}K_{s}ds+\int_{0}^{t}U_{s}\#dS_{s},\end{align*}

where $X_{0}\in M_{0}$, $s\mapsto K_{s}$ weakly measurable with $K_{s}\in
L^{1}([0,T],L^{1}(M_{s}))$,  and
$U\in \B_{2}^{a}$. We also assume $X_{t}\in M$ (in a bounded way in $t$). Let say $||X_t||<1$.

Then, there exists $X_{t}^n=X_{0}+\int_{0}^{t}K_{s}^nds+\int_{0}^{t}U_{s}^n\#dS_{s}$
 with $s\mapsto K_{s}^n$ weakly measurable with $K_{s}^n\in
L^{\infty}([0,T])\overline{\o}M_{s}$, $K^n$ converging to $K$ in $L^{1}([0,T],L^{1}(M_{s}))$,  and
$U^n\in \B_{\infty}^{a}$, $U^n$ converging to $U$ in $\B_{2}^{a}$. Moreover, we have $||X_t^n||\leq 1$. 
\end{lemma}

The following variant of the Ito product formula (proposition 4.3.2  in
\cite{BS98}) is now obvious~:
\begin{lemmas}\label{Ito}

Let \begin{align*}&X_{t}=X_{0}+\int_{0}^{t}K_{s}ds+\int_{0}^{t}U_{s}\#dS_{s},\\ &Y_{t}=Y_{0}+\int_{0}^{t}L_{s}ds+\int_{0}^{t}V_{s}\#dS_{s},\end{align*}

where $X_{0},Y_{0}\in M_{0}$, $s\mapsto K_{s}$,$s\mapsto L_{s}$ weakly measurable with $K_{s},L_{s}\in
L^{1}([0,T],L^{1}(M_{s}))$,  and
$U,V\in \B_{2}^{a}$. We also assume $X_{t},Y_{t}\in M$ (in a bounded way in $t$).

Then, for any $t\leq T$~:
\begin{align*}X_{t}Y_{t}=X_{0}Y_{0}+&\int_{0}^{t}(X_{s}L_{s}+K_{s}Y_{s})ds+\int_{0}^{t}m\circ{\gre (1\o(\tau\circ
m)\o 1)}\left(U_{s}\o V_{s}\right)ds\\ &+\int_{0}^{t}(X_{s}V_{s}+U_{s}Y_{s})\#dS_{s}.\end{align*}
\end{lemmas}


\subsection{Boundedness}
In this subsection, we are interested in the example of part 2. Under assumption 0, we write $X_{t}\in\B_{2,\phi\bar{\delta} }^{a},X_{t}^{\epsilon}\in\B_{2,\Delta^{1/2} }^{a}$ the solutions given by theorem \ref{main} (i) and lemma \ref{main2-0} . We moreover consider an initial condition $X_{0}\in M_{0}\cap D(\bar{\delta})$.

\begin{proposition}\label{bound}With those assumptions, for any complex number $z$ with $\Im z>0$,
  $\frac{1}{z+X_{s}^{\epsilon}}$ is in $\B_{2,\Delta^{1/2}}^{a}$ and
\begin{align*}&(z+X_{t}^{\epsilon})^{-1}  = \phi_{t}((z+X_{0}^{\epsilon})^{-1}) +(1-\epsilon)\int_{0}^{t}\phi_{t-s}(\bar{\delta }((z+X_{s}^{\epsilon})^{-1})\#dS_{s})
\\ &+((1-\epsilon)^{2}-1)\sum_{i=1}^{N}\\ & \int_{0}^{t}\phi_{t-s}\left(m\circ(1\otimes\tau\otimes1)((z+X_{s}^{\epsilon})^{-1}\bar{\delta _{i}}(X_{s}^{\epsilon})(z+X_{s}^{\epsilon})^{-1}\bar{\delta _{i}}(X_{s}^{\epsilon})(z+X_{s}^{\epsilon})^{-1})\right)ds.
\end{align*}

As a consequence, if we  assume moreover $||X_{t}||_{2}=||X_{0}||_{2}$ (a.e. t, this is the case e.g. for a mild solution given by Theorem \ref{main} (ii)) then $X_{t}\in M$ for all $t$ (recall we supposed $X_{0}\in
M$) and we also have $||X_{t}||\leq ||X_{0}||$ (actually equal a.e.), and likewise for any $\epsilon>0$, $X_{t}^{\epsilon}\in M$.
{\gre Finally, if $||X_{t}||_{2}=||X_{0}||_{2}$ a.e. in t, then $||X_t-X_t^{\epsilon}||_2\to 0$ a.e. in t. }
\end{proposition}

\begin{proof}[Proof]
Let $\epsilon>0$. 
Since we have a mild solution at $\epsilon$ level, by theorem \ref{main} and since a mild solution is a weak solution as seen in  Proposition \ref{MImpS}, we get by self-adjointness of $\Delta$ for characterizing its domain that $\int_{0}^{t}X_{s}^{\epsilon}ds\in D(\Delta)$ and~:
$$X_{t}^{\epsilon}=X_{0}-\frac{1}{2}\Delta \int_{0}^{t}X_{s}^{\epsilon}ds+(1-\epsilon)\int_{0}^{t}\delta (X_{s}^{\epsilon})\#dS_{s}.$$

Thus, applying a resolvent, using lemma \ref{deltaInt} (ii), we deduce for  any $\alpha>0$~:
$$\eta_{\alpha}(X_{t}^{\epsilon})=\eta_{\alpha}(X_{0})-\frac{1}{2} \int_{0}^{t}\Delta\eta_{\alpha}(X_{s}^{\epsilon})ds+(1-\epsilon)\int_{0}^{t}\eta_{\alpha}^{\otimes}\delta (X_{s}^{\epsilon})\#dS_{s},$$
where $X_{0}\in M_{0}$, $s\mapsto K_{s}=-\frac{1}{2}\Delta\eta_{\alpha}( X_{s}^{\epsilon})$ weakly measurable with $K_{s}\in
L^{2}(M_{s})$, $\int_{0}^{T}||K_{s}||_{2}^{2}ds<\infty \forall T>0$ (all this using the definition of $\B_{2,\Delta^{1/2 }}^{a}$) and
$U_{s}=\eta_{\alpha}^{\otimes}\delta (X_{s}^{\epsilon})\in \B_{2}^{a}$. Recall 
that $\eta_{\alpha}(X_{0})\in M_0$ by the general Dirichlet form theory implying $\eta_{\alpha}$ is a completely positive contraction on $M$. We are in position to apply
proposition \ref{Ito-res},
thus we have~:
\begin{align*}(z&+\eta_{\alpha}(X_{t}^{\epsilon}))^{-1}  = (z+\eta_{\alpha}(X_{0}^{\epsilon}))^{-1} \\ &-
(1-\epsilon)\int_{0}^{t}\left[(z+\eta_{\alpha}(X_{s}^{\epsilon}))^{-1}\otimes(z+\eta_{\alpha}(X_{s}^{\epsilon}))^{-1}\right]\#\eta_{\alpha}^{\o}\bar{\delta }(X_{s}^{\epsilon})\#dS_{s}
\\ &+\frac{1}{2} \int_{0}^{t}\left[(z+\eta_{\alpha}(X_{s}^{\epsilon}))^{-1}\otimes(z+\eta_{\alpha}(X_{s}^{\epsilon}))^{-1}\right]\#\Delta  \eta_{\alpha}(X_{s}^{\epsilon})ds
\\ &+(1-\epsilon)^{2}\sum_{i=1}^{N} \int_{0}^{t}ds\\ &m\circ(1\otimes\tau\otimes1)((z+\eta_{\alpha}(X_{s}^{\epsilon}))^{-1}\eta_{\alpha}^{\o}\bar{\delta _{i}}(X_{s}^{\epsilon})(z+\eta_{\alpha}(X_{s}^{\epsilon}))^{-1}\eta_{\alpha}^{\o}\bar{\delta _{i}}(X_{s}^{\epsilon})(z+\eta_{\alpha}(X_{s}^{\epsilon}))^{-1}).\end{align*}

But, note that for any $x\in D(\bar{\delta })$, $(z+x)^{-1}\in
D(\bar{\delta })$, and
$\bar{\delta }((z+x)^{-1})=-(z+x)^{-1}\bar{\delta }(x)(z+x)^{-1}$. Indeed, we
check this easily on $D(\delta )\subset M$ by Leibniz rule, and taking
$x_{n}\in D(\delta )$ converging to $x$ in $D(\bar{\delta })$, a usual formula
on resolvent operators
$(z+x_{n})^{-1}-(z+x)^{-1}=(z+x_{n})^{-1}(x-x_{n})(z+x)^{-1}$ {\gre gives}
convergence of $(z+x_{n})^{-1}$ to $(z+x)^{-1}$ in $L^{2}(M)$, and thus of
$\bar{\delta }((z+x_{n})^{-1})$ in $L^{1}(M\otimes M)$ to
$(z+x)^{-1}\bar{\delta }(x)(z+x)^{-1}$. A fortiori, we have weak convergence
in $L^{2}(M\otimes M)$. 
Since a
 convex set in $L^{2}(M)\oplus L^{2}(M\otimes M)$ is closed if and only if it is weakly closed by Hahn-Banach
 theorem, we get $(z+x)^{-1}\in
D(\bar{\delta })$ and the result.

Analogously, we have for any $x\in D(\Delta )$, $(z+x)^{-1}\in
D(\Delta ^{1})$ (cf. the paragraph before lemma \ref{delta1} for a definition) and moreover~:
\begin{align*} -&\Delta ^{1}((z+x)^{-1})=(z+x)^{-1}\Delta (x)(z+x)^{-1}\\ &+2\sum_{i=1}^{N}
m\circ(1\otimes\tau\otimes1)(1\otimes
\bar{\delta _{i}}(x)\#(z+x)^{-1}\otimes(z+x)^{-1}\otimes(z+x)^{-1}\#\bar{\delta _{i}}(x)\otimes1).\end{align*}

Let us write $R_{t,z,\alpha,\epsilon}=(z+\eta_{\alpha}(X_{t}^{\epsilon}))^{-1}.$
Thus, we have obtained, if we apply this formula to our previous equation in making appear terms by emphasizing  "commutators" of $\eta_{\alpha}^{\o}$ and $\delta$. We also write $Y_{s,z,\alpha,\epsilon,i}:=(R_{s,z,\alpha,\epsilon}(\eta_{\alpha}^{\o}\bar{\delta _{i}}(X_{s}^{\epsilon})-\bar{\delta _{i}}\eta_{\alpha}(X_{s}^{\epsilon}))R_{s,z,\alpha,\epsilon}$~:
\begin{align*}&R_{t,z,\alpha,\epsilon}  = R_{0,z,\alpha,\epsilon} +
(1-\epsilon)\int_{0}^{t}(\bar{\delta }(R_{s,z,\alpha,\epsilon}){\gre -}Y_{s,z,\alpha,\epsilon})\#dS_{s}-\frac{1}{2} \int_{0}^{t}\Delta^{1}((z+\eta_{\alpha}(X_{s}^{\epsilon}))^{-1})\ ds
\\ &+\sum_{i=1}^{N} \int_{0}^{t}m\circ(1\otimes\tau\otimes1)(Y_{s,z,\alpha,\epsilon,i}\eta_{\alpha}^{\o}\bar{\delta _{i}}(X_{s}^{\epsilon})R_{s,z,\alpha,\epsilon}+ R_{s,z,\alpha,\epsilon}\bar{\delta _{i}}\eta_{\alpha}(X_{s}^{\epsilon})Y_{s,z,\alpha,\epsilon,i})ds\\ &+((1-\epsilon)^{2}-1)\sum_{i=1}^{N} \int_{0}^{t}m\circ(1\otimes\tau\otimes1)(R_{s,z,\alpha,\epsilon}\eta_{\alpha}^{\o}\bar{\delta _{i}}(X_{s}^{\epsilon})R_{s,z,\alpha,\epsilon}\eta_{\alpha}^{\o}\bar{\delta _{i}}(X_{s}^{\epsilon})R_{s,z,\alpha,\epsilon})ds.\end{align*}


As in the proof of proposition \ref{SImpM} showing that a strong solution
is a mild solution, but take here $\zeta\in D(\Delta)\cap M$ in the proof, we have~: 
\begin{align*}&R_{t,z,\alpha,\epsilon}  = \phi_t(R_{0,z,\alpha,\epsilon}) +
(1-\epsilon)\int_{0}^{t}\phi_{t-s}^{\o}(\bar{\delta }(R_{s,z,\alpha,\epsilon})-Y_{s,z,\alpha,\epsilon})\#dS_{s}
\\ &+\sum_{i=1}^{N} \int_{0}^{t}\phi_{t-s}\left(m\circ(1\otimes\tau\otimes1)(Y_{s,z,\alpha,\epsilon,i}\eta_{\alpha}^{\o}\bar{\delta _{i}}(X_{s}^{\epsilon})R_{s,z,\alpha,\epsilon})+R_{s,z,\alpha,\epsilon}\bar{\delta _{i}}\eta_{\alpha}(X_{s}^{\epsilon})Y_{s,z,\alpha,\epsilon,i}\right)ds\\ &+((1-\epsilon)^{2}-1)\sum_{i=1}^{N} \int_{0}^{t}\phi_{t-s}m\circ(1\otimes\tau\otimes1)(R_{s,z,\alpha,\epsilon}\eta_{\alpha}^{\o}\bar{\delta _{i}}(X_{s}^{\epsilon})R_{s,z,\alpha,\epsilon}\eta_{\alpha}^{\o}\bar{\delta _{i}}(X_{s}^{\epsilon})R_{s,z,\alpha,\epsilon})ds.\end{align*}

We now want to make $\alpha$ tend to $\infty$. The three terms with $Y$ tend to zero by dominated convergence theorem (domination modulo constant by $||\bar{\delta}(X_{s}^{\epsilon})||_{2}^{2}$ since $X^{\epsilon}\in\B_{2,\Delta^{1/2 }}^{a}$.) In the last line we can remove $\eta_{\alpha}^{\o}$ in the same way and we get weak convergence in $L^{1}$ to the expected limit 
(of course we have to use $\phi$ bounded on $M$). Clearly, the two resolvent operators in the first line converge in $L^{2}$ and the same kind of reasoning already made shows that
$\bar{\delta }((z+\eta_{\alpha}(X_{s}^{\epsilon}))^{-1})$ weakly converges in $L^{2}$ to
$\bar{\delta }((z+X_{s}^{\epsilon})^{-1})$\footnote{remark that this second term is already
known to exists; by boundedness in $L^{2}$ of the  convergent
$\delta \eta_{\alpha}(X_{s}^{\epsilon})$, we get that
$(z+\eta_{\alpha}(X_{s}^{\epsilon}))^{-1}\delta (\eta_{\alpha}(X_{s}^{\epsilon}))(z+\eta_{\alpha}(X_{s}^{\epsilon}))^{-1}$
is close in $||.||_{1}$ of
$(z+X_{s}^{\epsilon})^{-1}\delta (\eta_{\alpha}(X_{s}^{\epsilon}))(z+X_{s}^{\epsilon})^{-1}$, and finally, with convergence in $L^{2}$ of $(z+X_{s}^{\epsilon})^{-1}\delta \eta_{\alpha}(X_{s}^{\epsilon})(z+X_{s}^{\epsilon})^{-1}$
to $(z+X_{s}^{\epsilon})^{-1}\delta (X_{s}^{\epsilon})(z+X_{s}^{\epsilon})^{-1}$; we have thus obtained the
 convergence in $L^{1}$, using that the two terms are known to be  in $L^{2}$
and the sequence bounded in this space, you get the result.}. A dominated convergence theorem concludes as above for the corresponding stochastic integral. At the end, we have got weak convergence in $L^{1}$ of all terms so that :
\begin{align*}&(z+X_{t}^{\epsilon})^{-1}  = \phi_{t}((z+X_{0}^{\epsilon})^{-1}) +(1-\epsilon)\int_{0}^{t}\phi_{t-s}(\bar{\delta }((z+X_{s}^{\epsilon})^{-1})\#dS_{s})
+((1-\epsilon)^{2}-1)\times\\ &\sum_{i=1}^{N} \int_{0}^{t}ds \phi_{t-s}\left(m\circ(1\otimes\tau\otimes1)((z+X_{s}^{\epsilon})^{-1}\bar{\delta _{i}}(X_{s}^{\epsilon})(z+X_{s}^{\epsilon})^{-1}\bar{\delta _{i}}(X_{s}^{\epsilon})(z+X_{s}^{\epsilon})^{-1})\right).
\end{align*}

We now want to make $\epsilon$ tend to 0, after taking the trace, to get the second statement. Note that in our context of section 2 where $||\delta(x)||_2=||\Delta^{1/2}(x)||_2$, (\ref{crucial}) gives :$$(1-(1-\epsilon)^{2})\int_{0}^{t}||\bar{\delta}(X_{s}^{\epsilon})||_{2}^{2}ds=||X_{0}||_{2}^{2}-||X_{t}^{\epsilon}||_{2}^{2}.$$
Incidentally, this proves the statement that $||X_{0}||_{2}^{2}=||X_{t}||_{2}^{2}$ in case (ii) of theorem \ref{main} since we proved there convergence of $X_t$ in $L^2$ and boundedness of $||\bar{\delta}(X_{s}^{\epsilon})||_{2}$.

 But (modulo extraction) the weak limit defining $X_{t}$ gives $||X_{t}||_{2}\leq\liminf||X_{t}^{\epsilon}||_{2}$ and thus 
\begin{equation}\label{epsisom}\limsup_{\epsilon\rightarrow 0}(1-(1-\epsilon)^{2})\int_{0}^{t}||\bar{\delta}(X_{s}^{\epsilon})||_{2}^{2}ds\leq||X_{0}||_{2}^{2}-||X_{t}||_{2}^{2}.\end{equation}

And the last term is almost everywhere $0$ under our assumption. {\gre As a consequence, since we already know $X_{t}^{\epsilon}$ converge{\gre s} to $X_{t}$ weakly in $L^2$ by theorem \ref{main}(i), we deduce the stated $||.||_{2}$ convergence of $X_{t}^{\epsilon}$ to $X_{t}$ on the a.e. set where $||X_{t}||_{2}=||X_{0}||_{2}$}. Moreover, the trace of the second line of the equality of  proposition \ref{bound} is bounded up to the cube of  an inverse of $\Im(z)$ by this quantity, and thus we get almost everywhere (in $t$ independent of $z$) equality of the Cauchy transforms of $X_{0}$ and $X_{t}$, giving a.e. boundedness (and equality of von Neumann algebra norms). Now we can use the weak continuity proved in theorem \ref{main} to extend boundedness everywhere.

Second, to prove that
$X_{t}^{\epsilon}\in M$, consider $S_{t}^{(i,J)}$ $1\leq i\leq N$, $J\in \{a, b\}$ a family of 
free Brownian motions, on which we extend $\delta$ by 0.
We can always write $S_{s}^{(i)}=(1-\epsilon)S_{s}^{(i,a)} +
\sqrt{1-(1-\epsilon)^{2}}S_{s}^{(i,b)}$.

We have
thus $$X_{t}=\phi_{t}(X_{0})+(1-\epsilon)\int_{0}^{t}\phi_{t-s}(\delta (X_{s})\#dS_{s}^{(a)})+\sqrt{1-(1-\epsilon)^{2}}\int_{0}^{t}\phi_{t-s}(\delta (X_{s})\#dS_{s}^{(b)}).$$

We want to prove that, if we apply $E_{a}$, the conditional expectation on the von
Neumann algebra $M^{(a)}$ generated by $M_{0}$ and $S_{s}^{(a)}$, we get~:
$$E_{a}(X_{t})=\phi_{t}(X_{0})+(1-\epsilon)\int_{0}^{t}\phi_{t-s}(\delta (E_{a}(X_{s}))\#dS_{s}^{(a)}),$$
which says nothing but by changing $S_{s}$ in $S_{s}^{(a)}$, $E_{a}(X_{t})$
is an instance of (the unique solution) $X_{t}^{\epsilon}$. {\gre As a consequence, this gives} the stated boundedness.

Since
$E_{a}(\int_{0}^{t}\phi_{t-s}(\delta (\frac{1}{z+X_{s}})\#dS_{s}^{(b)}))=0$
is a consequence of freeness between $\{S_{s}^{(a)}\}$ and
$\{S_{s}^{(b)}\}$, we just have to show several commutations of $E_{a}$ with
several operations, more precisely~: $E_{a}\phi_{t}=\phi_{t}E_{a}$, $E_{a}(.\#(S_{t}^{(a)}-S_{s}^{(a)}))=(E_{a}\overline{\o}
E_{a}(.))\#(S_{t}^{(a)}-S_{s}^{(a)})$ on $L^{2}(M_{s})$ and $E_{a}\overline{\o}
E_{a}\circ\bar{\delta }=\bar{\delta }\circ E_{a}$. With that and obvious lemmas
about stochastic integrals, we will have what we want. The first equation
is nothing but an instance of the preservation (contained in the
preliminaries of part 2.1 with this new case of zero extension) by $\Delta $ of $M^{(a)}$ (and
characterization of conditional expectation). The second is proved in using
also the characterization of conditional expectation once noted that we can
use instead of someone in $L^{2}(M^{(a)})$, someone in $L^{2}(M_{s}^{(a)}\o
M_{s}^{(a)})\#(S_{t}^{(a)}-S_{s}^{(a)})$ by  orthogonality. The third
one is verified by using the fact that
$\delta ^{*}: L^{2}(M^{(a)}\otimes M^{(a)})\rightarrow L^{2}(M^{(a)})$ (and
characterization of conditional expectation).\end{proof}

\subsection{Stationarity}

\begin{proposition}\label{stat}
Let us call $\Phi_{t}:X_{0}\in M_{0}\cap D(\bar{\delta} )\mapsto X_{t}\in M_{t}$ the previous ultramild solution of theorem \ref{main}(i) assuming $||X_{t}||_{2}=||X_{0}||_{2}$ a.e. for all $X_{0}\in M_{0}\cap D(\bar{\delta} )$. 
 Then,
$\Phi_{t}(X_{0}Y_{0})=\Phi_{t}(X_{0})\Phi_{t}(Y_{0})$ if
$X_{0},Y_{0}\in D(\bar{\delta} )\cap M_{0}$.
\end{proposition}

\begin{proof}[Proof]
Since $\Phi_{t}(X^{*})=\Phi_{t}(X)^{*}$, $\Phi_{t}(1)=1$ and $\tau$ is faithful, $D(\bar{\delta} )\cap M_{0}$ a *-algebra, it suffices to prove that for any $X_{0},Y_{0},Z_{0},T_{0}\in D(\bar{\delta} )\cap M_{0}$

\n $\tau(\Phi_{t}(X_{0})\Phi_{t}(Y_{0})\Phi_{t}(Z_{0})\Phi_{t}(T_{0}))=\tau(X_{0}Y_{0}Z_{0}T_{0})$. For notational convenience, we prove only the case $Z_{0}=T_{0}=1$ (even if this case is also a direct consequence of the assumed isometry by polarization), the general similar case being left to the reader.

Let also $\Phi_{t}^{\epsilon}:X_{0}\in M_{0}\cap D(\bar{\delta} )\mapsto
X_{t}^{\epsilon}\in M_{t}$. 

Apply Ito's formula (assumptions of lemma \ref{Ito}) to $\eta_{\alpha}(X_{t}^{\epsilon})$, and $\eta_{\alpha}(Y_{t}^{\epsilon})$ (using the result of proposition \ref{bound}
they are valued in $M$)~:
\begin{align*}&\eta_{\alpha}(X_{t}^{\epsilon})\eta_{\alpha}(Y_{t}^{\epsilon})=\eta_{\alpha}(X_{0}^{\epsilon})\eta_{\alpha}(Y_{0}^{\epsilon}) +
(1-\epsilon)\int_{0}^{t}\eta_{\alpha}^{\o}(\bar{\delta
}(X_{s}^{\epsilon}))\eta_{\alpha}(Y_{s}^{\epsilon})+\eta_{\alpha}(X_{s}^{\epsilon})\eta_{\alpha}^{\o}(\bar{\delta
}(Y_{s}^{\epsilon}))\#dS_{s}\\ &
-\frac{1}{2}\int_{0}^{t}\eta_{\alpha}(X_{s}^{\epsilon})\Delta \eta_{\alpha}(Y_{s}^{\epsilon})+\Delta \eta_{\alpha}(X_{s}^{\epsilon})\eta_{\alpha}(Y_{s}^{\epsilon})ds\\ & + (1-\epsilon)^{2}\sum_{i=1}^{N}
\int_{0}^{t}m\circ(1\otimes(\tau\circ m)\otimes1)(
\eta_{\alpha}^{\o}(\bar{\delta _{i}}(X_{s}^{\epsilon}))\o\eta_{\alpha}^{\o}(\bar{\delta _{i}}(Y_{s}^{\epsilon})))ds.\end{align*}


We can now use lemma \ref{delta1} to get~:
\begin{align*}&\eta_{\alpha}(X_{t}^{\epsilon})\eta_{\alpha}(Y_{t}^{\epsilon})=\eta_{\alpha}(X_{0}^{\epsilon})\eta_{\alpha}(Y_{0}^{\epsilon})\\ &+
(1-\epsilon)\int_{0}^{t}\bar{\delta }(\eta_{\alpha}(X_{s}^{\epsilon})\eta_{\alpha}(Y_{s}^{\epsilon}))\#dS_{s}
-\frac{1}{2}\int_{0}^{t}\Delta^{1}(\eta_{\alpha}(X_{t}^{\epsilon})\eta_{\alpha}(Y_{t}^{\epsilon}))ds \\ &+ ((1-\epsilon)^{2}-1)\sum_{i=1}^{N}
\int_{0}^{t}m\circ(1\otimes(\tau\circ m)\otimes1)(
\bar{\delta _{i}}(\eta_{\alpha}(X_{s}^{\epsilon}))\o\bar{\delta _{i}}(\eta_{\alpha}(Y_{s}^{\epsilon})))ds\\ &+
(1-\epsilon)\int_{0}^{t}\left[(\eta_{\alpha}^{\o}(\bar{\delta
}(X_{s}^{\epsilon}))-\bar{\delta
}(\eta_{\alpha}(X_{s}^{\epsilon})))\eta_{\alpha}(Y_{s}^{\epsilon})+\eta_{\alpha}(X_{s}^{\epsilon})(\eta_{\alpha}^{\o}(\bar{\delta
}(Y_{s}^{\epsilon}))-\bar{\delta
}(\eta_{\alpha}(Y_{s}^{\epsilon})))\right]\#dS_{s}\\ &+ (1-\epsilon)^{2}\sum_{i=1}^{N}
\int_{0}^{t}m\circ(1\otimes(\tau\circ m)\otimes1)(
(\eta_{\alpha}^{\o}(\bar{\delta_{i}
}(X_{s}^{\epsilon}))-\bar{\delta_{i}
}(\eta_{\alpha}(X_{s}^{\epsilon})))\o\eta_{\alpha}^{\o}(\bar{\delta _{i}}(Y_{s}^{\epsilon})))ds\\ &+ (1-\epsilon)^{2}\sum_{i=1}^{N}
\int_{0}^{t}m\circ(1\otimes(\tau\circ m)\otimes1)(
\bar{\delta _{i}}(\eta_{\alpha}(X_{s}^{\epsilon})))\o(\eta_{\alpha}^{\o}(\bar{\delta_{i}
}(Y_{s}^{\epsilon}))-\bar{\delta_{i}
}(\eta_{\alpha}(Y_{s}^{\epsilon}))))ds.\end{align*}

Using once again the trick of proposition \ref{SImpM} to pass to something
which looks like a mild
solution, then we can take the limit $\alpha\rightarrow\infty$ as in Proposition \ref{bound} and finally we get (using that $\Phi_{t}^{\epsilon}(X_{0}Y_{0})$ is a mild solution since $X_{0}Y_{0}\in D(\bar{\delta})\cap M$)~:
\begin{align*}\Phi_{t}^{\epsilon}&(X_{0})\Phi_{t}^{\epsilon}(Y_{0})-\Phi_{t}^{\epsilon}(X_{0}Y_{0})=
(1-\epsilon)\int_{0}^{t}\phi_{t-s}(\bar{\delta }(\Phi_{s}^{\epsilon}(X_{0})\Phi_{s}^{\epsilon}(Y_{0})- \Phi_{s}^{\epsilon}(X_{0}Y_{0}))\#dS_{s})
\\ &+ ((1-\epsilon)^{2}-1)\sum_{i=1}^{N}
\int_{0}^{t}\phi_{t-s}\left(m\circ(1\otimes(\tau\circ m)\otimes1)(
\bar{\delta _{i}}(\Phi_{s}^{\epsilon}(X_{0}))\o\bar{\delta _{i}}(\Phi_{s}^{\epsilon}(Y_{0}))\right)ds.\end{align*}

 Since $\Phi_{t}^{\epsilon}(X_{0})$
converges in $||.||_{2}$-norm to  $\Phi_{t}(X_{0})$ (a.e) by the last statement of proposition \ref{bound}, we can show that, after taking the trace, this equation converges to the relation $\tau(\Phi_{t}(X_{0})\Phi_{t}(Y_{0}))=\tau(X_{0}Y_{0})$, using also  the fact that  the
last term  goes to zero via (\ref{epsisom}) as in proposition \ref{bound} .


\end{proof}

\section{Applications}
\subsection{Free Difference Quotient}
\begin{corollary}\label{MicrostateChap2}
Assume assumption 1 and $X_{1},...,X_{\red n}\in D(\Delta)\cap M_{0}$. Then, for any $t\geq 0$, there exists an embedding $\Phi_{t}:M_{0}=W^{*}(X_{1},...,X_{n})\rightarrow M_{0}*L(F(\infty))$ and $S_{1},...,S_{N}\in L(F(\infty))$ a free $(0,1)$-semicircular family (depending on t), free from $M_{0}$ and such that~:
$$||\Phi_{t}(X_{j})-X_{j}-\sqrt{t}\sum_{i=1}^{N}\bar{\partial_{i}}(X_{j})\#S_{i}||_{2}\leq c_j\ t,$$

for a fixed constant $$c_j^{2}=\frac{1}{4}||\Delta(X_{j})||_{2}^{2}+\frac{1}{2}\left( ||\Delta^{\o1/2}(\delta(X_{j}))||_{2}^{2}+\frac{\pi}{4}||\Delta(X_{j})||_{2}||||\Delta^{\o1/2}(\delta(X_{j}))||_{2}\right).$$ Moreover, $\Phi_{t}(X_{j})\in W^{*}(X_{1},...,X_{\red n},S_{1},...,S_{N},\{S_{j}'\}_{j=0}^{\infty})$ where $\{S_{j}'\}_{j=0}^{\infty}$ is a free semicircular family free with $\{X_{1},...,X_{n},S_{1},...,S_{N}\}.$

As a consequence, if we define $c^2=\sum c_j^2$, we have the following inequality for the Wasserstein-Biane-Voiculescu distance (\cite{BV01})~:
$$d_{W}(\mu_{X_{1},...,X_{\red n}},\mu_{X_{1}+\sqrt{t}\delta(X_{1})\#S,...,X_{\red n}+\sqrt{t}\delta(X_{\red n})\#S})\leq c\ t.$$

As another consequence, using \cite[Theorem {\gre 16}]{S07}, any $R^\omega$-embeddable von Neumann algebra generated by $X_1,...,X_{\red n}$ with Lipschitz conjugate variable have $\delta_0(  X_1,...,X_{\red n})={\red n  }$.
\end{corollary}

\begin{remark}
  This result is analogous to proposition 2 in \cite{S07} and to an inequality in \cite{BV01}. But the latter is for the free difference quotient for ${\red n}=1$ with only finite Fisher information. And the former deals with  any derivation, for a general ${\red n}$, assuming $\partial(X_{j})$ and $\partial^{*}\partial(X_{j})$ can be written in terms of non-commutative power series. Compared to these results, our result can be applied for a general ${\red n}$ but for coassociative {\gre (or even as we will see elsewhere also ``almost coassociative")} derivations, and for the free difference quotient with only the assumption Lipschitz conjugate variable ( i.e. $\bar{\partial}\partial_{j}^{*}1\o1\in (M\overline{\o}M^{op})^{{\red n}}$, which corresponds to Lipschitz conjugate variable in the ${\red n}=1$ case, cf. also \cite{Weav} for a more general justification of this terminology.) Note also that, in this case, the constant is expressed in terms of free Fisher information $\Phi^{*}(X_{1},...,X_{n})=\sum_{i} ||\Delta(X_i)||_2^2$, it  becomes the expected $c=\Phi^{*}(X_{1},...,X_{n}))^{1/2}/2$, so that for instance if $X_{1},...,X_{{\red n}}$ is such that the associated Orstein-\"{U}hlenbeck process $Y_{i}(t)=e^{-t/2}X_{i}+(1-e^{-t})^{1/2}S_{i}$  satisfy $X_{1}(t),...,X_{n}(t)$ have Lipschitz conjugate variable (in the above sense for all $t>0$, which is by no means a trivial assumption) then the argument of \cite{BV01} gives the corresponding free Talagrand transportation cost inequality~:
\begin{align*}d_{W}((X_{1},...,X_{N})&,(S_{1},...,S_{N}))\leq \\ &\sqrt{2}\left( \chi^{*}(S_{1},...,S_{n})-\chi^{*}(X_{1},...,X_{n})-\frac{n}{2}+\frac{1}{2}\sum_{i=1}^{N}\tau(X_{i}^{2})\right)^{1/2}.\end{align*}
We prove in \cite{Dab10} this result in full generality using another way of solving stochastic differential equations.

We give a concrete non-trivial example of Lipschitz conjugate variable in subsection 4.3.
\end{remark}

\begin{proof}[Sketch of Proof]
For the reader's convenience, we outline how this follows from the beginning of the paper.
Using Assumption 1, theorem \ref{main2} gives the conditions to apply theorem \ref{main} (ii) with $\omega=0$. Then $\Phi_t(X)=X_t$ is given by the mild solution of the SDE from (ii) and the stated inequality is the one coming from (i) in theorem \ref{main} (the inequality on Wasserstein distance is then an obvious consequence, note that $\delta(X_0)\in D((\Delta\o1 +1\o \Delta)^{1/2})$ follows from lemma \ref{in2deltaDelta} (iii) as in the proof of assumption 1 $h$). The fact that $\Phi_t$ gives a $*$-homomorphism comes from Proposition \ref{stat}. Since it preserves the trace by the SDE it satisfies, we can extend it at the von Neumann algebraic level. ($S_i$,$S_j'$ are produced from the free Brownian motion of the SDE).
Assumption 1 is true in case of Lipschitz conjugate variable as follows. First, the free difference quotient being coassociative, (a) is true in choosing non-commutative polynomials as $D(\partial)$. 
(a') is true because having Lipschitz conjugate variables imply the conjugate variables are in $M$ (using e.g. the equality (1) in \cite{nonGam}). 
(b) is valid directly by Lipschitz conjugate variable assumption. 


As stated, the equality on microstate free entropy dimension then comes from \cite[Theorem {\gre 16}]{S07}.
\end{proof}

\subsection{Preliminaries and relations of three natural derivations on $q$-Gaussian factors}
Our goal is to study three derivations on $q$-Gaussian factors : the free difference quotient, the commutator with right creation operators and the one giving the number operator as generator of the associated Dirichlet form. Especially, we want to find values of $q$'s for which they can be seen as closed derivations with value in the coarse correspondence, with the same domain and equivalent norms.

We will use this preliminaries to apply our results {\gre in the next subsection and give an interesting example of Lipschitz conjugate variables.}

\subsubsection{Preliminaries on $q$-Gaussian factors}
We recall the construction of $q$-Gaussian variables given by
Bo\.zejko and Speicher in \cite{BS91}.

Let $N<\infty$ be an integer, $\H=\R^{N}$, $\H_C=\C^{N}$ its complexification, and $-1<q<1$. Consider
the vector space

$$F_{alg}(\H)=\C\Omega \oplus \bigoplus _{n\geq 1}\H_C^{\otimes n}$$
(algebraic direct sum and tensor products). This vector space is endowed with a positive
definite inner product  given by
\begin{align*}
\langle\xi_{1}\otimes\cdots\otimes\xi_{n},\zeta_{1}\otimes\cdots\otimes\zeta_{m}\rangle_q&=\delta_{n=m}\sum_{\pi\in S_{n}}q^{i(\pi)}\prod_{j=1}^{n}\langle\xi_{j},\zeta_{\pi(j))}\rangle\\ &=\delta_{n=m}\langle\xi_{1}\otimes\cdots\otimes\xi_{n},P_q^{(n)}\zeta_{1}\otimes\cdots\otimes\zeta_{n}\rangle_0,
\end{align*}
where $i(\pi)=\#\{(i,j):i<j\ \textrm{and }\pi(i)>\pi(j)\}$, and $P_q^{(n)}=\sum_{\pi\in S_{n}}q^{i(\pi)}\pi$ where $\pi$ acts via $\pi^{-1}(\zeta_{1}\otimes\cdots\otimes\zeta_{n})=\zeta_{\pi(1)}\otimes\cdots\otimes\zeta_{\pi(n)}$. Denote by $F_{q}(\H)$
the completion of $F_{alg}(\H)$ with respect to this inner product.

For $h\in \H$, define $\ell (h):F_{q}(\H)\to F_{q}(\H)$ by extending
continuously the map\begin{eqnarray*}
\ell (h)h_{1}\otimes \cdots \otimes h_{n} & = & h\otimes h_{1}\otimes \cdots \otimes h_{n},\\
\ell (h)\Omega  & = & h.
\end{eqnarray*}
The adjoint is given by\begin{eqnarray*}
\ell ^{*}(h)h_{1}\otimes \cdots \otimes h_{n} & = & \sum _{k=1}^{n}q^{k-1}\langle h_{k},h\rangle h_{1}\otimes \cdots \otimes \hat{h}_{k}\otimes \cdots \otimes h_{n},\\
\ell ^{*}(h)\Omega  & = & 0,
\end{eqnarray*}
where $\hat{\cdot }$ denotes omission. $\omega(h)=\ell (h)+\ell ^{*} (h)$ are q-Gaussian variables. $\Gamma_q(\H)$ is the von Neumann algebra generated by $\omega(h) $ $h\in \H$, acting as bounded operators on $F_{q}(\H)$. We use on it the faithful trace $\tau_q(X)=\langle X\Omega, \Omega\rangle$. It is well-known that $L^2(\Gamma_q(\H),\tau_q)\simeq F_{q}(\H)$. For $\xi\in F_{alg}(\H)$ we write $\psi(\xi)$ the element in $\Gamma_q(\H)$ such that $\psi(\xi)\Omega=\xi$, associated to this identification (since it is easy to see that $F_{alg}(\H)\subset F_{q}(\H)$ is identified with a subspace of $\Gamma_q(\H)\subset L^2(\Gamma_q(\H),\tau_q)$ corresponding to polynomials in $\omega(h)$'s).

Consider also $r(h)$ given by\begin{eqnarray*}
r(h)h_{1}\otimes \cdots \otimes h_{n} & = & h_{1}\otimes \cdots \otimes h_{n}\otimes h\\
r(h)\Omega  & = & h.
\end{eqnarray*}
Finally, let $P_{n}:F_{q}(\H)\to F_{q}(\H)$ be the orthogonal projection
onto tensors of rank $n$. Let $\Xi _{q}=\sum _{n\geq 0}q^{n}P_{n}$. It is obvious that $\Xi _{q}$ is an Hilbert-Schmidt operator as soon as $q^2N<1$. We also introduce a natural finite rank approximation $\Xi_{q}^{Q}=\sum_{n=0}^Qq^nP_n$.

\subsubsection{Three natural derivations on $q$-Gaussian factors}

Fix an orthonormal basis $\{h_{i}\}_{i=1}^{N}\subset \R^{N}$
and let $X_{i}=\omega(h_{i})$. Thus $\Gamma_q(\H)=W^{*}(X_{1},\ldots,X_{N})$, $N=\dim \H_{\mathbb{R}}$.  We may also write for $\underline{i}=(i_1,...,i_n)\in \N^n$ $\psi_{\underline{i}}=\psi(h_{i_1}\o ...\o h_{i_n})$. Finally, for {\red a} von Neumann algebra $M$, $M^{op}$ will be as usual the opposite algebra. Later, I will consider $M=\Gamma_q(\H)$.

The following lemma is proven in \cite{Shly03}(and stated exactly in that way in \cite[Lemma {\gre 10}]{S07}
).

\begin{lemma}
\label{qcommutators}\cite{Shly03} For $j=1,\ldots,N$, $q^2N<1$, let
$\partial_{j}^{(q)}:\mathbb{C}\langle X_{1},\ldots,X_{N}\rangle\to HS$ be the
derivation given by $\partial_{j}^{(q)}(X_{i})=\delta_{i=j}\Xi _{q}=[X_i,r(h_j)]=[r(h_j)^*,X_i]$. Let $\partial:\mathbb{C}\langle X_{1},\ldots,X_{N}\rangle\to HS^{N}$
be given by $\partial^{(q)}=\partial_{1}^{(q)}\oplus\cdots\oplus\partial_{N}^{(q)}$
and regard $\partial$ as an unbounded operator densely defined on
$L^{2}(\Gamma_q(\H))$. Then:\\
(i) $\partial^{(q)}$ is closable. \\
(ii) If we denote by $Z_{j}$ the vector $0\oplus\cdots\oplus P_{\Omega}\oplus\cdots\oplus0\in HS^{N}$
(nonzero entry in $j$-th place, $P_{\Omega}$ is the orthogonal projection
onto $\mathbb{C}\Omega\in F_{q}(\H)$), then $Z_{j}$ is in the domain
of $\partial^{*}$ and $\partial^{(q)*}(Z_{j})=h_{j}$.\\
(iii)$1\o \tau(\partial_{j}^{(q)}(X))=\partial_{j}^{(q)}(X)\Omega=r(h_j)^*(X.\Omega)$ (in the first equality we identify isometrically $HS$ with $L^2(\Gamma_q(\H)\o\Gamma_q(\H)^{op})$ as usual via $a\o b$ with the rank one operator $a\tau(b.)$)\end{lemma}

Let us recall the following crucial result of Bo\.zejko (\cite{Boz}) giving an Haagerup like inequality for q-Gaussian variables.

\begin{theorem}\label{HBIn}(Haagerup-Bo\.zejko Inequality \cite{Boz})
If $C_q^{-1}=\prod_{m=1}^{\infty}(1-q^m)$ then for any $\xi\in \H^{\otimes n}\subset F_{alg}(\H)$ :
$$||\psi(\xi)||_{L^2(\Gamma_q(\H),\tau_q)}\leq ||\psi(\xi)||_{\Gamma_q(\H)}\leq C_{|q|}^{3/2}(n+1)||\psi(\xi)||_{L^2(\Gamma_q(\H),\tau_q)}.$$

Moreover, for any $\eta\in \H^{\otimes n}\o\H^{\otimes m}\subset F_{alg}(\H)\o_{alg}F_{alg}(\H)$ ($\epsilon$ either $op$ or nothing)
$$
||\psi\o\psi(\eta)||_{\Gamma_q(\H)\overline{\o}\Gamma_q(\H)^{\epsilon}}\leq C_{|q|}^{3}(n+1)(m+1)||\psi\o\psi(\eta)||_{L^2(\Gamma_q(\H)\o\Gamma_q(\H)^{op},\tau_q\o \tau_q)}.$$

\end{theorem}
A short proof of the first part can be found in \cite{Nou} (basically a variant of \cite{Boz} without writing the computations), the argument obviously giving the second part too. {\red Alternatively, as pointed out by our referee, we can apply to $u_i=\psi$ (and a variant with right multiplication in the case $\epsilon=op$) the following fact. If $u_i:H_i\to B(K_i)$ are  bounded maps from Hilbert spaces to bounded maps on a Hilbert space (nothing but trilinear forms on Hilbert spaces), then their tensor product $u_1\o u_2$ is bounded from $H_1\o H_2$ to $B(K_1\o K_2)$ with $||u_1\o u_2||\leq ||u_1||\ ||u_2||$.}

From now on, $\psi$ may not be written explicitly, no more than identifications between $L^2(\Gamma_q(\H)\o\Gamma_q(\H)^{op})$ and Hilbert-Schmidt operators (following section 2, but here the adjoint being the one coming from $\Gamma_q(\H)\o\Gamma_q(\H)^{op}$ if not specified explicitly).

As a consequence, for any $\xi \in  \bigoplus_{p\leq n}\H^{\otimes p}$ of component $\xi_p$, we also have by Cauchy-Schwarz~:
\begin{align}\label{HBInsum}\begin{split}||\psi(\xi)||_{\Gamma_q(\H)}&\leq C_{|q|}^{3/2} \sum (p+1)||\psi(\xi_p)||_{L^2(\Gamma_q(\H),\tau_q)}\\ &\leq C_{|q|}^{3/2} (n+1)^{3/2}(\sum ||\psi(\xi_p)||_{L^2(\Gamma_q(\H),\tau_q)}^2)^{1/2}\\ &=C_{|q|}^{3/2} (n+1)^{3/2}||\psi(\xi)||_{L^2(\Gamma_q(\H),\tau_q)}.\end{split}
\end{align}
Likewise, for any $\eta \in  \bigoplus_{p+q\leq n}\H^{\otimes p}\o\H^{\otimes q}$, 
we also have
~:
\begin{align}\label{HBInsum2}\begin{split}||\psi(\eta)||_{\Gamma_q(\H)\overline{\o}\Gamma_q(\H)^{op}}&\leq 
C_{|q|}^{3} (n+1)^{3}||\psi(\eta)||_{L^2(\Gamma_q(\H)\o\Gamma_q(\H)^{op},\tau_q\o \tau_q)}.\end{split}
\end{align}


In order to state the next result, let us fix several notation about tensor products (similar to those of \cite{Vo2} 3.1). $M$ is a given finite $II_1$ factor with faithful normal trace $\tau$.  $M\hat{\o} M^{op}$ is the projective tensor product of $M$ with its opposite algebra, with the corresponding $*$-Banach algebra structure. Let $\alpha : M\hat{\o} M^{op}\rightarrow B(M)$ be the contractive homomorphism given by $\alpha(a\o b)=L_{a}R_{b}$, where $L_{a}$ and $R_b$ are respectively the left and right multiplication operators by $a$ and $b$. We will denote $LR(M)$ the algebra $\alpha( M\hat{\o} M^{op})$. It is easily seen that $||\alpha(x)m||_p\leq ||x||_{ M\hat{\o} M^{op}} ||m||_p$, for $1\leq p\leq \infty$ so that $LR(M)$ acts in a bounded way on $L^{p}(M,\tau)$ (the completion of $M$ with respect to $||x||_p=\tau(|x|^p)^{1/p}$). Consistently with our previous notation, we will write $x\# m$ any of those actions (and several others we are about to discuss). For $p=2$, this gives a map $\beta : LR(M) \rightarrow C^{*}(M,M')$ where $M$, and $M'$ are with respect to the standard form of $M$ on $L^2(M)$. Further, we have a $*$-homomorphism $\gamma:  C^{*}(M,M')\to M\overline{\o} M^{op}$ with value in the von Neumann algebra tensor product given by the general $C^*$ tensor product theory. We will of course see $M\overline{\o} M^{op}$ as a $II_1$ factor with canonical trace $\tau\o\tau$. Finally, we will write $\#$ any "side multiplication" when defined. For instance,  $a\o b\#a'\o b' \#a''\o b''=aa'a''\o b''b' b$ so that $\#$  may be in this case multiplication in $M\overline{\o} M^{op}$, or any of its induced actions on $L^{2}(M\o M^{op})$. More generally, for $i\in [1,p-1]$, $a_i,b_j\in M$, we write $$(a_1 \o a_2 \o... \o a_p)\#_i(b_1 \o ...\o b_n)=a_1\o...\o a_ib_1\o b_2\o ...\o b_na_{i+1}\o ...\o a_p$$ (if $p=2,\ \#_1=\#$), and likewise the corresponding extension for instance $M^{\overline{\o}i}\hat{\o}M^{\overline{\o}p-i}\times M^{\overline{\o}n} \to M^{\overline{\o}n+p-2}$ (or any analogues containing $M^{op}$ the multiplication being then consistently defined to get what expected above in $M$ as if there where everywhere $M$, for instance if $a\o a', c\o c'\in M\o M^{op}$, $b\o b'\o b''\in (M\o M) \hat{\o} M^{op}$ we have $(a\o a')\# ((b\o b' \o b'')\#_2 (c\o c'))=((a\o a')\# (b\o b' \o b''))\#_2 (c\o c')=(ab\o b'c\o c'b''a')\in M\o M \o M^{op}$ (all multiplications written in $M$, if we {\red were} more consistent with $M^{op}$ we would have written $a'b''c'$). However, we won't use this notation if $c\o c'$ is thought of in $M \o M$, but everything would be the same if also $b\o b'\o b''\in (M\o M^{op}) \hat{\o} M^{op}$ except for the value in this space in $M\o M^{op} \o M^{op}$). 

We will often use the following assumption and give an easy sufficient condition deduced from Bo\.zejko inequality in the next corollary.

\begin{minipage}{15,5cm}\textbf{Assumption $I_{q}$}~:\textit{$q\sqrt{N}<1$ and $\Xi_{q}$ is invertible in $M\overline{\o} M^{op}$.}
 \end{minipage}

Even if we will scarcely use it, for $R>1$ and a non-commutative power series (of radius of convergence larger than {\gre $R$} with value in a tensor product) $F(Y_1,...,Y_n)=\sum a_{i_1,...,i_n,p} Y_{i_1}...Y_{i_p}\o Y_{i_{p+1}}...Y_{i_n}$ we write  the usual norm $||F||_R=\sum |a_{i_1,...,i_n,p}| R^n$. We will use the same notation with less or more tensors in the space of value.

\begin{corollary}\label{qproj}
When the right hand side in the inequalities bellow is finite, $\Xi_{q}$ comes from an element in $M\hat{\o} M^{op}$, and  respectively $M\overline{\o} M^{op}$ via $\iota\gamma\beta\alpha$ or $\iota: M\overline{\o} M^{op}\to L^{2}(M\overline{\o} M^{op})$ and with an obvious notation:
$$||\Xi_{q}-1\o 1||_{M\hat{\o} M^{op}}\leq (C_{|q|})^3 \left[ \frac{4|q|N}{1-|q|N}+\frac{5(|q|N)^2}{(1-|q|N)^2}+\frac{2(|q|N)^3}{(1-|q|N)^3}\right]=:\nu(q,N),$$
$$||\Xi _{q}-1\o 1||_{M\overline{\o} M^{op}}\leq (C_{|q|})^3 \left[ \frac{4|q|\sqrt{N}}{1-|q|\sqrt{N}}+\frac{5(|q|\sqrt{N})^2}{(1-|q|\sqrt{N})^2}+\frac{2(|q|\sqrt{N})^3}{(1-|q|\sqrt{N})^3}\right]=:\rho(q,N).$$
Especially ($N\geq 2$) if q is such that $\nu(q,N)<1$, e.g. for $|q|N\leq 0.13$, then $\Xi_{q}$ is invertible in $M\hat{\o} M^{op}$ (resp. if q is such that $\rho(q,N)<1$ e.g. when $|q|\sqrt{N}\leq 0.13$ then $I_{q}$ holds). 
Moreover, if $q\sqrt{N}<1$, $||\Xi_{q}^{Q}-\Xi_{q}||_{M\overline{\o} M^{op}}\to_{Q\to\infty}0$ and $\Xi _{q}\in C^{*}(X_1\o 1,1\o X_1,...1\o X_N)\subset M\overline{\o}M^{op}$ is positive so that $\Xi_{q}^{1/2}$ is well defined.

Moreover, if $\epsilon>0$ and $(({\gre3}+\epsilon)^2N+2)|q|<1$ there exists a non-commutative power series $\Xi_q(Y_1,...,Y_N)$ with radius of convergence greater than $R=(1+\epsilon{\gre /2})\frac{2}{1-|q|}>||X_i||$ such that $\Xi_q(X_1,...,X_N)=\Xi_q$, and 
$$||\Xi_q-1\o 1||_R\leq \frac{({\gre3}+\epsilon)^2N|q|}{1-(2+({\gre3}+\epsilon)^2N)|q|}=:\pi(q,N),$$
and likewise, \begin{align*}\max({\gre \sum_{i=1}^N}||\partial_i\o 1(\Xi_q)||_R,\sum_{i=1}^N||1\o\partial_i(\Xi_q)||_R) 
\leq \frac{({\gre3}+\epsilon)^2N|q|(1-2|q|)}{(1-(2+({\gre3}+\epsilon)^2N)|q|)^{2}}.
\end{align*}
\end{corollary}

\begin{proof}
Since $P_{n}$ can be seen as a finite rank operator written as $\sum_{\xi}\xi \o \xi^*$ with the usual identification (the sum running over an orthonormal basis of $\H^{\otimes n}$), the previous theorem gives :$||P_{n}||_{M\hat{\o} M^{op}}\leq \sum_{\xi}||\xi||^2\leq C_{|q|}^3(n+1)^2\sum_{\xi}||\xi||_2^2=C_{|q|}^3(n+1)^2N^n$. The inequality follows from a standard computation.

Likewise \begin{align*}||\Xi _{q}-1\o 1||_{M\overline{\o} M^{op}}&\leq \sum_{n\geq 1} q^n ||\sum_{\xi}\xi \o \xi^*||\\&\leq C_{|q|}^3\sum_{n\geq 1} q^n (n+1)^2||\sum_{\xi}\xi \o \xi^*||_2=C_{|q|}^3\sum_{n\geq 1} q^n (n+1)^2N^{n/2}.\end{align*}

Let us call $f(|q|N)=\nu(q,N)/(C_{|q|})^3$. To get $f(|q|N)<C_{|q|}^{-3}$, it suffices to have $f(|q|N)<(1+|q|)^3\prod_{m=1}^\infty(1-|q|^m)^3/(1+|q|^m)^3={\gre (1+|q|)^3}(\sum_{n\in \Z}(-1)^n|q|^{n^2})^3$, and again keeping only the smallest order it suffices to have $f(|q|N)<(1-|q|-2|q|^2)^3$ and solving numerically $f(|q|N)<(1-|q|N/2-|q|^2N^2/2)^3$ (sufficient since $N\geq 2$) one gets $|q|N<0.1386...$. 

For the last statement, we only improve an estimate in \cite{S07}. We write $p_{\underline{i}}$ the polynomials giving, by evaluation on $X_1,...,X_N$, the orthonormalization of $\psi_{\underline{i}}$ defined in lemma {\gre 13} in \cite{S07}. More specifically, we consider $\Gamma_{n}$ the Gramm matrix of  q-scalar products in the space of tensors of length $n$ given (for $|\underline{j}|=|\underline{l}|=n$) by : $(\Gamma_{n})_{(\underline{j},\underline{l})}=\langle\psi_{j_{1},\ldots,j_{n}},\psi_{l_{1},\ldots,l_{n}}\rangle_q$. This is an $N^n\times N^n$ matrix known to be positive and invertible (with real coefficients)
, and we consider $B=\Gamma_n^{-1/2}$. Note that by definition {\gre $\Gamma_n$} is given by the image of the element of $P_q^{(n)}=\sum_{\pi\in S_{n}}q^{i(\pi)}\pi$ in the algebra of the symmetric group $S_{n}$ by the obvious representation $\pi_{q,N,n}$ of $S_{n}$ on (the formal basis of the $\C^{N^n}$) $\xi_{\underline{l}}$ and  it is known from \cite{ND93} and \cite{Zagier} a formula for $P_q^{(n)-1}$ given by the inductive relation $P_q^{(n)}=\pi_{n-1,n}(P_q^{(n-1)})M_n$,$\pi_{n-1,n}$ the usual embedding of $S_{n-1}$ in $S_n$ with image leaving $1$ invariant and $M_n=\sum_{k=1}^nq^{k-1}(1\rightarrow k)$ (with the notation $(k\rightarrow l)$ the cycle sending $k+i$ to $k+i+1$ for $0\leq i\leq l-k-1$ and sending $l$ to $k$) via $M_n^{-1}=\prod_{j=n-1}^1(1-q^j(1\rightarrow j+1){\gre )}\prod_{j=n-2}^0(1-q^{n-j}(2\rightarrow n-j))^{-1}$. We will use it through $B^2=\pi_{q,N,n}(P_q^{(n)-1})$. We also write $\psi_{ \underline{j}}(Y_1,...,Y_N)$ the non-commutative polynomial defined inductively by ($\psi_{\epsilon}=1$ for the empty word $\epsilon$):
\begin{align}\label{psirec}\psi_{i_{1},\ldots,i_{n}}=Y_{i_{1}}\psi_{i_{2},\ldots,i_{n}}-\sum_{j\geq2}q^{j-2}\delta_{i_{1}=i_{j}}\psi_{i_{2},\ldots,\hat{i_{j}},\ldots,i_{n}}.\end{align}

As in the proof of proposition 2.7 in \cite{BKS}, we use the following identity for $\psi_{\underline{i}}=\psi_{\underline{i}}(X_1,...,X_N)$ for $\psi_{\underline{i}}$ introduced before.

Then, by definition, $p_{\underline{i}}(Y_1,...,Y_N)=\sum_{\underline{j},|\underline{j}|=n}B_{\underline{i},\underline{j}}\psi_{\underline{j}}(Y_1,...,Y_N)$ so that (as checked in lemma {\gre 13} in \cite{S07}) $\{p_{\underline{i}}(X_1,...X_N)\Omega\}_{|\underline{i}|=n}$ is obviously an orthonormal basis of $\H^{\o n}$. 

 $\Xi_q(Y_1,...,Y_N)=\sum_n q^n \sum_{\underline{i}}p_{\underline{i}}(Y_1,...,Y_N)\o p_{\underline{i}}^*(Y_1,...,Y_N)$ will be the power series we are looking for, once proved an estimate on its norm. It suffices to bound (using symmetry of the matrix B) : \begin{align*}||\sum_{\underline{i}}p_{\underline{i}}(Y_1,...,Y_N)\o p_{\underline{i}}^*(Y_1,...,Y_N)||_{R}&=||\sum_{\underline{i},\underline{j},\underline{l}}B_{\underline{i},\underline{j}}\psi_{\underline{j}}(Y_1,...,Y_N)\o B_{\underline{i},\underline{l}}\psi_{\underline{l}}^*(Y_1,...,Y_N)||_{R}
 \\ & \leq \sum_{\underline{l}}||\sum_{\underline{j}}B^2_{\underline{l},\underline{j}}\psi_{\underline{j}}(Y_1,...,Y_N)||_{R}||\psi_{\underline{l}}^*(Y_1,...,Y_N)||_{R}.\end{align*}

Now using the expression for $B^2$ expanded from the inverse coming from the action of the symmetric group algebra, it involves only $|| \psi_{\underline{\sigma(j)}}(Y_1,...,Y_N)||_{R}$ and from the bound in \cite{ND93} lemma 4.1, one gets $$||\sum_{\underline{j}}B^2_{\underline{l},\underline{j}}\psi_{\underline{j}}(Y_1,...,Y_N)||_{R}\leq \left((1-|q|)\prod_{k=1}^{\infty}\frac{1+|q|^k}{1-|q|^k}\right)^n\sup_{\sigma\in S_n}|| \psi_{\sigma(\underline{j})}(Y_1,...,Y_N)||_{R}.$$

{\gre Finally, using (\ref{psirec}), if we call $C_n=\sup_{i_1,...,i_n} ||\psi_{i_1,...,i_n}||_R$, then $C_n\leq R C_{n-1}+ C_{n-2}/(1-|q|)$, thus $C_n\leq (R+\frac{1}{1-|q|})^n$.
Likewise, if $D_n=\sup_{i_1,...,i_n}\sum_i||\partial_i \psi_{i_1,...,i_n}(Y_1,...,Y_N)||_{R}$ then $D_n\leq C_{n-1}+ R D_{n-1}+D_{n-2}/(1-|q|)$
thus one checks by induction :$D_n\leq n(R+\frac{1}{1-|q|})^{n-1}$.
}

Finally, we proved :\begin{align*}||\sum_{\underline{i}}p_{\underline{i}}(Y_1,...,Y_N)\o p_{\underline{i}}^*(Y_1,...,Y_N)||_{R}&\leq \left((1-|q|)\prod_{k=1}^{\infty}\frac{1+|q|^k}{1-|q|^k}\right)^n(R+\frac{\gre 1}{1-|q|})^{2n}N^n
\\ &\leq \left(\frac{(1-|q|)^2}{1-2|q|}\right)^n(\frac{({\gre 3+}\epsilon)}{1-|q|})^{2n}N^n.\end{align*}
The last rough estimate is as above, in this proof, for the estimate on $f(|q|N)$ and detailed in lemma {\gre 13} in \cite{S07}, and it concludes.
Likewise, we have : 
 \begin{align*}\sum_j||\sum_{\underline{i}}\partial_jp_{\underline{i}}(Y_1,...,Y_N)\o p_{\underline{i}}^*(Y_1,...,Y_N)||_{R}
\leq n\left(\frac{(1-|q|)^2}{1-2|q|}\right)^n(\frac{({\gre3}+\epsilon)}{1-|q|})^{2n}N^n.
\end{align*}
Note that positivity comes from the identification of $\sum_n q^{n}  P_n=\Gamma_{q}(qId)$ with the second quantization.
\end{proof}

As a consequence, for $q$ such that $I_q$ holds (e.g. $\rho(q,N)<1$), if $\partial_j$ is  the $j$-th free difference quotient with respect to $X_{1},\ldots,X_{N}$ we have $\partial_j=\partial_{j}^{(q)}\#\Xi _{q}^{-1}$  since $\partial_{j}(X_{i})=1_{i=j}\Xi _{q}\#\Xi _{q}^{-1}=\delta_{i=j}1\o 1$. Recall $\#$ in this context is multiplication in $M\overline{\o} M^{op}$.

Finally, we want to introduce a derivation giving the number operator as generator of the corresponding  Dirichlet form. We define first the $\hat{\partial}_{j}^{(q)}:=\partial_{k}\#X_{q}^{k'}$ valued in $\Gamma_q(\H\oplus \H)$ where $X_q^{k'}$ is the q-Gaussian variable corresponding to the second copy of the eigenvector $h_k$ in the second term of the direct sum. Said otherwise this is the only derivation sending $X_q^{k}$ to $X_q^{k'}$. This derivation is defined for any $q$. We also want to compare this derivation to another derivation valued in the coarse correspondence. For $q$ such that $q\sqrt{N}<1$, we can define $\tilde{\partial}_{j}^{(q)}:=\partial_{j}\#\Xi _{q}^{1/2}$.


\begin{proposition}\label{p33}
For any $\xi\in \H^{\o n},\eta\in \H^{\o m}$, any $q\in (-1,1)$, $\psi(\xi)\in D(\hat{\partial}_{k}^{(q)})$ and we have~:
\begin{align*}\sum_k\langle\hat{\partial}_{k}^{(q)}(\psi(\xi)),\hat{\partial}_{k}^{(q)}(\psi(\eta))\rangle_q &=n\langle\xi,\eta\rangle_q.
\end{align*}
As a consequence, $\hat{\partial}^{(q)}=(\hat{\partial}_{1}^{(q)},...,\hat{\partial}_{n}^{(q)})$ is a closable derivation, with $\hat{\partial}^{(q)*}\hat{\partial}^{(q)}=\tilde{\Delta}$ the number operator satisfying $\tilde{\Delta}(\xi)=n\xi,\xi\in \H^{\o n}$.

Moreover if $q\sqrt{N}<1$
, for any polynomial $P,Q,R,S\in \C\langle X_1,...,X_n\rangle$,
$$\langle R\hat{\partial}_{k}^{(q)}(P),S\hat{\partial}_{k}^{(q)}(Q)\rangle=\langle R\tilde{\partial}_{k}^{(q)}(P),S\tilde{\partial}_{k}^{(q)}(Q)\rangle.$$

Thus, one can see $\hat{\partial}_{k}^{(q)}$ as valued in a bimodule included in the coarse correspondence.

Finally, if $I_q$ holds, 
$\tilde{\partial}_{k}^{(q)}$, $\partial_{k}^{(q)}$, 
 and $\partial_{k}$ are all closable and their closures share the same domain, with, for any $x$ in their common domain :
$$||\partial_{k}^{(q)}(x)||_2\leq ||\Xi_q^{1/2}||_{M\overline{\o} M^{op}}||\tilde{\partial}_{k}^{(q)}(x)||_2\leq ||\Xi_q^{1/2}||_{M\overline{\o} M^{op}}^2||\partial_{k}(x)||_2$$
$$||\partial_{k}(x)||_2\leq ||\Xi_q^{-1/2}||_{M\overline{\o} M^{op}}||\tilde{\partial}_{k}^{(q)}(x)||_2\leq ||\Xi_q^{-1/2}||_{M\overline{\o} M^{op}}^2||\partial_{k}^{(q)}(x)||_2.$$
\end{proposition}

\begin{proof}
The domain property stated is obvious since $\psi(\xi)$ is a non-commutative polynomial in $X_1,...,X_N$.
Moreover, by linearity, we need to check the first equality only for $\psi(\xi)=\psi_{j_{1} \ldots,j_{n}}$ and $\psi(\eta)=\psi_{l_{1} \ldots,l_{p}}$.

As in the proof of proposition 2.7 in \cite{BKS}, we use the following identity :

$$\psi_{i_{1},\ldots,i_{n}}=X_{i_{1}}\psi_{i_{2},\ldots,i_{n}}-\sum_{j\geq2}q^{j-2}\delta_{i_{1}=i_{j}}\psi_{i_{2},\ldots,\hat{i_{j}},\ldots,i_{n}}.$$
Applying $\partial_k$, we find :

$$\partial_k(\psi_{i_{1},\ldots,i_{n}})=1_{i_{1}=k}\o \psi_{i_{2},\ldots,i_{n}}+X_{i_{1}}\partial_k(\psi_{i_{2},\ldots,i_{n}})-\sum_{j\geq2}q^{j-2}\delta_{i_{1}=i_{j}}\partial_k(\psi_{i_{2},\ldots,\hat{i_{j}},\ldots,i_{n}}).$$

As a consequence, we deduce by an immediate induction :
$$\partial_k(\psi_{i_{1},\ldots,i_{n}})\#X_{q}^{k'}=\sum_j 1_{i_{j}=k}
\psi_{i_{1},\ldots i_{j}' \ldots,i_{n}}$$ (where the prime indicates we have to consider the $i_j$ of the second copy of $\H$).

We can thus compute (using the definition of the scalar product in the second and fourth lines, and removing properly summations and Kronecker functions $1_{a=b}$ in the third and fifth lines):

\begin{align*}\sum_k\langle\hat{\partial}_{k}^{(q)}(\psi_{j_{1} \ldots,j_{n}}),\hat{\partial}_{k}^{(q)}(\psi_{l_{1}\ldots l_{m}})\rangle_q &=\sum_k \sum_{i,\iota}1_{j_{i}=k=l_{\iota}}\langle\psi_{j_{1},\ldots j_{i}' \ldots,j_{n}},\psi_{l_{1},\ldots l_{\iota}' \ldots,l_{m}}\rangle_q
\\&=1_{n=m}\sum_k \sum_{i,\iota}1_{j_{i}=k=l_{\iota}}\sum_{\pi\in S_{n}}q^{i(\pi)}1_{\gre\pi(i)=\iota}\prod_{p=1}^{n}1_{j_p=l_{\gre\pi(p)}}
\\&=1_{n=m}\sum_k \sum_{i}1_{j_{i}=k}\sum_{\pi\in S_{n}}q^{i(\pi)}\prod_{p=1}^{n}1_{j_p=l_{\gre\pi(p)}}
\\&=1_{n=m}\sum_k \sum_{i}1_{j_{i}=k}\langle\psi_{j_{1},\ldots,j_{n}},\psi_{l_{1},\ldots,l_{m}}\rangle_q
\\&=n\langle\psi_{j_{1},\ldots,j_{n}},\psi_{l_{1},\ldots,l_{m}}\rangle_q.
\end{align*}

We now assume $q\sqrt{N}<1$.
To explain the second equality, note that we can rewrite (since $\Xi_q$ self-adjoint) $\langle a\o b\#\Xi_q^{1/2},a'\o b'\#\Xi_q^{1/2}\rangle=\langle a\o b,a'\o b'\#\Xi_q\rangle$ and then : $$\langle a\o b,a'\o b'\#\Xi_q\rangle=\sum_n q^{n}\sum_{\xi} \tau(a^*a'\xi)\tau(\xi^*b'b^*)=\sum_n q^{n} \tau(a^*a' P_n(b'b^*))=\tau(a^*a'\Gamma_{q}(qId)(b'b^*)),$$

where $\Gamma_{q}(qId)$ is the second quantization. Then, our claim follows for instance from Theorem 3.2 in \cite{qBrowDon} which implies $\tau(a^*a'\Gamma_{q}(qId)(b'b^*))= \langle a\o b\#X_k',a'\o b'\#X_k'\rangle$ (Theorem 3.2 is a variant of Ito formula, one can apply it after identifying the first copy of $\H$ with $span \{\sqrt{n}1_{[k/2n,(k+1)/2n)}, k=1,...,n\}$ in $L^2([0,1])$ and the second with $span \{\sqrt{n}1_{[k/2n,(k+1)/2n)}, k=n+1,...,2n\}$).

The last inequalities in the proposition on non-commutative polynomials follow from corollary \ref{qproj} and assumption $I_q$. It implies closability since $\partial^{(q)}$ is closable (by lemmas \ref{triv} and \ref{qcommutators}) and the result extended to the closures.


\end{proof}

\begin{remark}\label{qcmap}
Even if we won't use significantly later the analytic bound we got in Corollary \ref{qproj}, it is worth noting it can enable us using our last derivation $\tilde{\partial}^{(q)}$ to prove complete metric approximation property for $\Gamma_{q}(\H)$ with small $q$, or (reprove) absence of non-trivial projections for the corresponding $C^*$-algebras following the lines of \cite{SG}. Indeed, first note that using the analytic expansion $\Xi_q^{1/2}(Y_1,...,Y_N)=1\o1+\sum_{k=1}^{\infty}{1/2 \choose k}(\Xi_q(Y_1,...,Y_N)-1\o1)^k$ so that we get a Lipschitz bound \begin{align*}||\Xi_q^{1/2}&(Y_1,...,Y_N)-\Xi_q^{1/2}(Z_1,...,Z_N)||\\ &\leq \sum_{k=1}^{\infty}\left|{1/2 \choose k}\right|k||\Xi_q-1\o1||_{R}^{k-1}\sum_i(||\partial_i\o 1(\Xi_q)||_R+||1\o\partial_i(\Xi_q)||_R)||X_i-Z_i||
\\ &\leq \frac{1}{2\sqrt{1-||(\Xi_q-1\o1)||_R}}\sup_i(||X_i-Z_i||)2\frac{(3+\epsilon)^2N|q|(1-2|q|)}{(1-(2+(3+\epsilon)^2N)|q|)^{2}}
\\ &\leq \frac{1}{\sqrt{1-(2+2(3+\epsilon)^2N)|q|}}\sup_i(||X_i-Z_i||)\frac{(3+\epsilon)^2N|q|(1-2|q|)}{(1-(2+(3+\epsilon)^2N)|q|)^{3/2}}\\ &\leq \kappa\sup_i(||X_i-Z_i||),\end{align*}
the last inequality being true for $\kappa<1/2$ if
$(3+\epsilon)^2N|q|(1-2|q|)\leq (3+\epsilon)^2N|q|<(1-2(2+2(3+\epsilon)^2N)|q|)/2\leq(1-(2+2(3+\epsilon)^2N)|q|)^{2}/2$, i.e. e.g. for $(4+6(3+\epsilon)^2N)|q|<1.$

As in \cite{SG}, one can consider the solutions (given by Picard iteration) $X_{i,t}=X_i-\frac{1}{2}\int_0^tds X_{i,s}+\int_0^t\Xi_q^{1/2}(X_{1,s},...,X_{N,s})\#dS_s^i$, $Y_{i,t}=0-\frac{1}{2}\int_0^tds Y_{i,s}+\int_0^t\Xi_q^{1/2}(Y_{1,s},...,Y_{N,s})\#dS_s^i$. From \cite{S07} or \cite{Dab10} (and the above proposition \ref{p33}), $X_{i,t}$ is stationary so that $\alpha_t(X_i)=X_{i,t}$ ($i=1,...,N$) defines a trace preserving homomorphism. By variation of constants, one gets: \begin{align*}(X_{i,t}-Y_{i,t})&=X_i-\frac{1}{2}\int_0^tds (X_{i,s}-Y_{i,s})+\int_0^t(\Xi_q^{1/2}(X_{1,s},...,X_{N,s})-\Xi_q^{1/2}(Y_{1,s},...,Y_{N,s}))\#dS_s^i\\ &=e^{-1/2t}X_i+\int_0^te^{-1/2(t-s)}(\Xi_q^{1/2}(X_{1,s},...,X_{N,s})-\Xi_q^{1/2}(Y_{1,s},...,Y_{N,s}))\#dS_s^i.
\end{align*}

And from  our inequality above and Biane-Speicher's $L^\infty$ version of Burkholder-Gundy inequality,  we deduce : $$\sup_i||X_{i,t}-Y_{i,t}||\leq e^{-t/2}\sup_i||X_i||+\left(\int_0^tds\  e^{-(t-s)}\kappa^2 \sup_i||X_{i,s}-Y_{i,s}||^2\right)^{1/2}.$$
so that from Gronwall's lemma (in line 2 after using a trivial bound on squares and $\kappa<1/2$)~:  
\begin{align*}\sup_i||X_{i,t}-Y_{i,t}||^2&\leq 2e^{-t}\sup_i||X_i||^2+1/2\left(\int_0^tds\ e^{-(t-s)} \sup_i||X_{i,s}-Y_{i,s}||^2\right)\\ &\leq 2e^{-t/2}\sup_i||X_i||^2\to 0.\end{align*}

Thus, since $Y_{i,s}\in C^{*}(S_t^i)$ we got the property of corollary 4.1 in \cite{SG} and by the reasoning of Theorem 4.2 there, $C^*(X_1,...,X_N)$ has no non-trivial projections (remember this applies when  $(4+6(3+\epsilon)^2N)|q|<1$). Likewise, by the reasoning of theorem 4.3 in \cite{SG} we get complete metric approximation property in the way they get Haagerup property. {\red This last result has been recently extended by Stephen Avsec  \cite{Avsec} to all $q\in(-1,1)$. Of course, in the smaller range of $q$ we consider we have almost inclusion in $L(\F_{_\infty})$ too.} 

\end{remark}

\subsubsection{Regularity for $\Xi_q$} 

Let us write $\partial_{i}^{(k,j)}=1^{\o(k-1)}\o\partial_i\o 1^{\o(j)} :L^{2}(M)^{\o (k+j)}\to L^{2}(M)^{\o (k+j+1)}$ and the corresponding $L^2$ closure $\overline{\partial_{i}^{(k,j)}}$.
We start by noting the following consequence of proposition \ref{p33} 
~:
\begin{lemmas}
If $I_q$ holds and for $\xi \in \H^{\o n}$, for $\DD$ any among $\overline{\partial_{m_p}^{(k_p,p-k_p)}}\circ\ldots\circ \overline{\partial_{m_1}^{(k_1,1-k_1)}}$  $p\in [1,n], k_l\in [1,l],m_l\in[1,N],l=1...p$, $||\DD(\xi)||_2^2\leq (n||\Xi_q^{-1}||_{M\overline{\o}M^{op}})^p||\xi||_2^2$.
\end{lemmas}

\begin{lemma}\label{Claim2}
Assume $I_q$ and $|q|N<1$, then $||\partial_{k}\o 1\Xi_{q}||_{(M\overline{\o}M^{op})\hat{\o}M}<\infty$. Likewise, with $U,V$ any among $\partial_j,\partial^{(q)}_i$, we have  $U\o V(\Xi_q)\in(M\overline{\o} M^{op})\hat{\o}(M^{op}\overline{\o} M)$, $U\o V(\Xi_q)\in (M^{op}\overline{\o} M)\hat{\o}(M\overline{\o} M^{op})$,  $(U\o 1\o1)(V\o 1)(\Xi_q)\in (M\overline{\o} M^{op}\overline{\o} M)\hat{\o}M^{op}$, $(1\o1\o U)(1\o V)(\Xi_q)\in M\hat{\o}(M^{op}\overline{\o} M\overline{\o} M^{op})$.

\end{lemma}

\begin{proof}[Proof]
We compute (the first inequality bellow is obvious from lemma \ref{qproj}, the second equality comes from proposition \ref{p33}) :
$$||\partial_{k}(p_{\underline{i}})||_{L^2(M,\tau_q)\o L^2(M,\tau_q)}^2\leq ||\Xi_q^{-1/2}||^2\langle\partial_{k}(p_{\underline{i}})\#\Xi_q,\partial_{k}(p_{\underline{i}})\rangle=||\Xi_q^{-1/2}||^2|i|.$$
Now, one can use Theorem \ref{HBIn} and (\ref{HBInsum2})  in the second line and our previous inequality in the third to conclude to the first result.
\begin{align*}||\partial_{k}\o 1\Xi_{q}&||_{(M\overline{\o}M^{op})\hat{\o}M}\leq \sum_n q^n \sum_{|\underline{i}|=n}||\partial_{k}(p_{\underline{i}})||_{M\overline{\o}M^{op}}||p_{\underline{i}}||_M\\ &\leq C_{|q|}^{9/2}\sum_n |q|^n \sum_{|\underline{i}|=n}(n+1)^4||p_{\underline{i}}||_2||\partial_{k}(p_{\underline{i}})||_{L^2(M\overline{\o}M)}
\\ &\leq C_{|q|}^{9/2}\sum_n |q|^n (n+1)^4||\Xi_q^{-1/2}||n^{1/2}N^{n}. 
\end{align*}
As stated, as soon as $||\Xi_q^{-1/2}||<\infty$ and $|q|N<1$, this sum is finite. The proof of $\partial_j\o\partial^{(q)}_i(\Xi_q)\in(M\overline{\o} M^{op})\hat{\o}(M^{op}\overline{\o} M)$ is really similar. For the last statement, we need likewise a bound e.g. on $||\partial_j\o 1\partial^{(q)}_i(p_{\underline{j}})||_{M\overline{\o} M^{op}\overline{\o} M}\leq ||\partial_j\o 1(\Xi_q)||_{(M\overline{\o} M^{op})\hat{\o} M}||\partial_i(p_{\underline{j}})||_{M\overline{\o} M}+\|\Xi_q\|_{M^{op}\overline{\o} M}||\partial_j\o 1\partial_i(p_{\underline{j}})||_{M\overline{\o} M^{op}\overline{\o} M}$ coming from the previous lemma (with a 3 tensor product variant of Bo\.zejko inequality and lemma \ref{p33}).

\end{proof}

\subsection{An example of Lipschitz conjugate variable: q-Gaussian families for small $q$}

We now want to play with the three previous derivations to get regularity results for conjugate variables.

\begin{theorem}\label{mainq}
 Assume $\rho(q,N)<1$ as defined in corollary \ref{qproj} (e.g. $|q|\sqrt{N}\leq 0.13$) and $|q|N<1$ then q-Gaussian variables have finite free Fisher information $\Phi^{*}(X_1,...,X_N)<\infty$ (and actually the conjugate variable is in the domain of the $L^2$-closure of the free difference quotient).
 
Furthermore assume also condition $\nu(q,N)<1$ in corollary \ref{qproj} (e.g. $|q|N\leq 0.13$), in that case the conjugate variables are in $\Gamma_q(\H)$ and $X_1,...,X_N$ have even Lipschitz conjugate variables. As a consequence, under condition $\nu(q,N)<1$ we have $\delta_{0}(X_1,...,X_N)=N$.

{\re Finally, if we assume $\pi(q,N)<1$, then there exists a non-commutative power series $\xi_j$ of radius $R=(1+\epsilon{\gre /2})\frac{2}{1-|q|}>||X_i||$ such that $\xi_j(X_1,...,X_N)$ are the conjugate variables of $X_1,...,X_N$. Moreover, there exits a self-ajoint potential $V$ which is also a non-commutative power series of radius $R$ such that its cyclic gradient is $D_iV=\xi_i$.} 
\end{theorem}
\begin{remark}
In \cite{S07}, Shlyakhtenko proved $\delta_{0}(X_1,...,X_N)\to N$ when $q\to0$, we can prove this value is identically equal to $N$ on a small neighborhood of 0. Actually, he proved $\delta_{0}(X_1,...,X_N)\geq N\left(1 - \frac{q^2N}{1-q^2N}\right)$ for $|q|<(4N^3+2)^{-1}$. Here the improvement in terms of value of $\delta_0$ mainly comes from using a better derivation in that respect (the free difference quotient). The improvement in terms of values of $q$ comes from the fact we only need a Lipschitz condition instead of a analyticity condition on the conjugate variable. However, in considering like us the free difference quotient and with a better estimate of the domain of analyticity, one would also get a range of order $|q|<1/CN$ in inverse of the number of generators (with a huger $C$ than ours, cf. Rmk \ref{qcmap}). 
{\red Note finally that corollary 2.11 in \cite{Shly03} implies $\delta^{*}(X_1,...,X_N)=N$ as soon as $I_q$ holds, thus e.g. assuming only $\rho(q,N)<1$.}
\end{remark}

\begin{proof}[Proof]
Let $M=\Gamma_q(\H)$. 
\setcounter{Step}{0}

\begin{step}Finite Fisher Information under $|q|N<1$ and $\rho(q,N)<1$
\end{step}
Recall the notation introduced before (and in) Corollary \ref{qproj} so that $\iota\gamma\beta\alpha$ is the natural map from $M\hat{\o}M^{op}$ to $L^2(M\o M^{op})$ (we may use later implicitly).

\begin{claim}\label{qdom}$\iota\gamma\beta\alpha(M\hat{\o}M^{op})\subset D(\partial_{j}^{(q)*})$ and for any $a,b\in M$ $$\partial_{j}^{(q)*}(a\o b)=aX_jb-r(h_j)^*(a)b-a(l(h_j)^*(b)).$$
\end{claim}
\begin{proof}[Proof of Claim]
As reminded in lemma \ref{qcommutators}, $1\o\tau\partial_{j}^{(q)}=r(h_j)^*$. Moreover, since $\partial_{j}^{(q)}$ is a real derivation for any $x\in D(\partial_{j}^{(q)})$, we have $1\o\tau\partial_{j}^{(q)}(x^*)=(\tau\o 1(\partial_{j}^{(q)}(x))^*$. Thus if $J$ denotes the antilinear isometry extending $J(x)=x^*$ to $L^2(M)$, we have $\tau\o 1\partial_{j}^{(q)}=J1\o\tau\partial_{j}^{(q)}J=Jr(h_j)^*J=l(h_j)^*$. The last equality follows from formulas for annihilation operators and $J\psi_{i_1,...,i_n}=\psi_{i_n,...,i_1}$.

From lemma \ref{triv} and $\partial_{j}^{(q)*}(1\o 1)=X_j$, one deduces for $a,b\in D(\partial_{j}^{(q)})\cap M$, $\partial_{j}^{(q)*}(a\o b)=aX_jb-r(h_j)^*(a)b-a(l(h_j)^*(b))$, so that 
\begin{align*}||\partial_{j}^{(q)*}(a\o b)||_2&\leq ||a|| ||b|| ||X_j||+||r(h_j)||_{B(L^2(M){\gre)}}||a||_2||b||+||l(h_j)||_{B(L^2(M){\gre)}}||b||_2||a||\\ &\leq 4||a|| ||b||/\sqrt{1-|q|}.\end{align*}

Now for any $a,b\in M$, if $\eta_{\alpha}=\alpha(\alpha+\partial^{(q)*}\partial^{(q)})^{-1}$ the completely positive (thus contractive on $M$) resolvent associated to the generator of the corresponding Dirichlet form, we have for any $x\in M,  \ \eta_{\alpha}(x)\in D(\partial_{j}^{(q)})\cap M$ and $||\eta_{\alpha}(x)-x||_2\to 0$ when $\alpha\to \infty$. 
Since $||\partial_{j}^{(q)*}(\eta_{\alpha}(a)\o \eta_{\alpha}(b))||_2\leq 4||a|| ||b||/\sqrt{1-|q|}$ we have weak convergence in $L^2$ up to extraction and as $\eta_{\alpha}(a)\o \eta_{\alpha}(b)\to a\o b \in L^2(M\o M)$, we get  $ a\o b$ in the domain of the closed operator  $\partial_{j}^{(q)*}$ with the formula and inequality above remaining true. This concludes.


\end{proof}

Note that assuming $\nu(q,N)<1$, one thus deduces $\Xi _{q}^{-1}\in D(\partial_{j}^{(q)*})$ with the formula : $$\partial_{j}^{(q)*}(\Xi _{q}^{-1})=\Xi _{q}^{-1}\#X_j-m(r(h_j)^*\o 1 + 1 \o l(h_j)^*)(\Xi _{q}^{-1}).$$

Since $(\Xi_q)^*=\Xi_q \in M\overline{\o} M^{op}$, we have thus shown our first result about finite Fisher information in this case.


First recall $\{h_{i}\}_{i=1}^{N}\subset \R^{N}$ is an orthonormal basis.  We write for $\underline{i}=(i_1,...,i_n)\in N^n$ $\psi_{\underline{i}}=\psi(h_{i_1}\o ...\o h_{i_n})$. We define the length $| \underline{i}|=n$.

We now want to prove finite Fisher information under the less restrictive condition $\rho(q,N)<1$, $|q|N<1$. We need to show $\Xi_q^{-1}\in D(\partial^{(q)*}_i)$ and we only know from lemma \ref{Claim2} : $\Xi_q\in M\hat{\o}M^{op}, \partial_k\o 1\Xi_q\in (M\overline{\o} M^{op})\hat{\o}M, (1\o\partial_k)\Xi_q\in M^{op}\hat{\o}(M\overline{\o} M^{op})$,$(1\o\partial_i^{(q)})\Xi_q\in M\overline{\o}M\overline{\o} M^{\gre op}$,
$(\partial_i^{(q)}\o 1)\Xi_q\in M\overline{\o}M^{op}\overline{\o} M^{op}$, $\Xi_q^{-1}\in M\overline{\o} M^{op}$,$(\partial_j\o\partial^{(q)}_i)(\Xi_q)\in(M\overline{\o} M^{op})\hat{\o}(M^{op}\overline{\o} M)$, $(\partial^{(q)}_i\o\partial_j)(\Xi_q)\in(M^{op}\overline{\o} M)\hat{\o}(M\overline{\o} M^{op})$, $((\partial_j\o 1\o1\partial^{(q)}_i)\o 1)(\Xi_q)\in (M\overline{\o} M^{op}\overline{\o} M)\hat{\o}M^{\gre op}$,$(1\o((1\o1\o\partial_j)\partial^{(q)}_i))(\Xi_q)\in M\hat{\o}(M^{op}\overline{\o} M
\overline{\o} M^{op})$ (the norms of those quantities bellow are always taken in those spaces if not otherwise specified). 

Let us call $U_n=\sum_{i=0}^n (-1)^i (\Xi_q-1\o 1)^i$ (power in $M\hat{\o} M^{op}$) so that we know $U_n\to\Xi_q^{-1}$ in $L^2$, $U_n\in M\hat{\o}M^{op}$ and by our first claim 
 $U_n\in D(\partial^{(q)*}_i)$. Since $\partial^{(q)*}_i$ is closed it suffices to show $\partial^{(q)*}_i(U_n)$ bounded in $L^2$ to get a weak limit up to extraction and $\Xi_q^{-1}\in D(\partial^{(q)*}_i)$ and to get also $\Xi_q^{-1}\in D(\partial_j\partial^{(q)*}_i)$, it suffices to bound 
$\partial_j\partial^{(q)*}_i(U_n)$ (since such a bound gives also a bound on $||\partial^{(q)*}_i(U_n)||_2^2= \langle \partial^{(q)}_i\partial^{(q)*}_i(U_n), U_n\rangle$, we only sketch the proof of  both at once).

This is mainly a computation using $U_n$ is almost an inverse and thus will behave almost as inverse when computing derivatives coming from application of $\partial$. The second key point will be that, apart from a bunch of terms we can gather in something of the form $\partial^{(q)*}_i(U_n)$, the $\partial_j$ will enable us to use only a bound on terms coming from $U_n$ in von Neumann norm. Recall notation $\#_i$ was introduced before Corollary \ref{qproj}. We get (after using our formula for $\partial^{(q)*}_i$, we mainly use derivation property of $\partial_j$ and changes of summation)~:

\begin{align*}
&\partial_j\partial^{(q)*}_i(U_n)=\partial_j(U_n\# X_i -m\circ (1\o \tau\o 1)(\partial^{(q)}_i\o 1(U_n)+1\o \partial^{(q)}_i(U_n)))
\\ &\partial_j(U_n\# X_i)=1_{i=j}U_n+\sum_{i=1}^{n}(-1)^i\sum_{k=0}^{i-1} (\Xi_q-1\o 1)^{k}\#\left(\partial_j\o 1(\Xi_q)\#_2((\Xi_q-1\o 1)^{i-k-1}\# X_i)\right)\\&+\sum_{i=1}^{n}(-1)^i\sum_{k=0}^{i-1}(\Xi_q-1\o 1)^{k}\#\left( 1\o\partial_j(\Xi_q)\#_1((\Xi_q-1\o 1)^{i-k-1}\# X_i)\right)\\ &=1_{i=j}U_n -\\ &\sum_{k=0}^{n-1} (-1)^k(\Xi_q-1\o 1)^{k}\#\left(\partial_j\o 1(\Xi_q)\#_2(U_{n-k-1}\# X_i)+ 1\o\partial_j(\Xi_q)\#_1(U_{n-k-1}\# X_i)\right).
\end{align*}
\begin{align*}&\partial_j(m\circ (1\o \tau\o 1)(\partial^{(q)}_i\o 1(U_n))\\&= -\partial_j \mt[\sum_{k=0}^{n-1} (-1)^{k}(\Xi_q-1\o 1)^{k}\#\left(\partial^{(q)}_i\o 1(\Xi_q)\#_2(U_{n-k-1})\right)]\\ &=-\sum_{k=0}^{n-1} (-1)^{k}\times \\&\left\{(1\o(\mt))\left[(\partial_j\o1(\Xi_q-1\o 1)^{k})\#_{2}\left(\partial^{(q)}_i\o 1(\Xi_q)\#_2(U_{n-k-1})\right)\right]\right.\\ &+((\mt)\o 1)\left[(1\o \partial_j(\Xi_q-1\o 1)^{k}\#_{1}\left(\partial^{(q)}_i\o 1(\Xi_q)\#_2(U_{n-k-1})\right)\right]\\ &+((\mt)\o 1)(\Xi_q-1\o 1)^{k}\#\left(\partial^{(q)}_i\o 1(\Xi_q)\#_2(1\o \partial_j (U_{n-k-1}))\right)\\ &+ (\Xi_q-1\o 1)^{k}\#\left[(1\o(\mt))(\partial_j\o 1\o 1\partial^{(q)}_i\o 1(\Xi_q)\#_3 U_{n-k-1})\right]\\ &\left.+ (\Xi_q-1\o 1)^{k}\#\left[((\mt)\o 1)(\partial^{(q)}_i\o \partial_j(\Xi_q)\#_2 U_{n-k-1})\right]\right\}.
\end{align*}
Preparing for the reintroduction of $\partial^{(q)*}_i(U_{n-k-1})$ we rewrite (a part of) the first line in our last right hand side :
\begin{align*}&\sum_{k=0}^{n-1} (-1)^{k}\left[(\partial_j\o1(\Xi_q-1\o 1)^{k})\#_{2}\left(\partial^{(q)}_i\o 1(\Xi_q)\#_2(U_{n-k-1})\right)\right]\\ &=\sum_{k=0}^{n-1} (-1)^{k}\times\\&\left[\sum_{l=0}^{k-1} (\Xi_q-1\o 1)^{l}\#(\partial_j\o1(\Xi_q))\#_2(\Xi_q-1\o 1)^{k-l-1}\right]\#_{2}\left(\partial^{(q)}_i\o 1(\Xi_q)\#_2(U_{n-k-1})\right)\\ &=\sum_{l=0}^{n-2} (-1)^{l}\left((\Xi_q-1\o 1)^{l}\#(\partial_j\o1(\Xi_q)\right)\#_2\\ &\ \ \ \ \ \ \ \ \ \ \ \left[\sum_{k=l+1}^{n-1}(-1)^{k-l} (\Xi_q-1\o 1)^{k-l-1}\#\left(\partial^{(q)}_i\o 1(\Xi_q)\#_2(U_{n-k-1})\right)\right]
\\ &=\sum_{l=1}^{n-1} (-1)^{l}\left((\Xi_q-1\o 1)^{l}\#(\partial_j\o1(\Xi_q)\right)\#_2(\partial^{(q)}_i\o 1(U_{n-l-1})).
\end{align*}
(in the last  line, note that the term with $l=n-1$ is zero since $\partial^{(q)}_i\o 1(U_{0})=0$);

We will now write $\tilde{\tau}=\mt$.
Putting everything together and reintroducing in the last line $\partial^{(q)*}_i(U_{n-k-1})$ when useful in the right hand side :

\begin{align*}& \partial_j\partial^{(q)*}_i(U_n)\\&=1_{i=j}U_n+\sum_{k=0}^{n-1} (-1)^{k}\left\{(\Xi_q-1\o 1)^{k}\#(\tilde{\tau}\o 1)\left(\partial^{(q)}_i\o 1(\Xi_q)\#_2(1\o \partial_j (U_{n-k-1}))\right)\right.\\ &
+ (\Xi_q-1\o 1)^{k}\#(1\o\tilde{\tau})\left(1\o\partial^{(q)}_i(\Xi_q)\#_1( \partial_j\o 1 (U_{n-k-1}))\right)\\ &+ (\Xi_q-1\o 1)^{k}\#(1\o\tilde{\tau})\left(\partial_j\o 1\o 1\partial^{(q)}_i\o 1(\Xi_q)\#_3 U_{n-k-1}+\partial_j\o\partial^{(q)}_i(\Xi_q)\#_2 U_{n-k-1}\right)\\ &+(\Xi_q-1\o 1)^{k}\#(\tilde{\tau}\o 1)\left(1\o 1\o\partial_j 1\o\partial^{(q)}_i(\Xi_q)\#_1 U_{n-k-1}+\partial^{(q)}_i\o \partial_j(\Xi_q)\#_2 U_{n-k-1}\right)\\ &
\left.-(\Xi_q-1\o 1)^{k}\#\left(\partial_j\o 1(\Xi_q)\#_2\partial^{(q)*}_i(U_{n-k-1})+ 1\o\partial_j(\Xi_q)\#_1\partial^{(q)*}_i(U_{\gre n-k-1})\right)\right\}.
\end{align*}

We can now deduce from this a bound for $p\in [2,\infty]$ in $L^{p}(M\o M^{op})$ if we know a bound on $||\partial^{(q)*}_i(U_{k})||_p$. Under the assumption $|q|N<1$ we know this is finite for $p=2$, we will use it later in the case $p=\infty$ under a stronger assumption. (the second line bellow corresponds to the last line of our last equation, the first and third to the first and second,  the fourth and fifth to the third and fourth).
\begin{align*} &
||\partial_j\partial^{(q)*}_i(U_n)||_p\leq 1_{i=j}||U_n||_p\\ &+\left(\sup_{k\leq n-1}||\partial^{(q)*}_i(U_{k})||_p\right)(||\partial_j\o 1(\Xi_q)||+||1\o\partial_j(\Xi_q)||)\sum_{k=0}^{n-1} ||\Xi_q-1\o 1||_{M\overline{\o}M^{op}}^{k}\\ &
+(||1\o\partial^{(q)}_i(\Xi_q)||||\partial_j\o 1(\Xi_q)||+||\partial^{(q)}_i\o 1(\Xi_q)||||1\o \partial_j(\Xi_q)||)\times \\ &\ \ \ \ \ \ \sum_{k=2}^{n} \frac{k(k-1)}{2}||(\Xi_q-1\o 1)||_{M\overline{\o}M^{op}}^{k-2}
\\ & +(||\partial_j\o\partial^{(q)}_i(\Xi_q)||+||\partial_j\o 1\o1\partial^{(q)}_i\o 1(\Xi_q)||)\sum_{k=0}^{n-1} (k+1)||(\Xi_q-1\o 1)||_{M\overline{\o}M^{op}}^{k}
\\ & +(||\partial^{(q)}_i\o \partial_j(\Xi_q)||+||1\o 1\o \partial_j1\o\partial^{(q)}_i(\Xi_q)||)\sum_{k=0}^{n-1} (k+1)||(\Xi_q-1\o 1)||_{M\overline{\o}M^{op}}^{k}.
\end{align*}
Since $||(\Xi_q-1\o 1)||_{M\overline{\o}M^{op}}<1$ all the sums of the right hand side  extended to infinity converge so that we get constants C,D $||\partial_j\partial^{(q)*}_i(U_n)||_2\leq C+ D\left(\sup_{k\leq n-1}||\partial^{(q)*}_i(U_{k})||_2\right)$ and thus $||\partial^{(q)*}_i(U_{n})||_2^2\leq ||\Xi_q||\ ||U_n||_2(C+ D\left(\sup_{k\leq n-1}||\partial^{(q)*}_i(U_{k})||_2\right))$, and a standard bound concludes to finiteness of $\sup_{k}||\partial^{(q)*}_i(U_{k})||_2$.

\begin{step}
Bounded conjugate variable under $\nu(q,N)<1$.
\end{step}
From the previous step, we know : \begin{align*}&\partial^{(q)*}_i(U_n)=U_n\# X_i+\sum_{k=0}^{n-1} (-1)^{k}(\Xi_q-1\o 1)^{k}\# \\ &\ \ \ \ \ \ \ \ \ m\circ (1\o \tau\o 1)\left(\partial^{(q)}_i\o 1(\Xi_q)\#_2(U_{n-k-1})+1\o\partial^{(q)}_i(\Xi_q)\#_1(U_{n-k-1})\right).\end{align*}

	And thus,  \begin{align*}&||\partial^{(q)*}_i(U_n)||\leq ||U_n||_{M\hat{\o} M^{op}}|| X_i||+\sum_{k=0}^{n-1} (k+1)||\Xi_q-1\o 1||_{M\hat{\o} M^{op}}^{k}\times\\ &\left(||\partial^{(q)}_i\o 1(\Xi_q)||_{(M\overline{\o} M^{op})\hat{\o} M^{op}}+||1\o\partial^{(q)}_i(\Xi_q)||_{ M\hat{\o}(M\overline{\o}M^{op})}\right)
	\\& \leq || X_i|| \sum_{k=0}^{\infty}||\Xi_q-1\o 1||_{M\hat{\o} M^{op}}^{k}+\sum_{k=0}^{\infty} (k+1)||\Xi_q-1\o 1||_{M\hat{\o} M^{op}}^{k}\times\\ &\left(||\partial^{(q)}_i\o 1(\Xi_q)||_{(M\overline{\o} M^{op})\hat{\o} M^{op}}+||1\o\partial^{(q)}_i(\Xi_q)||_{ M\hat{\o}(M\overline{\o}M^{op})}\right).\end{align*}
The last inequality gives a finite bound for $\nu(q,N)<1$ as, then, by corollary \ref{qproj}, we have $||\Xi_q-1\o 1||_{M\hat{\o} M^{op}}<1$.
Since we showed in step 1 $\partial^{(q)*}_i(U_n)\to \partial^{(q)*}_i(\Xi_q^{-1})$ weakly in $L^2$ up to extraction, this means we have ultraweak convergence of the same extraction. Thus especially $\partial^{(q)*}_i(\Xi_q^{-1})\in M$.

\begin{step}
Lipschitz conjugate variable under $\nu(q,N)<1$.
\end{step}
Since we now know  $\sup_{k}||\partial^{(q)*}_i(U_{k})||_M<\infty$ from the second step, the end of the first step gives :
$||\partial_j\partial^{(q)*}_i(U_n)||_{M\overline{\o} M^{op}}\leq C+ D\left(\sup_{k}||\partial^{(q)*}_i(U_{k})||_M\right)$. 

Again since we saw in step one : $\partial_j\partial^{(q)*}_i(U_n)\to \partial_j\partial^{(q)*}_i(\Xi_q^{-1})$ weakly in $L^2$ up to extraction, we got $\partial_j\partial^{(q)*}_i(\Xi_q^{-1})\in M\overline{\o} M^{op}$.

Putting everything together, this concludes the proof of the second part of our theorem (the statement on microstate free entropy dimension uses the $R^{\omega}$ embeddability result of \cite{Sni} and corollary \ref{MicrostateChap2}).





{\re \begin{step}
Analytic conjugate variable coming from a potential under $\pi(q,N)<1$.
\end{step}

Since by corollary \ref{qproj}, we have $||\Xi_q-1\o 1||_R<1$, we have a non-commutative power series $\Xi_q^{-1}$. If we define $$\xi_i(Y_1,...,Y_N)=\Xi_q^{-1}(Y_1,...,Y_N)\#Y_i-m\circ((1\o\tau)\partial_j^{(q)an}\o 1+1\o (\tau\o 1)\partial_j^{(q)an})(\Xi_q^{-1}(Y_1,...,Y_N)),$$

where $\partial_j^{(q)an}(P(Y))=\partial_jP(Y)\#\Xi_q(Y_1,...,Y_N)$ is the analytic version of $\partial_j^{(q)}$.
This is now obviously a power series with radius of convergence $R$ ($\tau$ here is the tracial state of $q$-gaussians), we have by the claim in step 1 $\xi_i(X_1,...,X_N)$ is the conjugate variable of q-gaussian variables $X_1,...,X_N$.

Let us define $$V(Y_1,...,Y_N)=N^{-1}\left(\frac{1}{2}\sum_{i=1}^N\xi_i(Y_1,...,Y_N)Y_i+Y_i\xi_i(Y_1,...,Y_N)\right),$$
where $N$ is the operator defined on non-commutative power series having each monomial of degree $n$ as eigenvector of eigenvalue $n$.
Obviously, $V$ is a selfadjoint potential since $Y_i$ and $\xi_i(Y)$ are self-adjoint. We have to check in the spirit of \cite{VoCG} that $D_iV=\xi_i.$ Of course this is equivalent to $D_i(NV)=(1+N)(D_iV)=(1+N)(\xi_i)=\xi_i+\sum_{j=1}^N\partial_j(\xi_i)\#Y_j$. In order to prove this, using lemma \ref{analrel} bellow, it suffices to show we have  $D_i(NV)(X_1,...,X_N)=\xi_i(X_1,...,X_N)+\sum_{j=1}^N\partial_j(\xi_i(X_1,...,X_N))\#X_j.$

But by corollary 5.12 in \cite{Vo6}, we have $NV(X_1,...,X_N)=\sum_{i=1}^N\xi_i(X_1,...,X_N)X_i.$

Note that the computation at the end of step 1, we know e.g. $\partial_i\xi_i(X_1,...,X_N)\in L^2(M)\hat{\o}L^2(M).$ Note that $D_i\xi_i(X_1,...,X_N)=m\circ flip( \partial_i\xi_i(X_1,...,X_N))$ is then defined in $L^1$ with $flip(a\o b)=b\o a.$

Applying cyclic gradients and using the relation $$D_i(PQ)=flip(\partial_i(P))\#Q+flip(\partial_i(Q))\#P,$$ we thus deduce :

$$D_i\sum_{j=1}^N\xi_j(X_1,...,X_N)X_j=\sum_{j=1}^Nflip(\partial_i(\xi_j(X_1,...,X_N))\#X_j +\xi_i(X_1,...,X_N).$$

To conclude we just have to recall $flip(\partial_i(\xi_j(X_1,...,X_N)))=\partial_j(\xi_i(X_1,...,X_N)),$ a priori in $L^2(M\o M)$ thus also in the subspace $L^2(M)\hat{\o}L^2(M).$ This follows by a duality argument in lemma \ref{graddual}.

}
\end{proof}
{\re
\begin{lemma}\label{graddual}
If $X_1,...,X_N$ have conjugate variables $\xi_1,...,\xi_N\in L^2$, then  for any $a,b\in D(\partial_j)\cap M$:
\begin{align*}\langle&\xi_i,\partial_j^*(a\o b)\rangle=\langle\xi_j,\partial_i^*(b\o a)\rangle.
\end{align*}
\end{lemma}

\begin{proof}
The result is shown, by density, for $a,b$ non-commutative polynomials in $X_1,...,X_N$, using lemma \ref{basic} and coassociativity of the free difference quotient : 
\begin{align*}\langle&\xi_i,\partial_j^*(a\o b)\rangle\\&=\tau(\xi_i^*a\xi_jb)-\tau(\xi_i^*[(1\o \tau\partial_j)(a)b+a(\tau\o 1\partial_j)(b)])
\\ &=\langle\xi_j,\partial_i^*(b\o a)\rangle+\tau(\xi_j^*[(1\o \tau\partial_i)(b)a+b(\tau\o 1\partial_i)(a)])\\ &\ \ \ -\tau\o \tau((\partial_i\o \tau)(\partial_j)(a)b)-\tau\o \tau(a(\tau\o \partial_i)\partial_j)(b))
\\ &\ \ \ -\tau\o \tau((1\o \tau\partial_j)(a)\partial_i(b))-\tau\o \tau(\partial_i(a)(\tau\o 1)\partial_j)(b))
\\ &=\langle\xi_j,\partial_i^*(b\o a)\rangle+\tau\o \tau([(\partial_j\o \tau\partial_i)(b)a+b(\tau\o \partial_j)\partial_i)(a)])\\ &\ \ \ +\tau\o \tau([(1\o \tau\partial_i)(b)\partial_j(a)+\partial_j(b)(\tau\o 1)\partial_i)(a)])\\ &\ \ \ -\tau((\tau\o((1\o \tau)\partial_j))(\partial_i)(a)b)- \tau(a( ((\tau\o 1)\partial_j)\o\tau)\partial_i)(b))
\\ &\ \ \ -\tau((1\o \tau)\partial_j(a)(1\o\tau)\partial_i(b))- \tau((\tau\o 1)\partial_i(a)(\tau\o 1)\partial_j)(b))
\\ &=\langle\xi_j,\partial_i^*(b\o a)\rangle
\end{align*}
\end{proof}
\begin{lemma}\label{analrel}
Assume $\Phi^*(X_1,..,X_n)<\infty$ then there is no non-zero non-commutative power series $P(X_1,...,X_n)$ of radius of convergence $R>||X_i||$ such that $P(X_1,...,X_n)=0$.
\end{lemma}
\begin{proof}
Since $\Phi^*(X_1,..,X_n)<\infty$, the free difference quotient is closable. As a consequence, taking a sequence of polynomials $P_n\to P$ in analytic norm, we have $P_n(X_1,...,X_n)\to 0$ in $L^2$ norm. $\partial_iP(X_1,...,X_n)$ converges to $\partial_i P(X_1,...,X_n)$ (since $||\partial_iP_n-\partial_iP||_S\to 0$ for any $S<R$), thus by closability $\partial_i P(X_1,...,X_n)=0.$ We also get vanishing of any higher order derivatives by induction.

Taking successive non-commutative derivatives and multiplying, if we assume for contradiction $P\neq 0$ one can assume $P(0,...,0)\neq 0$. Now for any non-commutative polynomial $Q$,  one has\begin{align}\label{SeriesExp}\begin{split}&Q(X_1+Y_1,...,X_n+Y_n)\\&=\sum_{k=0}^\infty \sum_{i_1,...,i_k\in[1,n]}[(\partial_{i_1}\o 1^{\o k-1})\circ ...\circ \partial_{i_k}(Q)](X_1,...,X_n)\#(Y_{i_1},...,Y_{i_k}),\end{split}\end{align}

where $(a_0\o...\o a_{k})\#(Y_{i_1},...,Y_{i_k})=a_0Y_{i_1}a_1...Y_{i_k}a_k$ and the sum over $k$ is finite here.

Let $T=\max_i (||X_i||)<R$. Consider $c_n(P)$ the sum of absolute values of coefficients of degree $n$ of $P$. Then $\phi_P(x)=\sum_{n=0}^\infty c_n(P)x^n$ is a commutative power series of radius of convergence at least $R$ and  $\sum_{i_1,...,i_k\in[1,n]}||[(\partial_{i_1}\o 1^{\o k-1})\circ ...\circ \partial_{i_k}(P)]||_T\leq \frac{1}{k!}\phi_P^{(k)}(T).$

Since $\phi_P$ is analytic in the ball of center $0$ and radius $R$, it admits a Taylor power series expansion around $T$ and as a consequence, the right hand side of (\ref{SeriesExp}) makes sense if $||Y_i||<R-T$. As a consequence approximating $P$ by polynomials, one gets $(\ref{SeriesExp})$ for $P$ and such $Y_i$'s. Applying this for $Y_i=(t-1)X_i$ one gets $P(tX_1,...,tX_n)=0$ for $t$ close to $1$ and then after iterating for $t\in[0,1]$, this contradicts $P(0,...,0)\neq 0.$
\end{proof}

}

\subsection{Group Cocycles}

Since assumption 1 is hard to verify in practice, it is interesting to work only under assumption 0, and prove directly that the ultramild solution of theorem \ref{main} (i) satisfy $||X_{t}||_{2}=||X_{0}||_{2}$ a.e. to get a stationary solution. In this part, we find a necessary and sufficient condition for derivations coming from group cocycles to get results in the spirit of Corollary 3 in \cite{S07}. 

Let $\Gamma$ be a discrete group. To a(n additive left) cocycle $c$ with value in the regular representation $c\in C^{1}(\Gamma,\ell^{2}(\Gamma))$ we associate a derivation $\delta_{c}:\C\Gamma\rightarrow \ell^{2}(\Gamma)\o\ell^{2}(\Gamma)=L^{2}(M_{0}\o M_{0})$ ($M_{0}$ the group von Neumann algebra of $\Gamma$) given by $\delta_{c}(\gamma)=B(c(\gamma))\gamma$ where $B:\ell^{2}(\Gamma)\rightarrow \ell^{2}(\Gamma)\o\ell^{2}(\Gamma)$ the isometric map given by $B(\gamma)=\gamma\o\gamma^{-1}$. Indeed, $\delta(\gamma_{1}\gamma_{2})=B(\gamma_{1}c(\gamma_{2})+c(\gamma_{1}))\gamma_{1}\gamma_{2}=\gamma_{1}B(c(\gamma_{2}))\gamma_{1}^{-1}\gamma_{1}\gamma_{2}+\delta_{c}(\gamma_{1})\gamma_{2}=\gamma_{1}\delta_{c}(\gamma_{2})+\delta_{c}(\gamma_{1})\gamma_{2}$ so that $\delta_{c}$ is a derivation with the same bimodule structure used earlier on $L^{2}(M_{0}\o M_{0})$. Moreover $\delta_{c}$ is easily seen to be a real derivation if $c$ takes values in $i\R$ (we will consider only such cocycles). Let us note that $\langle\delta_{c}(\gamma) , 1\o 1\rangle=0$ for any $\gamma$ so that we easily deduce that $\delta_{c}^{*}(1\o 1)=0$ so that $\delta_{c}$ is always closable. Any $\delta_{c_{1}},...,\delta_{c_{n}}$ therefore satisfy assumption 0. Moreover, as noted e.g. in the proof of Corollary {\gre 19} in \cite{S07},  $\langle\delta_{c}\gamma,\delta_{c}\gamma'\rangle=\delta_{\gamma}^{\gamma'}||c(\gamma)||_{2}^{2}$ so that $\delta_{c}^{*}\delta_{c}(\gamma)=||c(\gamma)||_{2}^{2}\gamma$. We now fix $c_{1},...,c_{n}$ such cocycles and {\gre write} $\delta_{i}$ the extension to $M$ of $\delta_{c_{i}}$ described at the beginning of section 2. We write $X_{t},X_{t}^{\epsilon}$ the ultramild (resp mild) solution given by theorem \ref{main} when the initial condition is $X_{0}$.

We now want to describe a first equivalent formulation of the isometry $||X_{t}||_{2}=||X_{0}||_{2}$. To this end, we want to give an equation on certain component{\gre s} of the free product $L^{2}(M)$. {\gre Let us call} $N$ the von Neumann algebra generated by free Brownian motions, it is well known that $M$ is the orthogonal direct sum of  $L^{2}(N)$ and $L^{2}(N)\gamma_{1}(L^{2}(N)\ominus\C)\gamma_{2}...(L^{2}(N)\ominus\C)\gamma_{n}L^{2}(N)$ where $\gamma_{i}$'s run over $\Gamma-\{1\}$. Since $X_{t}$ and $X_{t}^{\epsilon}$ are orthogonal to  $L^{2}(N)$ (since $\delta_{i}^{*}1\o 1=0$) we may consider only $X_{t;\gamma_{1},...,\gamma_{n}}^{\epsilon}\in L^{2}(N)\o (L^{2}(N)\ominus\C)^{n-1}\o L^{2}(N)$   such that $X_{t;\gamma_{1},...,\gamma_{n}}^{\epsilon}\#(\gamma_{1}\o...\o\gamma_{n})$ are the orthogonal projections  on those spaces. We wrote here $U\#(\gamma_{1}\o...\o\gamma_{n})$ the extension given by freeness of $(a_1\o ...\o a_{n+1})\#(\gamma_{1}\o...\o\gamma_{n})=a_1\gamma_{1}a_2...a_n\gamma_{n}a_{n+1}$.We will {\gre also write} $(a_{1}\o ...a_{i}\o a_{i+1}...\o a_{n})\#_{i}(1\o(S_{t}-S_{s})\o 1)=a_{1}\o ...a_{i}\o(S_{t}-S_{s})\o a_{i+1}...\o a_{n}$, $(a_{1}\o ...a_{i}\o a_{i+1}...\o a_{n})\#_{i}((S_{t}-S_{s})\o 1)=a_{1}\o ...a_{i}(S_{t}-S_{s})\o a_{i+1}...\o a_{n}$,$(a_{1}\o ...a_{i}\o a_{i+1}...\o a_{n})\#_{i}(1\o(S_{t}-S_{s}))=a_{1}\o ...a_{i}\o(S_{t}-S_{s})a_{i+1}...\o a_{n}$ and the obvious corresponding adapted stochastic integrals.  We now have the following :

\begin{propositions}
Assume $X_{0}=\gamma$, then :
\begin{align*}&X_{t;\gamma_{1},...,\gamma_{n}}^{\epsilon}=\delta_{n=1}\delta_{\gamma_{1}=\gamma}e^{-\frac{t}{2}\left(\sum_{j=1}^{N}||c_{j}(\gamma)||_{2}^{2}\right)}1\o 1\\ &+(1-\epsilon)\sum_{i=1}^{n}\sum_{j=1}^{N}\\ &\int_{0}^{t}\ e^{\frac{s-t}{2}\left(\sum_{i=1}^{n}\sum_{j=1}^{N}||c_{j}(\gamma_{i})||_{2}^{2}\right)}X_{s;\gamma_{1},...,\gamma_{n}}^{\epsilon}\#_{i}\left(\langle \gamma_{i}, c_{j}(\gamma_{i})\rangle 1\o dS_{s}^{(j)}+\langle 1, c_{j}(\gamma_{i})\rangle  dS_{s}^{(j)}\o 1\right)\\ &+ (1-\epsilon)\delta_{n\neq 1}\sum_{i=1}^{n-1}\sum_{j=1}^{N}\langle \gamma_{i}, c_{j}(\gamma_{i}\gamma_{i+1})\rangle\times \\ &\times\int_{0}^{t}\ e^{\frac{s-t}{2}\left(\sum_{i=1}^{n}\sum_{j=1}^{N}||c_{j}(\gamma_{i})||_{2}^{2}\right)}X_{s;\gamma_{1},...,\gamma_{i}\gamma_{i+1},...,\gamma_{n}}^{\epsilon}\#_{i}1\o dS_{s}^{(j)}\o 1,\end{align*}

which is non zero only if $\gamma_{1}...\gamma_{n}=\gamma$.  Moreover this relation with $\epsilon=0$ is thus also valid for $X_{t}$ (by the weak convergence defining it).
\end{propositions}

 As a consequence, using freeness and the definition of the space where $X_{t;\gamma_{1},...,\gamma_{n}}^{\epsilon}$ lives (especially the orthogonal complements to $\C$) we get :

\begin{align*}||&X_{t;\gamma_{1},...,\gamma_{n}}||_{2}^{2}=\delta_{n=1}\delta_{\gamma_{1}=\gamma}e^{-t\left(\sum_{j=1}^{N}||c_{j}(\gamma)||_{2}^{2}\right)}\\ &+\sum_{i=1}^{n}\sum_{j=1}^{N}\left(|\langle \gamma_{i}, c_{j}(\gamma_{i})\rangle|^{2}+ |\langle 1, c_{j}(\gamma_{i})\rangle|^{2} \right)\int_{0}^{t}ds\ e^{(s-t)\left(\sum_{i=1}^{n}\sum_{j=1}^{N}||c_{j}(\gamma_{i})||_{2}^{2}\right)}||X_{s;\gamma_{1},...,\gamma_{n}}||_{2}^{2}\\ &+ \delta_{n\neq 1}\sum_{i=1}^{n-1}\sum_{j=1}^{N}|\langle \gamma_{i}, c_{j}(\gamma_{i}\gamma_{i+1})\rangle|^{2}\int_{0}^{t}ds\ e^{(s-t)\left(\sum_{i=1}^{n}\sum_{j=1}^{N}||c_{j}(\gamma_{i})||_{2}^{2}\right)}||X_{s;\gamma_{1},...,\gamma_{i}\gamma_{i+1},...,\gamma_{n}}||_{2}^{2}.\end{align*}

As a consequence, solving the equation by variation of constants, and using the following convenient notation $||\hat{c}_{j}(\gamma_i)||_{2}^{2}=||c_{j}(\gamma_i)||_{2}^{2}-\left(|\langle \gamma_{i}, c_{j}(\gamma_{i})\rangle|^{2}+ |\langle 1, c_{j}(\gamma_{i})\rangle|^{2} \right)$, we obtain the following :

\begin{propositions}
Assume $X_{0}=\gamma$, then 
\begin{align*}||&X_{t;\gamma_{1},...,\gamma_{n}}||_{2}^{2}=\delta_{n=1}\delta_{\gamma_{1}=\gamma}e^{-t\left(\sum_{j=1}^{N}||\hat{c}_{j}(\gamma)||_{2}^{2}\right)}\\ &+ \delta_{n\neq 1}\sum_{i=1}^{n-1}\sum_{j=1}^{N}|\langle \gamma_{i}, c_{j}(\gamma_{i}\gamma_{i+1})\rangle|^{2}\int_{0}^{t}ds\ e^{(s-t)\left(\sum_{i=1}^{n}\sum_{j=1}^{N}||\hat{c}_{j}(\gamma_{i})||_{2}^{2}\right)}||X_{s;\gamma_{1},...,\gamma_{i}\gamma_{i+1},...,\gamma_{n}}^{\epsilon}||_{2}^{2}.\end{align*}
\end{propositions}

This equation is nothing but a forward Kolmogorov equation, and the question we ask is whether $1=||\gamma||_{2}^{2}=||X_{0}||_{2}^{2}=\sum_{n}\sum_{\gamma_{1},...,\gamma_{n}} ||X_{t;\gamma_{1},...,\gamma_{n}}||_{2}^{2}$, i.e. nothing but if the solution of the Kolmogorov equation is conservative. In order to state a result, let us define a corresponding continuous time Markov chain to give a probabilistic counterpart to the stationarity of $X_{t}$, using usual results on Kolmogorov equations (cf. e.g. \cite{liggett}).

\begin{notation}
Given a countable group $\Gamma$ and additive left cocycles with value in the left regular representation $c_{1},...,c_{N}$ as above. We write $M(\Gamma;c_{1},...,c_{N})$ 
 the continuous time Markov process defined on the countable state space of finite non trivial sequences valued in $\Gamma$: $F(\Gamma)=(\Gamma-\{ 1 \})^{(<\omega)}$ 
  defined by the following rates $R((\gamma_{1},...,\gamma_{n}))=\sum_{i=1}^{n}\sum_{j=1}^{N}||\hat{c}_{j}(\gamma_{i})||_{2}^{2}$, 
and with transition probabilities non zero only from $(\gamma_{1},...,\gamma_{n})$ to $(\gamma_{1},...,\delta_{i},\delta_{i}',...,\gamma_{n})$ with $\delta_{i}\delta_{i}'=\gamma_{i}$ (of course $\delta_{i},\delta_{i}'\neq 1$), given by $$P((\gamma_{1},...,\gamma_{n}),(\gamma_{1},...,\delta_{i},\delta_{i}',...,\gamma_{n}))=\frac{\sum_{j=1}^{N}|\langle \delta_{i}, c_{j}(\gamma_{i})\rangle|^{2}}{R((\gamma_{1},...,\gamma_{n}))}.$$ 
\end{notation}

We can now state the following trivial :

\begin{corollaries}
  Let $X_{t}$ be the ultramild solution given by theorem \ref{main}  with $\delta=(\delta_{1},...,\delta_{N})$ associated as above to cocycles $(c_{1},...,c_{N})$. Then $||X_{t}||_{2}=||X_{0}||_{2}$ (for any $X_{0}\in \ell^{2}(\Gamma)$) for all $t\in[0,T)$ (and as a consequence is stationary in $[0,T)$  on $M_{0}=L(\Gamma)$) if and only if $M(\Gamma;c_{1},...,c_{N})$ has almost surely no explosion before T. 
\end{corollaries}



\end{document}